\documentclass[11pt]{amsart}
\usepackage[english]{babel}
\usepackage{amsthm} 
\usepackage{amssymb}
\usepackage{amsmath}
\usepackage{hyperref}
\usepackage{amsfonts}
\usepackage{verbatim}
\usepackage{stmaryrd}
\usepackage{fancyhdr}
\usepackage{mathrsfs}
\usepackage{graphicx}
\usepackage{color}
\usepackage{pdfsync}
\newtheorem{theorem}{Theorem}[section]

\newtheorem{corollary}[theorem]{Corollary}

\newtheorem{lemma}[theorem]{Lemma}
\newtheorem{proposition}[theorem]{Proposition}

\numberwithin{equation}{section}

\newcommand{\N}{\mathbb{N}}
\newcommand{\R}{\mathbb{R}}

\newcommand{\lv}{\left[}
\newcommand{\rv }{\right]}
\newcommand{\rr}{\mathbb{R}}

\newcommand{\CC}{\mathbb{C}}
\newcommand{\cc}{\mathbb{C}}
\newcommand{\eps}{\varepsilon}

\newcommand{\lde}{L^2(\rr^n)}

\newcommand {\be}{\begin{equation}}
\newcommand {\ee}{\end{equation}}
\newcommand {\ba}{\begin{array}}
\newcommand {\ea}{\end{array}}

\usepackage{varioref}
 
\def\Rom#1{\uppercase\expandafter{\romannumeral #1}}

\def\polhk#1{\setbox0=\hbox{#1}{\ooalign{\hidewidth \lower1.5ex\hbox{`}\hidewidth\crcr\unhbox0}}}

\begin{document}

\title[Short-time asymptotics of the regularizing effect for semigroups]{Short-time asymptotics of the regularizing effect for semigroups generated by quadratic operators}

\author{M. Hitrik, K. Pravda-Starov \& J. Viola}

\address{\noindent \textsc{Michael Hitrik, Department of Mathematics, UCLA, Los Angeles CA 90095-1555, USA}}
\email{hitrik@math.ucla.edu}
\address{\noindent \textsc{Karel Pravda-Starov, IRMAR, CNRS UMR 6625, Universit\'e de Rennes 1, Campus de Beaulieu, 263 avenue du G\'en\'eral Leclerc, CS 74205,
35042 Rennes cedex, France}}
\email{karel.pravda-starov@univ-rennes1.fr}
\address{\noindent \textsc{Joe Viola, Laboratoire de Math\'ematiques Jean Leray, CNRS UMR 6629, 2 rue de la Houssini\`ere, Universit\'e de Nantes, BP 92208, 44322 Nantes, cedex 3, France}}
\email{joseph.viola@univ-nantes.fr}


\begin{abstract}
We study accretive quadratic operators with zero singular spaces. These degenerate non-selfadjoint differential operators are known to be hypoelliptic and to generate contraction semigroups which are smoothing in the Schwartz space for any positive time. In this work, we study the short-time asymptotics of the regularizing effect induced by these semigroups. We show that these short-time asymptotics  of the regularizing effect depend on the directions of the phase space, and that this dependence can be nicely understood through the structure of the singular space. As a byproduct of these results, we derive sharp subelliptic estimates for accretive quadratic operators with zero singular spaces pointing out that the loss of derivatives with respect to the elliptic case also depends on the phase space directions according to the structure of the singular space. Some applications of these results are then given to the study of degenerate hypoelliptic Ornstein-Uhlenbeck operators and  degenerate hypoelliptic Fokker-Planck operators.    
\end{abstract}

\maketitle

\section{Introduction}

\subsection{Quadratic operators} We study in this work quadratic operators. This class of operators stands for pseudodifferential operators
\begin{equation}\label{3.1}
q^w(x,D_x)u(x) =\frac{1}{(2\pi)^n}\int_{\R^{2n}}{e^{i(x-y) \cdot \xi}q\Big(\frac{x+y}{2},\xi\Big)u(y)dyd\xi}, 
\end{equation}
defined by the Weyl quantization of complex-valued quadratic symbols 
\begin{align*}
q: \rr^{2n} &\rightarrow \cc, \quad  \quad  \quad  n \geq 1, \\
(x,\xi) &\mapsto q(x,\xi).
\end{align*}
These non-selfadjoint operators are in fact only differential operators since the Weyl quantization of the quadratic symbol
$x^{\alpha} \xi^{\beta}$, with $(\alpha,\beta) \in \N^{2n}$, $|\alpha+\beta|=2$, is simply given by
\begin{equation}\label{forsym}
(x^{\alpha} \xi^{\beta})^w=\textrm{Op}^w(x^{\alpha} \xi^{\beta})=\frac{x^{\alpha}D_x^{\beta}+D_x^{\beta} x^{\alpha}}{2}, 
\end{equation}
with $D_x=i^{-1}\partial_x$. The maximal closed realization of a quadratic operator $q^w(x,D_x)$ on $L^2(\rr^n)$, that is, the operator equipped with the domain
\begin{equation}\label{dom1}
D(q^w)=\big\{u \in L^2(\rr^n) : \ q^w(x,D_x)u \in L^2(\rr^n)\big\},
\end{equation}
where $q^w(x,D_x)u$ is defined in the distribution sense, is known to coincide with the graph closure of its restriction to the Schwartz space~\cite{mehler} (pp.~425-426),
$$q^w(x,D_x) : \mathscr{S}(\rr^n) \rightarrow \mathscr{S}(\rr^n).$$
Classically, to any quadratic form defined on the phase space 
$$q : \rr_x^n \times \rr_{\xi}^n \rightarrow \mathbb{C},$$
is associated a matrix $F \in M_{2n}(\CC)$ called its Hamilton map, or its fundamental matrix, which is defined as the unique matrix satisfying the identity
\begin{equation}\label{10}
\forall  (x,\xi) \in \R^{2n},\forall (y,\eta) \in \R^{2n}, \quad q\big((x,\xi),(y,\eta)\big)=\sigma\big((x,\xi),F(y,\eta)\big), 
\end{equation}
with $q(\cdot,\cdot)$ the polarized form associated to the quadratic form $q$, where $\sigma$ stands for the standard symplectic form
\begin{equation}\label{11}
\sigma\big((x,\xi),(y,\eta)\big)=\langle \xi, y \rangle -\langle x, \eta\rangle=\sum_{j=1}^n(\xi_j y_j-x_j \eta_j),
\end{equation}
with $x=(x_1,...,x_n)$, $y=(y_1,....,y_n)$, $\xi=(\xi_1,...,\xi_n)$, $\eta=(\eta_1,...,\eta_n) \in \cc^n$. We notice that a Hamilton map is skew-symmetric with respect to the symplectic form
\begin{multline}\label{a1}
\sigma\big((x,\xi),F(y,\eta)\big)=q\big((x,\xi),(y,\eta)\big)=q\big((y,\eta),(x,\xi)\big)\\
=\sigma\big((y,\eta),F(x,\xi)\big)=-\sigma\big(F(x,\xi),(y,\eta)\big), 
\end{multline}
by symmetry of the polarized form and skew-symmetry of the symplectic form. We observe from the definition that 
$$F=\frac{1}{2}\left(\begin{array}{cc}
\nabla_{\xi}\nabla_x q & \nabla_{\xi}^2q  \\
-\nabla_x^2q & -\nabla_{x}\nabla_{\xi} q 
\end{array} \right),$$
where the matrices $\nabla_x^2q=(a_{i,j})_{1 \leq i,j \leq n}$,  $\nabla_{\xi}^2q=(b_{i,j})_{1 \leq i,j \leq n}$, $\nabla_{\xi}\nabla_x q =(c_{i,j})_{1 \leq i,j \leq n}$,
$\nabla_{x}\nabla_{\xi} q=(d_{i,j})_{1 \leq i,j \leq n}$ are given by
$$a_{i,j}=\partial_{x_i,x_j}^2 q, \quad b_{i,j}=\partial_{\xi_i,\xi_j}^2q, \quad c_{i,j}=\partial_{\xi_i,x_j}^2q, \quad d_{i,j}=\partial_{x_i,\xi_j}^2q.$$

In~\cite{HPS}, the notion of singular space was introduced by the two first authors by pointing out the existence of a particular vector subspace in the phase space, which is intrinsically associated to a quadratic symbol $q$, and defined as the following finite intersection of kernels
\begin{equation}\label{h1bis}
S=\Big( \bigcap_{j=0}^{2n-1}\textrm{Ker}
\big[\textrm{Re }F(\textrm{Im }F)^j \big]\Big)\cap \rr^{2n},
\end{equation}
where $\textrm{Re }F$ and $\textrm{Im }F$ stand respectively for the real and imaginary parts of the Hamilton map $F$ associated to the quadratic symbol $q$,
$$\textrm{Re }F=\frac{1}{2}(F+\overline{F}), \quad \textrm{Im }F=\frac{1}{2i}(F-\overline{F}).$$
As pointed out in \cite{HPS,OPPS,karel,viola1}, the singular space plays a basic role in understanding the spectral and hypoelliptic properties of non-elliptic quadratic operators, as well as the spectral and pseudospectral properties of certain classes of degenerate doubly characteristic pseudodifferential operators~\cite{kps3,kps4,viola}. Some applications of these results to the study of degenerate hypoelliptic Ornstein-Uhlenbeck operators and degenerate hypoelliptic Fokker-Planck operators are also given in~\cite{OPPS2}.

When the quadratic symbol $q$ has a non-negative real part $\textrm{Re }q \geq 0$, the singular space can be defined in an equivalent way as the subspace in the phase space where all the Poisson brackets 
$$H_{\textrm{Im}q}^k \textrm{Re }q=\left(\frac{\partial \textrm{Im }q}{\partial\xi}\cdot \frac{\partial}{\partial x}-\frac{\partial \textrm{Im }q}{\partial x}\cdot \frac{\partial}{\partial \xi}\right)^k \textrm{Re } q, \quad k \geq 0,$$ 
are vanishing
$$S=\big\{X=(x,\xi) \in \rr^{2n} : \ (H_{\textrm{Im}q}^k \textrm{Re } q)(X)=0,\ k \geq 0\big\}.$$
This dynamical definition shows that the singular space corresponds exactly to the set of points $X \in \rr^{2n}$, where the real part of the symbol $\textrm{Re }q$ under the flow of the Hamilton vector field $H_{\textrm{Im}q}$ associated to its imaginary part
\begin{equation}\label{evg5}
t \mapsto \textrm{Re }q(e^{tH_{\textrm{Im}q}}X),
\end{equation}
vanishes to infinite order at $t=0$. This is also equivalent to the fact that the function \eqref{evg5} is identically zero on~$\rr$.

In this work, we first study the class of quadratic operators whose Weyl symbols have a non-negative real part $\textrm{Re }q \geq 0$ and a zero singular space $S=\{0\}$. These possibly non-elliptic non-selfadjoint operators are known to be hypoelliptic and to enjoy specific subelliptic properties which can be nicely characterized by the structure of the singular space~\cite{karel}. 
We recall from~\cite{HPS} (Theorem~1.2.2) that the spectrum of these operators acting on $L^2$ and equipped with the domain (\ref{dom1}), is discrete and only composed of eigenvalues with finite algebraic multiplicities given by
\begin{equation}\label{jkk1}
\sigma(q^w(x,D_x))=\Big\{\sum_{\substack{\lambda \in \sigma(F)\\  -i\lambda \in \CC_+}} (r_{\lambda}+2k_{\lambda})(-i\lambda) : k_{\lambda} \in \mathbb{N} \Big\},
\end{equation}
with $\CC_+=\{z \in \CC : \textrm{Re }z>0\}$, where $\mathbb{N}$ is the set of non-negative integers, $F$ denotes the Hamilton map (\ref{10}) of~$q$, and where
$r_{\lambda}$ stands for the dimension of the space of generalized eigenvectors of $F$ in $\CC^{2n}$ associated to the eigenvalue $\lambda$. Afterwards, we study a more general class of quadratic operators whose Weyl symbols have a non-negative real part $\textrm{Re }q \geq 0$ and whose singular spaces $S$ can be non-zero but still have a symplectic structure.

\subsection{Setting of the analysis}

Let $q : \rr^{n}_x \times \rr_{\xi}^n \rightarrow \cc$, with $n \geq 1$, be a quadratic form with a non-negative real part $\textrm{Re }q \geq 0$. We know from~\cite{mehler} (pp. 425-426) that the quadratic operator $q^w(x,D_x)$ equipped with the domain (\ref{dom1}), is maximal accretive and generates a contraction semigroup $(e^{-tq^w})_{t \geq 0}$ on $L^2(\rr^n)$. Furthermore, the result of~\cite{mehler} (Theorem~4.2) shows that for all $u \in \mathscr{S}(\rr^n)$, the mapping 
\begin{equation}\label{conti}
\begin{array}{cc}
[0,+\infty) & \rightarrow  \mathscr{S}(\rr^n)\\
t &  \mapsto  e^{-tq^w}u
\end{array}
\end{equation} 
is continuous. 
We first assume that the singular space of the quadratic form $q$ is zero
\begin{equation}\label{e1}
S=\Big(\bigcap_{j=0}^{2n-1}\textrm{Ker}
\big[\textrm{Re }F(\textrm{Im }F)^j \big]\Big)\cap \rr^{2n}=\{0\}.
\end{equation}
Under these assumptions, the quadratic operator generates a contraction semigroup $(e^{-tq^w})_{t \geq 0}$ on $L^2(\rr^n)$,
$$\left\lbrace
\begin{array}{c}
\partial_tu(t,x)+q^w(x,D_x) u(t,x)=0,  \\
u(t,\textrm{\textperiodcentered})|_{t=0}=u_0 \in L^2(\rr^n),
\end{array} \right.$$
which is smoothing in the Schwartz space~\cite{HPS} (Theorem~1.2.1), 
$$\forall u_0 \in L^2(\rr^n), \forall t>0, \quad e^{-tq^w}u_0 \in \mathscr{S}(\rr^n),$$
for any positive time.

In the present work, we aim at studying how the regularizing effect induced by the semigroup $(e^{-tq^w})_{t \geq 0}$ acts for small times $0<t \ll 1$. To that end, we analyze the short-time asymptotics of the following differentiations of the semigroup
\begin{equation}\label{we2}
\|(\langle x_0, x \rangle+\langle \xi_0,D_x \rangle) e^{-tq^w}u_0\|_{L^2(\rr^n)},
\end{equation}
with $x_0=\big((x_0)_1,....,(x_0)_n\big) \in \rr^n$, $\xi_0=\big((\xi_0)_1,....,(\xi_0)_n\big) \in \rr^n$,
$$\langle x_0, x \rangle+\langle \xi_0,D_x \rangle=\sum_{j=1}^n\big((x_0)_jx_j+(\xi_0)_jD_{x_j}\big),$$
and we aim at understanding how the behavior of the terms (\ref{we2}) depend on the phase space direction $X_0=(x_0,\xi_0) \in \rr^{2n}$. As indicated by previous works, this dependence is expected to be non-trivial. This is pointed out for instance by the work of H\'erau~\cite{herau_JFA} (Theorem~1.1) on the short-time behavior of the Fokker-Planck equation 
$$\left\lbrace
\begin{array}{c}
\partial_tf(t,x,v)+Kf(t,x,v)=0,  \\
f(t,\textrm{\textperiodcentered})|_{t=0}=f_0 \in B^2,
\end{array} \right.$$
where
$$B^2=\big\{f \in \mathcal{D}'(\rr_{x,v}^{2n}) : f e^{\frac{v^2}{2}+V(x)} \in L^2(e^{-(\frac{v^2}{2}+V(x))}dxdv)\big\}$$
and
$$Kf=v \cdot \partial_xf-\partial_xV(x) \cdot \partial_vf-\partial_v \cdot (\partial_v+v)f,$$
with a confining (not necessarily quadratic) potential. This work shows that the short-time asymptotics of the regularizing effect induced by the semigroup
\begin{equation}\label{hjfa}
\|(-\partial_v+v)e^{-tK}\|_{\mathcal{L}(B^2)} \leq \frac{C}{t^{\frac{1}{2}}}, \qquad  \|(-\partial_x+\partial_xV)e^{-tK}\|_{\mathcal{L}(B^2)} \leq \frac{C}{t^{\frac{3}{2}}},
\end{equation}
when $0<t \leq 1$, clearly depend on the directions of the phase space in which the semigroup is differentiated. The results of the present work for accretive quadratic operators with zero singular spaces allow in particular to recover the short-time asymptotics (\ref{hjfa}) in the case when the confining potential $V$ is quadratic. Afterwards, we discuss the case when singular spaces are possibly non-zero but still have a symplectic structure.

\subsection{Statements of the main results}

\subsubsection{Results for quadratic operators with zero singular spaces}

According to the assumption (\ref{e1}), we may consider $0 \leq k_0 \leq 2n-1$ the smallest integer satisfying
\begin{equation}\label{e2}
\Big(\bigcap_{j=0}^{k_0}\textrm{Ker}
\big[\textrm{Re }F(\textrm{Im }F)^j \big]\Big)\cap \rr^{2n}=\{0\}.
\end{equation}
This non-negative integer is a structural parameter of the singular space. It plays a basic role in understanding the subelliptic properties of the quadratic operator $q^w(x,D_x)$, and in measuring the loss of derivatives with respect to the elliptic case in the subelliptic estimates satisfied by this operator~\cite{karel} (Theorem~1.2.1).
The structure of the singular space also allows one to define the following vector subspaces in the phase space
\begin{equation}\label{we1}
V_k=\Big(\Big[\bigcap_{j=0}^{k}\textrm{Ker}\big[\textrm{Re }F(\textrm{Im }F)^j \big]\Big]\cap \rr^{2n}\Big)^{\perp}, \qquad 0 \leq k \leq k_0,
\end{equation}
where the orthogonality is taken in $\rr^{2n}$ with respect to the Euclidean scalar product
$$\langle (x,\xi),(y,\eta) \rangle=\sum_{j=1}^n(x_jy_j+\xi_j\eta_j),$$
with $x=(x_1,....,x_n)$, $\xi=(\xi_1,...,\xi_n)$, $y=(y_1,....,y_n)$, $\eta=(\eta_1,...,\eta_n) \in \rr^{n}$. According to (\ref{e2}), this family of vector subspaces is increasing for the inclusion
\begin{equation}\label{byebye1}
V_0 \subsetneq V_1 \subsetneq ... \subsetneq V_{k_0}=\rr^{2n},
\end{equation}
and allows one to define the index of any point in the phase space $X_0=(x_0,\xi_0) \in \rr^{2n}$ with respect to the singular space, as the smallest integer $0 \leq k_{X_0} \leq k_0$ satisfying $X_0 \in V_{k_{X_0}}$.

The following result shows that this notion of index allows one to determine the short-time asymptotics of the regularizing effect induced by the semigroup $(e^{-tq^w})_{t \geq 0}$ in the phase space direction given by the vector $X_0=(x_0,\xi_0) \in \rr^{2n}$:

\medskip

\begin{theorem}\label{th}
Let $q : \rr^{n}_x \times \rr_{\xi}^n \rightarrow \cc$, with $n \geq 1$, be a quadratic form with a non-negative real part $\emph{\textrm{Re }}q \geq 0$ and a zero singular space
$$S=\Big(\bigcap_{j=0}^{2n-1}\emph{\textrm{Ker}}
\big[\emph{\textrm{Re }}F(\emph{\textrm{Im }}F)^j \big]\Big)\cap \rr^{2n}=\{0\}.$$
Let $0 \leq k_0 \leq 2n-1$ be the smallest integer satisfying (\ref{e2}).
Then, there exists a positive constant $C>0$ such that for all $X_0=(x_0,\xi_0) \in \rr^{2n}$, $u \in L^2(\rr^n)$,  
$$\forall 0<t \leq 1, \quad \|(\langle x_0, x \rangle+\langle \xi_0,D_x \rangle) e^{-tq^w}u\|_{L^2(\rr^n)}\leq C|X_0| t^{-\frac{2k_{X_0}+1}{2}}\|u\|_{L^2(\rr^n)},$$
$$\forall t \geq 1, \quad \|(\langle x_0, x \rangle+\langle \xi_0,D_x \rangle) e^{-tq^w}u\|_{L^2(\rr^n)}\leq C|X_0|e^{-\omega_0t}\|u\|_{L^2(\rr^n)},$$
with $0 \leq k_{X_0} \leq k_0$ the index of the point $X_0 \in \rr^{2n}$ with respect to the singular space,
\begin{equation}\label{qz-1}
\omega_0=\sum_{\substack{\lambda \in \sigma(F) \\
-i \lambda \in \cc_+}}r_{\lambda}\emph{\textrm{Re}}(-i\lambda)>0, \qquad \cc_+=\{z \in \cc : \emph{\textrm{Re }}z>0\},
\end{equation}
where $r_{\lambda}$ stands for the dimension of the space of generalized eigenvectors of $F$ in $\cc^{2n}$ associated to the eigenvalue $\lambda \in \sigma(F)$, and where $|\cdot|$ denotes the Euclidean norm on $\rr^{2n}$.
\end{theorem}

\medskip

This result shows that the structure of the singular space accounting for the family of vector subspaces $(V_k)_{0 \leq k \leq k_0}$, 
allows one to sharply describe the short-time asymptotics of the regularizing effect induced by the semigroup $(e^{-tq^w})_{t \geq 0}$. The degeneracy degree of the phase space direction $X_0=(x_0,\xi_0) \in \rr^{2n}$ given by the index with respect to the singular space directly accounts for the blow-up upper bound $t^{-(2k_{X_0}+1)/2}$, for small times $t \to 0^+$. Regarding the large time asymptotics, we notice from (\ref{jkk1}) that the rate of exponential decay $-\omega_0$ naturally corresponds to the spectral abscissa of the operator $-q^w(x,D_x)$.

As a corollary, we prove the following upper bounds for iterated differentiations of the semigroup:

\medskip

\begin{corollary}\label{thj}
Let $q : \rr^{n}_x \times \rr_{\xi}^n \rightarrow \cc$, with $n \geq 1$, be a quadratic form with a non-negative real part $\emph{\textrm{Re }}q \geq 0$ and a zero singular space
$$S=\Big(\bigcap_{j=0}^{2n-1}\emph{\textrm{Ker}}
\big[\emph{\textrm{Re }}F(\emph{\textrm{Im }}F)^j \big]\Big)\cap \rr^{2n}=\{0\}.$$ 
Let $0 \leq k_0 \leq 2n-1$ be the smallest integer satisfying (\ref{e2}).
Then, there exists a positive constant $C>1$ such that for all $m \geq 1$, $X_1=(x_1,\xi_1)  \in \rr^{2n}$, ..., $X_m=(x_m,\xi_m)  \in \rr^{2n}$, $u \in L^2(\rr^n)$,
\begin{multline*}
\forall 0<t \leq 1, \quad \big\|(\langle x_1, x \rangle+\langle \xi_1,D_x \rangle)\ ...\ (\langle x_m, x \rangle+\langle \xi_m,D_x \rangle) e^{-tq^w}u\big\|_{L^2(\rr^n)}\\
\leq C^{m}(m!)^{\frac{2k_0+1}{2}}\Big(\prod_{j=1}^m|X_j|\Big)t^{-\frac{(2k_0+1)m}{2}}\|u\|_{L^2(\rr^n)},
\end{multline*}
\begin{multline*}
\forall t \geq 1, \quad \big\|(\langle x_1, x \rangle+\langle \xi_1,D_x \rangle)\ ...\ (\langle x_m, x \rangle+\langle \xi_m,D_x \rangle) e^{-tq^w}u\big\|_{L^2(\rr^n)}\\
\leq C^m(m!)^{\frac{2k_0+1}{2}}\Big(\prod_{j=1}^m|X_j|\Big)e^{-\omega_0t}\|u\|_{L^2(\rr^n)},
\end{multline*}
with 
$$\omega_0=\sum_{\substack{\lambda \in \sigma(F) \\
-i \lambda \in \cc_+}}r_{\lambda}\emph{\textrm{Re}}(-i\lambda)>0, \qquad \cc_+=\{z \in \cc : \emph{\textrm{Re }}z>0\},$$
where $r_{\lambda}$ stands for the dimension of the space of generalized eigenvectors of $F$ in $\cc^{2n}$ associated to the eigenvalue $\lambda \in \sigma(F)$, and where $|\cdot|$ denotes the Euclidean norm on $\rr^{2n}$.
\end{corollary}

\medskip

The result of Corollary~\ref{thj} provides an upper bound in $\mathcal{O}(t^{-\frac{(2k_0+1)m}{2}})$ for the short-time asymptotics of $m$ differentiations of the semigroup $(e^{-tq^w})_{t \geq 1}$. Unfortunately, this result does not disclose that these short-time asymptotics do depend on the phase space directions of the differentiations. It would be actually most interesting to understand if the upper bound in Corollary~\ref{thj} may be sharpened to 
$$\mathcal{O}(t^{-\frac{2(k_{X_1}+...+k_{X_m})+m}{2}}),$$
where $k_{X_j}$ denotes the index of the direction $X_j=(x_j,\xi_j) \in \rr^{2n}$ with respect to the singular space.

By using the Sobolev inequality and the estimate $(k+l)! \leq 2^{k+l}(k!)(l!)$, which holds for any $k,l \geq 0$, we notice  that the result of Corollary~\ref{thj} implies that there exists a positive constant $C>1$ such that for all $\alpha, \beta \in \mathbb{N}^n$, $u \in L^2(\rr^n)$,
$$\forall 0<t \leq 1, \quad \sup_{x \in \rr^n}|x^{\alpha} \partial_x^{\beta}e^{-tq^w}u|   
\leq C^{1+|\alpha|+|\beta|}(\alpha!)^{\frac{2k_0+1}{2}}(\beta!)^{\frac{2k_0+1}{2}}t^{-\frac{2k_0+1}{2}(|\alpha|+|\beta|+s)}\|u\|_{L^2},$$
$$\forall t \geq 1, \quad \sup_{x \in \rr^n}|x^{\alpha} \partial_x^{\beta}e^{-tq^w}u|   
\leq C^{1+|\alpha|+|\beta|}(\alpha!)^{\frac{2k_0+1}{2}}(\beta!)^{\frac{2k_0+1}{2}}e^{-\omega_0t}\|u\|_{L^2},$$
where $s$ is an integer satisfying $s>n/2$. These estimates show that the semigroup $(e^{-tq^w})_{t \geq 0}$ is smoothing in the Gelfand-Shilov space $S_{\frac{2k_0+1}{2}}^{\frac{2k_0+1}{2}}(\rr^n)$ for any positive time $t>0$, and provide a precise control of the Gelfand-Shilov seminorms. We refer the reader to the appendix in Section~\ref{appendix} for the definition of the Gelfand-Shilov regularity. In the work~\cite{HPSVII} (Theorem~1.2) prepared simultaneously with the present one, this regularizing result is sharpened  and the semigroup $(e^{-tq^w})_{t \geq 0}$ is showed to be actually smoothing in the Gelfand-Shilov space $S_{1/2}^{1/2}(\rr^n)$ for any positive time $t>0$, with a control in $\mathcal{O}(t^{-\frac{2k_0+1}{2}(|\alpha|+|\beta|+2n+s)})$ of the seminorms as $t \to 0^+$, where $s$ is an integer satisfying $s>n/2$.

As mentioned above, quadratic operators whose symbols have a non-negative real part and a zero singular space satisfy specific subelliptic estimates where the loss of derivatives with respect to the elliptic case directly depends on the structural parameter of the singular space $0 \leq k_0 \leq 2n-1$. These subelliptic estimates were proven to hold in~\cite{karel} (Theorem~1.2.1), by using a multiplier method. The construction of the multiplier is the core of the work~\cite{karel} (Proposition~2.0.1). This construction is very technical and quite intricate. We show in the present work that these estimates can be recovered as a byproduct of the result of Theorem~\ref{th}. In the following statement, we denote by 
$$\langle (x,D_x) \rangle^{s}=(1+D_x^2+x^2)^{\frac{s}{2}}, \quad s \geq 0, \quad D_x=i^{-1}\partial_x,$$ 
the operator defined by functional calculus of the harmonic oscillator.

\medskip

\begin{corollary}\label{th0}
Let $q : \rr^{n}_x \times \rr_{\xi}^n \rightarrow \cc$, with $n \geq 1$, be a quadratic form with a non-negative real part $\emph{\textrm{Re }}q \geq 0$ and a zero singular space 
$$S=\Big(\bigcap_{j=0}^{2n-1}\emph{\textrm{Ker}}
\big[\emph{\textrm{Re }}F(\emph{\textrm{Im }}F)^j \big]\Big)\cap \rr^{2n}=\{0\}.$$ 
Let $0 \leq k_0 \leq 2n-1$ be the smallest integer satisfying (\ref{e2}).
Then, there exists a positive constant $C>0$ such that 
$$\forall u \in D(q^w), \quad \|\langle (x,D_x) \rangle^{\frac{2}{2k_0+1}}u\|_{L^2(\rr^n)} \leq C(\|q^w(x,D_x)u\|_{L^2(\rr^n)}+\|u\|_{L^2(\rr^n)}).$$
\end{corollary}

\medskip

In the work~\cite{HPSVII}, we develop an alternative approach to the proof of Corollary~\ref{th0}, based on the study of the semigroup 
generated by the operator $q^w(x,D_x)$, viewed as a quadratic Fourier integral operator with complex phase, on the FBI transform side. 

More generally, we show that these subelliptic estimates can be sharpened and improved in the directions of the phase space which are less degenerate, that is, with smaller indices with respect to the singular space. To that end, we introduce the following non-negative quadratic forms
\begin{equation}\label{rg0}
r_k(X)=\sum_{j=0}^k\textrm{Re }q\big((\textrm{Im }F)^jX\big) \geq 0, \quad X=(x,\xi) \in \rr^{2n},
\end{equation}
for any $0 \leq k \leq k_0$, where $0 \leq k_0 \leq 2n-1$ is the smallest integer satisfying (\ref{e2}).
Equipped with the domain 
\begin{equation}\label{rg2}
D(\Lambda_k^2)=D(r_k^w)=\big\{u \in L^2(\rr^n) : r_k^w(x,D_x)u \in L^2(\rr^n)\big\},
\end{equation}
the selfadjoint operator 
\begin{equation}\label{rg1}
\Lambda_k^2=1+r_k^w(x,D_x),
\end{equation} 
is positive (see Lemma~\ref{lem4}), 
$$\forall u \in D(\Lambda_k^2), \quad (\Lambda_k^2 u,u)_{L^2(\rr^n)}=(r_k^w(x,D_x)u,u)_{L^2(\rr^n)}+\|u\|_{L^2(\rr^n)}^2 \geq \|u\|_{L^2(\rr^n)}^2.$$
We shall consider the fractional powers of these positive operators 
\begin{equation}\label{tio1}
\Lambda_k^{\frac{2}{2k+1}}=(\Lambda_k^2)^{\frac{1}{2k+1}}=\big(1+r_k^w(x,D_x)\big)^{\frac{1}{2k+1}}.
\end{equation} 
We refer the reader for instance to~\cite{interpolation} (Chapter~4) for the definition of these fractional powers of positive operators. By using these operators, the result of Corollary~\ref{th0} can be improved as follows:

\medskip

\begin{theorem}\label{th-1}
Let $q : \rr^{n}_x \times \rr_{\xi}^n \rightarrow \cc$, with $n \geq 1$, be a quadratic form with a non-negative real part $\emph{\textrm{Re }}q \geq 0$ and a zero singular space 
$$S=\Big(\bigcap_{j=0}^{2n-1}\emph{\textrm{Ker}}
\big[\emph{\textrm{Re }}F(\emph{\textrm{Im }}F)^j \big]\Big)\cap \rr^{2n}=\{0\}.$$ 
Let $0 \leq k_0 \leq 2n-1$ be the smallest integer satisfying (\ref{e2}).
Then, there exists a positive constant $C>0$ such that 
$$\forall u \in D(q^w), \quad \|\Lambda_0u\|_{L^2(\rr^n)}+\sum_{k=1}^{k_0}\|\Lambda_k^{\frac{2}{2k+1}}u\|_{L^2(\rr^n)} \leq C(\|q^w(x,D_x)u\|_{L^2(\rr^n)}+\|u\|_{L^2(\rr^n)}),$$
where the operators $\Lambda_k^{\frac{2}{2k+1}}$ are defined in (\ref{tio1}).
\end{theorem}

\medskip

The result of Theorem~\ref{th-1} shows that the subelliptic estimates can be refined in the phase space directions which are less degenerate, that is, in the directions where the indices with respect to the singular space are stricly lower than $k_0$. The authors expect the powers $2/(2k+1)$ for the operators $\Lambda_k$, when $1 \leq k \leq k_0$, to be sharp. However, the power~$1$ for the operator $\Lambda_0$ is probably not optimal. The optimal power for the operator $\Lambda_0$ is expected to be equal to~$2$.

\subsubsection{Results for quadratic operators with symplectic singular spaces} 

We consider a quadratic operator $q^w(x,D_x)$ whose Weyl symbol has a non-negative real part $\textrm{Re }q \geq 0$ and whose singular space 
$$S=\Big(\bigcap_{j=0}^{2n-1}\textrm{Ker}
\big[\textrm{Re }F(\textrm{Im }F)^j \big]\Big)\cap \rr^{2n},$$
has a symplectic structure, that is, when the restriction of the symplectic form $\sigma|_{S}$ to the singular space is non-degenerate. This symplectic structure property holds in particular when the quadratic symbol $q$ is elliptic along its singular space~\cite{HPS} (Section~1.4.1),
$$(x,\xi) \in S, \quad q(x,\xi)=0 \Longrightarrow (x,\xi)=0.$$ 
When the singular space $S$ has a symplectic structure, it allows to decompose the phase space into the direct sum of the two symplectically orthogonal spaces
$$\rr^{2n}=S\oplus^{\sigma \perp} S^{\sigma \perp},$$
where $S^{\sigma \perp}$ stands for the orthogonal complement of the singular space in $\rr^{2n}$ with respect to the symplectic form 
$$S^{\sigma \perp}=\{X \in \rr^{2n} : \ \forall Y \in S, \ \sigma(X,Y)=0\}.$$
In this case, the vector space $S^{\sigma \perp}$ also enjoys a symplectic structure, and both vector subspaces $S$ and $S^{\sigma \perp}$ are stable by the real and imaginary parts $\textrm{Re }F$ and $\textrm{Im }F$ of the Hamilton map $F$.
Under these assumptions, it was shown in~\cite{HPS} (Theorem~1.2.1) that 
 for all $t>0$, $N \in \mathbb{N}$, $u \in L^2(\rr^n)$,
\begin{equation}\label{sm4bis}
(1+|x'|^2+|D_{x'}|^2)^Ne^{-t q^w}u \in L^2(\rr^n), 
\end{equation}
if $(x',\xi')$ are some linear symplectic coordinates on the symplectic space $S^{\sigma \perp}$.
It was also proved in~\cite{karel} (Theorem~1.2.2) that the quadratic operator $q^w(x,D_x)$ is subelliptic in any direction of the space $S^{\sigma \perp}$ in the sense that: $\exists C>0$, $\forall u \in D(q^w)$,
\begin{equation}\label{risten5}
\|\langle (x',D_{x'})\rangle^{2/(2k_0+1)} u\|_{L^2(\rr^n)} \leq C(\|q^w(x,D_x)u\|_{L^2(\rr^n)}+\|u\|_{L^2(\rr^n)}),
\end{equation}
where $0 \leq k_0 \leq 2n-1$ is the smallest integer satisfying
\begin{equation}\label{h1biskrist2bis}
S=\Big(\bigcap_{j=0}^{k_0}\textrm{Ker}\big[\textrm{Re }F(\textrm{Im }F)^j \big]\Big) \cap \rr^{2n}.
\end{equation}
When the singular space is possibly non-zero but still has a symplectic structure, the result of Corollary~\ref{thj} can be extended as follows
when differentiating the semigroup in the directions of the phase space given by 
the orthogonal complement of the singular space~$S^{\sigma \perp}$:

\medskip

\begin{corollary}\label{thj_symp}
Let $q : \rr^{n}_x \times \rr_{\xi}^n \rightarrow \cc$, with $n \geq 1$, be a quadratic form with a non-negative real part $\emph{\textrm{Re }}q \geq 0$ and a singular space with a symplectic structure. Let $0 \leq k_0 \leq 2n-1$ be the smallest integer satisfying
$$S=\Big(\bigcap_{j=0}^{k_0}\emph{\textrm{Ker}}
\big[\emph{\textrm{Re }}F(\emph{\textrm{Im }}F)^j \big]\Big)\cap \rr^{2n},$$ 
$(\tilde{e}_1,...,\tilde{e}_{n'},\tilde{\eps}_1,...,\tilde{\eps}_{n'})$ a symplectic basis of $S^{\sigma \perp}$, $(\tilde{e}_{n'+1},...,\tilde{e}_{n'+n''},\tilde{\eps}_{n'+1},...,\tilde{\eps}_{n'+n''})$ a symplectic basis of $S$ with $n=n'+n''$, and $\chi : \rr^{2n} \rightarrow \rr^{2n}$ the linear symplectic transformation satisfying 
$\chi(\tilde{e}_j)=e_j$, $\chi(\tilde{\eps}_j)=\eps_j$, $1 \leq j \leq n$, with $(e_1,...,e_n,\eps_1,...,\eps_n)$ the canonical symplectic basis of the phase space $\rr_x^n \times \rr_{\xi}^n$.
Then, there exists a positive constant $C>1$ such that for all $m \geq 1$, $X_1 \in S^{\sigma \perp}$, ..., $X_m \in S^{\sigma \perp}$, $u \in L^2(\rr^n)$,
\begin{multline*}
\forall 0<t \leq 1, \quad \big\|\langle \chi(X_1),\chi^w(x,D_x) \rangle\ ...\ \langle \chi(X_m),\chi^w(x,D_x) \rangle e^{-tq^w}u\big\|_{L^2(\rr^n)}\\
\leq C^{m}(m!)^{\frac{2k_0+1}{2}}\Big(\prod_{j=1}^m|X_j|\Big)t^{-\frac{(2k_0+1)m}{2}}\|u\|_{L^2(\rr^n)},
\end{multline*}
\begin{multline*}
\forall t \geq 1, \quad \big\|\langle \chi(X_1),\chi^w(x,D_x) \rangle\ ...\ \langle \chi(X_m),\chi^w(x,D_x) \rangle e^{-tq^w}u\big\|_{L^2(\rr^n)}\\
\leq C^m(m!)^{\frac{2k_0+1}{2}}\Big(\prod_{j=1}^m|X_j|\Big)e^{-\omega_0t}\|u\|_{L^2(\rr^n)},
\end{multline*}
with 
$$\omega_0=\sum_{\substack{\lambda \in \sigma(F|_{S^{\sigma \perp}}) \\
-i \lambda \in \cc_+}}r_{\lambda}\emph{\textrm{Re}}(-i\lambda)>0, \qquad \cc_+=\{z \in \cc : \emph{\textrm{Re }}z>0\},$$
where $F|_{S^{\sigma \perp}}$ denotes the restriction of the Hamilton map $F$ of $q$ to the vector subspace $S^{\sigma \perp}$, $r_{\lambda}$ stands for the dimension of the space of generalized eigenvectors of $F|_{S^{\sigma \perp}}$ associated to the eigenvalue $\lambda \in \sigma(F|_{S^{\sigma \perp}})$, and where $|\cdot|$ denotes the Euclidean norm on~$\rr^{2n}$.
\end{corollary}

\medskip
 
In the case when the singular space is zero $S=\{0\}$, the result of Corollary~\ref{thj_symp} exactly reduces to the result of Corollary~\ref{thj}. On the other hand, we observe that in the case when the singular space is equal to the whole phase space $S=\rr^{2n}$, the result of Corollary~\ref{thj_symp} is empty as the sole direction in $S^{\sigma \perp}$ is zero.
We close this paragraph by mentioning that the result of Theorem~\ref{th} can as well be extended to the case when the singular space is possibly non-zero but still has a symplectic structure. This follows from the very same arguments as the ones given in the proof of Corollary~\ref{thj_symp} thanks to the tensorization of the variables (\ref{evg23}) and the results already known in the case when the singular space is zero. Details are left to the interested reader.

 \subsubsection{Outline of the article}
 
The article is organized as follows. The proofs of the main results are given in Section~\ref{qz1}, whereas Section~\ref{applications} is devoted to provide some applications of these results to the study of degenerate hypoelliptic Ornstein-Uhlenbeck operators and  degenerate hypoelliptic Fokker-Planck operators. We show in particular in Section~\ref{applications}  how our results allow to recover some properties of degenerate hypoelliptic Ornstein-Uhlenbeck operators proven by Farkas and Lunardi in the work~\cite{lunardi1}, and how they relate to the recent results of Arnold and Erb~\cite{anton} on degenerate hypoelliptic Fokker-Planck operators with linear drift. The last section (Section~\ref{appendix}) is devoted to an appendix on the Gelfand-Shilov regularity.

\section{Proofs of the main results}\label{qz1}

\subsection{Proof of Theorem~\ref{th}}
This section is devoted to the proof of Theorem~\ref{th}. 
The idea of the proof is to construct a decaying-in-time functional following some techniques developed first by H\'erau~\cite{herau_JFA} in his work on the short and long time behavior of the Fokker-Planck equation in a confining potential, and more recently by Arnold and Erb~\cite{anton} for studying the entropy decay for hypocoercive and non-symmetric Fokker-Planck equations with linear drift. The proof of Theorem~\ref{th} draws its inspiration from the construction of the decaying-in-time functional made in~\cite{anton} (Theorem~4.8).

We begin by observing that condition (\ref{e2}) is equivalent to the positive definiteness of the following associated quadratic form 
\begin{equation}\label{e3}
\exists c_0>0, \forall X=(x,\xi) \in \rr^{2n}, \quad \sum_{j=0}^{k_0}\textrm{Re }q\big((\textrm{Im }F)^jX\big) \geq c_0 |X|^2.
\end{equation}
Indeed, if 
$$\sum_{j=0}^{k_0}\textrm{Re }q\big((\textrm{Im }F)^jX_0\big)=0, \quad X_0 \in \rr^{2n},$$
the non-negativity of the quadratic form $\textrm{Re } q \geq 0$ implies that 
\begin{equation}\label{2.3.99}
\forall j=0,...,k_0, \quad \textrm{Re }q\big((\textrm{Im } F)^j X_0\big)=0. 
\end{equation}
Letting $\textrm{Re }q(\cdot,\cdot)$ denote the polarized form associated to the quadratic form $\textrm{Re }q \geq 0$, we deduce from the Cauchy-Schwarz inequality and (\ref{2.3.99}) that for all $j=0,...,k_0$, $Y \in \rr^{2n}$, 
$$\big|\textrm{Re }q\big(Y,(\textrm{Im }F)^j X_0\big)\big|^2= \big|\sigma\big(Y,\textrm{Re }F (\textrm{Im }F)^j X_0\big)\big|^2
 \leq \textrm{Re }q(Y)\textrm{Re }q\big((\textrm{Im } F)^j X_0\big)=0.$$
It follows that for all $j=0,...,k_0$, $Y \in \rr^{2n}$, $\sigma\big(Y,\textrm{Re }F (\textrm{Im }F)^j X_0\big)=0$.
We deduce that for all $j=0,...,k_0$, $\textrm{Re }F(\textrm{Im }F)^j X_0=0$, because the symplectic form is non-degenerate. 
Conversely, we observe that the non-negative quadratic form 
$$\sum_{j=0}^{k_0}\textrm{Re }q\big((\textrm{Im }F)^jX_0\big)=\sum_{j=0}^{k_0}\sigma\big((\textrm{Im }F)^jX_0,\textrm{Re }F(\textrm{Im }F)^jX_0\big)=0,$$
is zero, when $\textrm{Re }F(\textrm{Im }F)^j X_0=0$, for all $j=0,...,k_0$. 
Condition (\ref{e2}) is therefore equivalent to the fact that the non-negative quadratic form (\ref{e3}) is positive definite.
We can thus find a positive constant $c_1>0$ such that 
\begin{equation}\label{e4}
\forall X \in \rr^{2n}, \quad 0 \leq \textrm{Re }q\big((\textrm{Im }F)^{k_0+2}X\big) \leq c_1\sum_{j=0}^{k_0}\textrm{Re }q\big((\textrm{Im }F)^jX\big).
\end{equation}
We consider the quadratic forms
\begin{equation}\label{e5}
\forall 0 \leq j \leq k_0+1, \quad M_j(X)=\textrm{Re }q\big((\textrm{Im }F)^{j}X\big) \geq 0,
\end{equation}
\begin{equation}\label{e6}
\forall 0 \leq j \leq k_0, \quad N_j(X)=2\textrm{Re }q\big((\textrm{Im }F)^{j}X,(\textrm{Im }F)^{j+1}X\big).
\end{equation}
We define
\begin{equation}\label{e7}
a_{k_0+1}=\frac{1}{c_1}>0, \quad b_{k_0}=\frac{2}{3}\big(1+(2k_0+4)a_{k_0+1}\big)>0, \quad a_{k_0}=\frac{2b_{k_0}^2}{a_{k_0+1}}>0.
\end{equation}
Then, we can define iteratively starting from the index $j=k_0$ and finishing with $j=1$ the following parameters
\begin{equation}\label{e8}
b_{j-1}=\frac{2}{3}\Big(2+c_1+(2j+1)a_{j}+b_j^2+\frac{2\big(a_j-(2j+2)b_j\big)^2}{b_j}\Big)>0, \quad a_{j-1}=\frac{8b_{j-1}^2}{a_{j}}>0,
\end{equation}
for all $1 \leq j \leq k_0$. We consider the time-dependent quadratic form
\begin{equation}\label{e9}
Q_t(X)=\sum_{j=0}^{k_0+1}a_jt^{2j+1}M_j(X)-\sum_{j=0}^{k_0}b_jt^{2j+2}N_j(X), \quad t \geq 0,
\end{equation}
where the quadratic forms $M_j$ and $N_j$ are defined in (\ref{e5}) and (\ref{e6}). The following lemma shows that this quadratic form is positive definite for all $t>0$:

\medskip

\begin{lemma}\label{lem1}
We have for all $t \geq 0$, $X \in \rr^{2n}$, 
$$Q_t(X) \geq \frac{1}{4}\sum_{j=0}^{k_0}a_jt^{2j+1}M_j(X) \geq \frac{c_0}{4}(\inf_{0 \leq j \leq k_0}a_j)\big(\inf(1,t)\big)^{2k_0+1}|X|^2.$$
\end{lemma}

\medskip

\begin{proof}
According to (\ref{e9}), the above estimate trivially holds for $t=0$. We may therefore assume that $t>0$.
We deduce from (\ref{e5}), (\ref{e6}) and the Cauchy-Schwarz inequality that for all $\eps>0$, $X \in \rr^{2n}$, 
\begin{multline}\label{e10}
|N_j(X)| \leq 2\sqrt{\textrm{Re }q\big((\textrm{Im }F)^{j}X\big)}\sqrt{\textrm{Re }q\big((\textrm{Im }F)^{j+1}X\big)}\\
\leq \frac{\textrm{Re }q\big((\textrm{Im }F)^{j}X\big)}{\eps}+\eps\textrm{Re }q\big((\textrm{Im }F)^{j+1}X\big)=\frac{M_j(X)}{\eps}+\eps M_{j+1}(X),
\end{multline}
since the quadratic form $\textrm{Re }q \geq 0$ is non-negative. The very same arguments also show that for all $\eps>0$, $X \in \rr^{2n}$, 
\begin{multline}\label{e11}
\big|2\textrm{Re }q\big((\textrm{Im }F)^{j}X,(\textrm{Im }F)^{j+2}X\big)\big| \leq 2\sqrt{\textrm{Re }q\big((\textrm{Im }F)^{j}X\big)}\sqrt{\textrm{Re }q\big((\textrm{Im }F)^{j+2}X\big)}\\
\leq \frac{\textrm{Re }q\big((\textrm{Im }F)^{j}X\big)}{\eps}+\eps\textrm{Re }q\big((\textrm{Im }F)^{j+2}X\big)=\frac{M_j(X)}{\eps}+\eps M_{j+2}(X).
\end{multline}
With $\eps=\frac{2b_j t}{a_j}>0$, it follows from (\ref{e10}) that for all $t>0$, $X \in \rr^{2n}$,
\begin{multline}\label{e12}
-b_j t^{2j+2}N_j(X) \geq -b_j t^{2j+2}\Big(\frac{a_j}{2b_j t}M_j(X)+\frac{2b_j t}{a_j}M_{j+1}(X)\Big)\\ =-\frac{1}{2}a_j t^{2j+1}M_j(X)-\frac{2b_j^2}{a_j}t^{2j+3}M_{j+1}(X).
\end{multline}
We deduce from (\ref{e12}) that for all $t>0$, $X \in \rr^{2n}$,
\begin{multline}\label{e13}
-\sum_{j=0}^{k_0}b_j t^{2j+2}N_j(X) \geq -\frac{1}{2}a_0 t M_0(X)-\sum_{j=1}^{k_0}\Big(\frac{a_j}{2}+\frac{2b_{j-1}^2}{a_{j-1}}\Big)t^{2j+1}M_j(X)\\
-\frac{2b_{k_0}^2}{a_{k_0}}t^{2k_0+3}M_{k_0+1}(X).
\end{multline}
It follows from (\ref{e7}) and (\ref{e8}) that 
\begin{equation}\label{e14}
\forall 1 \leq j \leq k_0, \quad \frac{a_j}{2}+\frac{2b_{j-1}^2}{a_{j-1}}=\frac{3a_j}{4}, \qquad   \frac{2b_{k_0}^2}{a_{k_0}}=a_{k_0+1}.
\end{equation}
We obtain from (\ref{e13}) and (\ref{e14}) that for all $t>0$, $X \in \rr^{2n}$,
\begin{multline}\label{e15}
-\sum_{j=0}^{k_0}b_j t^{2j+2}N_j(X) \geq -\frac{1}{2}a_0 t M_0(X)-\frac{3}{4}\sum_{j=1}^{k_0}a_jt^{2j+1}M_j(X)\\
-a_{k_0+1} t^{2k_0+3}M_{k_0+1}(X).
\end{multline}
It follows from (\ref{e9}) and (\ref{e15}) that for all $t>0$, $X \in \rr^{2n}$,
\begin{equation}\label{e16}
Q_t(X) \geq \frac{1}{2}a_0 t M_0(X)+\frac{1}{4}\sum_{j=1}^{k_0}a_j t^{2j+1}M_j(X)\geq \frac{1}{4}\sum_{j=0}^{k_0}a_j t^{2j+1}M_j(X).
\end{equation}
On the other hand, we deduce from (\ref{e3}), (\ref{e5}) and (\ref{e16}) that for all $t>0$, $X \in \rr^{2n}$,
\begin{align}\label{e17}
Q_t(X) \geq & \  \frac{1}{4}(\inf_{0 \leq j \leq k_0}a_j)\big(\inf(1,t)\big)^{2k_0+1}\sum_{j=0}^{k_0}M_j(X) \\ \notag
\geq & \  \frac{c_0}{4}(\inf_{0 \leq j \leq k_0}a_j)\big(\inf(1,t)\big)^{2k_0+1}|X|^2.
\end{align}
This ends the proof of Lemma~\ref{lem1}.
\end{proof}

This time-dependent quadratic form is used to obtain key upper bounds:

\medskip

\begin{lemma}\label{lem1f}
There exists a positive constant $c>0$ such that for all $t \geq 0$, $X=(x,\xi) \in \rr^{2n}$, $X_0=(x_0,\xi_0) \in \rr^{2n}$, $X_0 \neq 0$, 
$$Q_t(X) \geq \frac{c}{|X_0|^2}\big(\inf(1,t)\big)^{2k_{X_0}+1}(\langle x_0,x\rangle+\langle \xi_0,\xi\rangle)^2,$$
where $0 \leq k_{X_0} \leq k_0$ stands for the index of the point $X_0 \in \rr^{2n}$ with respect to the singular space, with $|\cdot|$ the Euclidean norm on $\rr^{2n}$.
\end{lemma}

\medskip

\begin{proof}
Let $0 \leq k \leq k_0$. The very same arguments as in (\ref{2.3.99}) show that the non-negative quadratic form
\begin{equation}\label{we3}
r_k(X)=\sum_{j=0}^{k}\textrm{Re }q\big((\textrm{Im }F)^jX\big) \geq 0,
\end{equation}
satisfies that
\begin{equation}\label{we4}
X \in \rr^{2n}, \ r_k(X)=0 \Longleftrightarrow X \in \Big(\bigcap_{j=0}^{k}\textrm{Ker}\big[\textrm{Re }F(\textrm{Im }F)^j \big]\Big)\cap \rr^{2n}.
\end{equation}
Let $R_k \in S_{2n}(\rr)$ be the symmetric matrix associated to the quadratic form $r_k$ for the Euclidean scalar product on $\rr^{2n}$, 
\begin{equation}\label{we5.1}
r_k(X)=\langle R_k X,X\rangle, \quad X \in \rr^{2n}.
\end{equation}
It follows from (\ref{we1}) and (\ref{we4}) that 
$$\textrm{Ker }R_k=V_k^{\perp}, \quad \textrm{Ran }R_k=V_k,$$
since $R_k \in S_{2n}(\rr)$. It therefore exists a positive constant $C>0$ such that 
\begin{equation}\label{we5}
\forall 0 \leq k \leq k_0, \forall X \in V_k, \quad r_k(X) \geq C|X|^2.
\end{equation}
Let $X_0=(x_0,\xi_0) \in \rr^{2n}$, $X_0 \neq 0$ and $0 \leq k_{X_0} \leq k_0$ be the index of the point $X_0 \in \rr^{2n}$ with respect to the singular space.
For all $X \in \rr^{2n}$, we decompose $X=X'+X''$ with  $X' \in V_{k_{X_0}}$ and $X'' \in V_{k_{X_0}}^{\perp}$.
Since $X_0=(x_0,\xi_0) \in V_{k_{X_0}}$, we deduce from (\ref{we5.1})  and (\ref{we5}) that
\begin{multline}\label{we6}
(\langle x_0,x\rangle+\langle \xi_0,\xi\rangle)^2=|\langle X_0,X\rangle|^2=|\langle X_0,X'\rangle|^2 \\
\leq |X_0|^2|X'|^2  \leq \frac{|X_0|^2}{C}r_{k_{X_0}}(X')=\frac{|X_0|^2}{C}r_{k_{X_0}}(X).
\end{multline}
On the other hand, it follows from Lemma~\ref{lem1}, (\ref{e5}), (\ref{we3}) and (\ref{we6}) that for all $t \geq 0$, $X=(x,\xi) \in \rr^{2n}$, $X_0=(x_0,\xi_0) \in \rr^{2n}$, $X_0 \neq 0$, 
\begin{multline*}
Q_t(X) \geq \frac{1}{4}\sum_{j=0}^{k_{X_0}}a_jt^{2j+1}M_j(X) \geq \frac{1}{4}\Big(\inf_{0 \leq j \leq k_0} a_j\Big)\big(\inf(1,t)\big)^{2k_{X_0}+1}r_{k_{X_0}}(X)\\
\geq \frac{C}{4|X_0|^2}\Big(\inf_{0 \leq j \leq k_0} a_j\Big)\big(\inf(1,t)\big)^{2k_{X_0}+1}(\langle x_0,x\rangle+\langle \xi_0,\xi\rangle)^2,
\end{multline*}
since $0 \leq k_{X_0} \leq k_0$.
This ends the proof of Lemma~\ref{lem1f}.
\end{proof}

We recall from~\cite{mz} (Lemma~2) the following instrumental lemma:

\medskip

\begin{lemma}\label{ll2}
If $q_1$ and $q_2$ are two complex-valued quadratic forms on $\rr^{2n}$, then the Hamilton map associated to the complex-valued quadratic form defined by the Poisson 
bracket 
$$\{q_1,q_2\}=\frac{\partial q_1}{\partial \xi}\cdot \frac{\partial q_2}{\partial x}-\frac{\partial q_1}{\partial x}\cdot \frac{\partial q_2}{\partial \xi},$$
is $-2[F_1,F_2]$, where $[F_1,F_2]$ denotes the commutator of $F_1$ and $F_2$ respectively the Hamilton maps of $q_1$ and $q_2$.  
\end{lemma}

\medskip

\begin{proof}
We notice from (\ref{10}) and (\ref{a1}) that for all $X,Y \in \rr^{2n}$,
$$q_j(X,Y)=\sigma(X,F_jY)=-\sigma(F_jX,Y)=\langle(-\sigma F_j)X,Y \rangle, \quad j \in \{1,2\},$$
with 
\begin{equation}\label{inf13.5}
\sigma=\left(
  \begin{array}{cc}
  0 & I_n \\
  -I_n & 0 \\
  \end{array}
\right), \quad  \langle X,Y \rangle=\sum_{j=1}^{2n}{X_j Y_j}.
\end{equation}
It follows that for all $X \in \rr^{2n}$,
\begin{equation}\label{inf14}
\nabla q_j(X)=-2 \sigma F_j X. 
\end{equation}
Since we have
$$\{q_1,q_2\}(X)=\sigma\big(\nabla q_1(X),\nabla q_2(X)\big), \quad X=(x,\xi) \in \rr^{2n},$$ 
we deduce from (\ref{11}), (\ref{a1}), (\ref{inf13.5}) and (\ref{inf14}) that for all $X \in \rr^{2n}$,
\begin{multline}\label{inf15}
\{q_1,q_2\}(X)=\sigma(-2 \sigma F_1 X,-2 \sigma F_2 X)=4\langle\sigma^2 F_1 X,\sigma F_2 X \rangle
=-4 \langle F_1 X, \sigma F_2 X \rangle \\ 
=-4 \langle \sigma^T F_1 X,  F_2 X \rangle=4 \langle \sigma F_1 X,  F_2 X \rangle=4\sigma(F_1 X,F_2 X)=-4\sigma(X,F_1 F_2X),
\end{multline}
which induces that 
\begin{multline}\label{inf16}
\{q_1,q_2\}(X)=\frac{1}{2}\big[4\sigma(F_1 F_2X, X) -4\sigma(X,F_1 F_2X)\big]\\ =2\big[\sigma(X, F_2 F_1X) -\sigma(X,F_1 F_2X)\big]
=-2\sigma(X,[F_1,F_2]X), 
\end{multline}
by skew-symmetry of $\sigma$ and skew-symmetry of Hamilton maps with respect to $\sigma$.
Since 
\begin{multline*}
\sigma(X,[F_1,F_2]Y)=\sigma(X, F_1 F_2Y) -\sigma(X,F_2 F_1Y) =\sigma(F_2 F_1X, Y) -\sigma(F_1 F_2X,Y)\\
=-\sigma([F_1,F_2]X,Y)=\sigma(Y,[F_1,F_2]X),
\end{multline*}
we deduce that 
$$\sigma(X,-2[F_1,F_2]Y),$$ 
is the polarized form associated to the quadratic form $\{q_1,q_2\}$ and that $-2[F_1,F_2]$ is its Hamilton map. 
\end{proof}

The following lemma provides an explicit computation for the Poisson bracket 
$$H_{\textrm{Im}q}Q_t=\{\textrm{Im }q,Q_t\}=\frac{\partial \textrm{Im }q}{\partial \xi} \cdot \frac{\partial Q_t}{\partial x}-\frac{\partial \textrm{Im }q}{\partial x} \cdot \frac{\partial Q_t}{\partial \xi}.$$

\medskip

\begin{lemma}\label{lem2}
We have for all $t \geq 0$, $X \in \rr^{2n}$,
\begin{multline*}
(H_{\emph{\textrm{Im}}q}Q_t)(X)=2\sum_{j=0}^{k_0+1}a_jt^{2j+1}N_j(X)-4\sum_{j=0}^{k_0}b_jt^{2j+2}M_{j+1}(X)\\
-4\sum_{j=0}^{k_0}b_jt^{2j+2}\emph{\textrm{Re }}q\big((\emph{\textrm{Im }}F)^{j}X,(\emph{\textrm{Im }}F)^{j+2}X\big).
\end{multline*}
\end{lemma}

\medskip

\begin{proof}
We first prove that for all $l_1$, $l_2 \geq 0$,
\begin{multline}\label{mari5}
H_{\textrm{Im}q} \big[\textrm{Re }q\big((\textrm{Im }F)^{l_1}X,(\textrm{Im }F)^{l_2}X\big)\big]=2\textrm{Re }q\big((\textrm{Im }F)^{l_1+1}X,(\textrm{Im }F)^{l_2}X\big)
\\+2\textrm{Re }q\big((\textrm{Im }F)^{l_1}X,(\textrm{Im }F)^{l_2+1}X\big).
\end{multline}
By using the skew-symmetry property of Hamilton maps (\ref{a1}), we notice that the Hamilton map of the quadratic form
$$\tilde{r}(X)=\textrm{Re }q\big((\textrm{Im }F)^{l_1}X,(\textrm{Im }F)^{l_2}X\big),$$ 
is given by 
\begin{equation}\label{lol10}
\tilde{F}=\frac{1}{2}\big((-1)^{l_1}(\textrm{Im }F)^{l_1}\textrm{Re }F(\textrm{Im }F)^{l_2}+(-1)^{l_2}(\textrm{Im }F)^{l_2}\textrm{Re }F(\textrm{Im }F)^{l_1}\big),
\end{equation} 
since 
\begin{multline}\label{mari6}
(-1)^{l_1}\sigma\big(X,(\textrm{Im }F)^{l_1}\textrm{Re }F(\textrm{Im }F)^{l_2}Y\big)= \sigma\big((\textrm{Im }F)^{l_1}X,\textrm{Re }F(\textrm{Im }F)^{l_2}Y\big)  \\
=  \textrm{Re }q\big((\textrm{Im }F)^{l_1}X,(\textrm{Im }F)^{l_2}Y\big)
=-(-1)^{l_2}\sigma\big((\textrm{Im }F)^{l_2}\textrm{Re }F(\textrm{Im }F)^{l_1}X,Y\big),
\end{multline}
for all $l_1$, $l_2 \geq 0$, $X,Y \in \rr^{2n}$. We deduce from Lemma~\ref{ll2} that the Hamiton map of the quadratic form 
$$H_{\textrm{Im}q} \ \tilde{r}=\big\{\textrm{Im }q,\tilde{r}\big\}=\frac{\partial \textrm{Im }q}{\partial \xi}\cdot\frac{\partial \tilde{r}}{\partial x}-\frac{\partial \textrm{Im } q}{\partial x}\cdot\frac{\partial \tilde{r}}{\partial \xi},$$
is given by the commutator $-2[\textrm{Im }F,\tilde{F}]$, that is,
$$H_{\textrm{Im}q} \ \tilde{r}(X)=-2\sigma\big(X,[\textrm{Im }F,\tilde{F}]X\big).$$
By using (\ref{mari6}), a direct computation provides (\ref{mari5}).
We deduce from (\ref{e5}), (\ref{e6}), (\ref{e9}) and (\ref{mari5}) that 
\begin{multline}\label{e19}
(H_{\textrm{Im}q}Q_t)(X)=2\sum_{j=0}^{k_0+1}a_jt^{2j+1}N_j(X)\\
-4\sum_{j=0}^{k_0}b_jt^{2j+2}\big[M_{j+1}(X)+\textrm{Re }q\big((\textrm{Im }F)^{j}X,(\textrm{Im }F)^{j+2}X\big)\big].
\end{multline}
This ends the proof of Lemma~\ref{lem2}.
\end{proof}

\medskip

\begin{lemma}\label{lem3}
There exists a positive constant $C>0$ such that for all $0<t \leq 1$, $X\in \rr^{2n}$,
$$\Big(\frac{d}{dt}Q_t\Big)(X)+\frac{1}{2}(H_{\emph{\textrm{Im}}q}Q_t)(X)-CM_0(X) \leq 0.$$
\end{lemma}
 
\medskip

\begin{proof} 
With $\eps=\frac{t^2}{b_j}>0$, it follows from (\ref{e11}) that for all $0 \leq j \leq k_0$, $t>0$, $X \in \rr^{2n}$, 
\begin{multline}\label{e20}
-4b_jt^{2j+2}\textrm{Re }q\big((\textrm{Im }F)^{j}X,(\textrm{Im }F)^{j+2}X\big) \\
\leq 2b_jt^{2j+2}\Big(\frac{b_j}{t^2}M_j(X)+\frac{t^2}{b_j}M_{j+2}(X)\Big)=2b_j^2t^{2j}M_j(X)+2t^{2j+4}M_{j+2}(X).
\end{multline}
We deduce from (\ref{e20}) and Lemma~\ref{lem2} that for all $t>0$, $X\in \rr^{2n}$,
\begin{multline*}
(H_{\textrm{Im}q}Q_t)(X) \leq 2\sum_{j=0}^{k_0+1}a_jt^{2j+1}N_j(X)-4\sum_{j=0}^{k_0}b_jt^{2j+2}M_{j+1}(X)\\
+2\sum_{j=0}^{k_0}b_j^2t^{2j}M_j(X)
+2\sum_{j=0}^{k_0}t^{2j+4}M_{j+2}(X).
\end{multline*}
This implies that for all $t>0$, $X\in \rr^{2n}$,
\begin{multline}\label{e21}
(H_{\textrm{Im}q}Q_t)(X) \leq 2\sum_{j=0}^{k_0+1}a_jt^{2j+1}N_j(X)+2b_0^2M_0(X)+\big(2b_1^2-4b_0\big)t^{2}M_1(X)\\
+\sum_{j=2}^{k_0}(2b_j^2-4b_{j-1}+2)t^{2j}M_j(X)+(2-4b_{k_0})t^{2k_0+2}M_{k_0+1}(X)+2t^{2k_0+4}M_{k_0+2}(X).
\end{multline}
We deduce from (\ref{e9}) and (\ref{e21}) that for all $t>0$, $X\in \rr^{2n}$,
\begin{multline*}
\Big(\frac{d}{dt}Q_t\Big)(X)+\frac{1}{2}(H_{\textrm{Im}q}Q_t)(X)-CM_0(X)\leq \sum_{j=0}^{k_0}\big(a_j-(2j+2)b_j\big)t^{2j+1}N_j(X)\\
+a_{k_0+1}t^{2k_0+3}N_{k_0+1}(X)
+(a_0+b_0^2-C)M_0(X)+\big(3a_1+b_1^2-2b_0\big)t^{2}M_1(X)+t^{2k_0+4}M_{k_0+2}(X) \\
+\sum_{j=2}^{k_0}\big((2j+1)a_j+b_j^2-2b_{j-1}+1\big)t^{2j}M_j(X)
+\big((2k_0+3)a_{k_0+1}+1-2b_{k_0}\big)t^{2k_0+2}M_{k_0+1}(X).
\end{multline*}
We notice from (\ref{e7}) that 
\begin{multline}\label{e22}
a_{k_0}-(2k_0+2)b_{k_0}=\big(2c_1 b_{k_0}-(2k_0+2)\big)b_{k_0}\\
=\Big(\frac{4c_1}{3}\big(1+(2k_0+4)a_{k_0+1}\big) -(2k_0+2)\Big)b_{k_0}
=\Big(\frac{4}{3}\big(c_1+(2k_0+4)\big) -(2k_0+2)\Big)b_{k_0}>0.
\end{multline}
On the other hand, it follows from (\ref{e8}) that for all $1 \leq j \leq k_0$,
\begin{multline}\label{e23}
a_{j-1}-2jb_{j-1}=\Big(\frac{8b_{j-1}}{a_{j}}-2j\Big)b_{j-1}\\
=\Big(\frac{16}{3a_j}\Big(2+c_1+(2j+1)a_{j}+b_j^2+\frac{2\big(a_j-(2j+2)b_j\big)^2}{b_j}\Big)-2j\Big)b_{j-1}>0.
\end{multline}
With $\eps=\frac{b_jt}{2(a_{j}-(2j+2)b_{j})}>0$, we deduce from (\ref{e10}), (\ref{e22}) and (\ref{e23}) that for all $t>0$, $0 \leq j \leq k_0$,
\begin{equation}\label{e24}
|N_j(X)| \leq \frac{2(a_{j}-(2j+2)b_{j})}{b_jt}M_{j}(X)+\frac{b_jt}{2(a_{j}-(2j+2)b_{j})}M_{j+1}(X),
\end{equation}
\begin{equation}\label{e25}
|N_{k_0+1}(X)| \leq \frac{1}{t}M_{k_0+1}(X)+t M_{k_0+2}(X).
\end{equation}
It follows from (\ref{e24}) and (\ref{e25}) that for all $t>0$, $X \in \rr^{2n}$,
\begin{multline}\label{e26}
\sum_{j=0}^{k_0}\big(a_j-(2j+2)b_j\big)t^{2j+1}N_j(X)
+a_{k_0+1}t^{2k_0+3}N_{k_0+1}(X) \\
\leq \frac{2(a_{0}-2b_{0})^2}{b_0}M_{0}(X)
+\sum_{j=1}^{k_0}\Big(\frac{2(a_{j}-(2j+2)b_{j})^2}{b_j}+\frac{b_{j-1}}{2}\Big)t^{2j}M_{j}(X)\\
+\Big(a_{k_0+1}+\frac{b_{k_0}}{2}\Big)t^{2k_0+2}M_{k_0+1}(X)+a_{k_0+1}t^{2k_0+4} M_{k_0+2}(X).
\end{multline}
We deduce from (\ref{e26}) that for all $t>0$, $X \in \rr^{2n}$,
\begin{multline*}
\Big(\frac{d}{dt}Q_t\Big)(X)+\frac{1}{2}(H_{\textrm{Im}q}Q_t)(X)-CM_0(X)\leq 
\Big(a_0+b_0^2+\frac{2(a_{0}-2b_{0})^2}{b_0}-C\Big)M_0(X)\\
+\Big(3a_1+b_1^2+\frac{2(a_{1}-4b_{1})^2}{b_1}-\frac{3b_{0}}{2}\Big)t^{2}M_1(X)
+\Big((2k_0+4)a_{k_0+1}+1-\frac{3b_{k_0}}{2}\Big)t^{2k_0+2}M_{k_0+1}(X)
+\\
\sum_{j=2}^{k_0}\Big((2j+1)a_j+b_j^2+1+\frac{2(a_{j}-(2j+2)b_{j})^2}{b_j}-\frac{3b_{j-1}}{2}\Big)t^{2j}M_j(X)
+(1+a_{k_0+1})t^{2k_0+4}M_{k_0+2}(X) .
\end{multline*}
It follows from (\ref{e4}), (\ref{e5}) and (\ref{e7}) that for all $0<t \leq 1$,
\begin{multline}\label{e27}
(1+a_{k_0+1})t^{2k_0+4}M_{k_0+2}(X)\\ \leq c_1 (1+a_{k_0+1})t^{2k_0+4}\sum_{j=0}^{k_0}M_j(X)\leq (c_1+1)\sum_{j=0}^{k_0}t^{2j}M_j(X).
\end{multline}
We deduce from (\ref{e27}) that for all $0<t \leq 1$,
\begin{multline*}
\Big(\frac{d}{dt}Q_t\Big)(X)+\frac{1}{2}(H_{\textrm{Im}q}Q_t)(X)-CM_0(X)\leq 
\Big(c_1+1+a_0+b_0^2+\frac{2(a_{0}-2b_{0})^2}{b_0}-C\Big)M_0(X)+\\
\Big(c_1+1+3a_1+b_1^2+\frac{2(a_{1}-4b_{1})^2}{b_1}-\frac{3b_{0}}{2}\Big)t^{2}M_1(X)
+\Big((2k_0+4)a_{k_0+1}+1-\frac{3b_{k_0}}{2}\Big)t^{2k_0+2}M_{k_0+1}(X)
\\
+\sum_{j=2}^{k_0}\Big(c_1+(2j+1)a_j+b_j^2+2+\frac{2(a_{j}-(2j+2)b_{j})^2}{b_j}-\frac{3b_{j-1}}{2}\Big)t^{2j}M_j(X).
\end{multline*}
It follows from (\ref{e7}) and (\ref{e8}) that for all $0<t \leq 1$,
\begin{multline}\label{e28}
\Big(\frac{d}{dt}Q_t\Big)(X)+\frac{1}{2}(H_{\textrm{Im}q}Q_t)(X)-CM_0(X)\\
\leq \Big(c_1+1+a_0+b_0^2+\frac{2(a_{0}-2b_{0})^2}{b_0}-C\Big)M_0(X)-t^{2}M_1(X).
\end{multline}
By choosing the positive constant $C>0$ so that 
$$c_1+1+a_0+b_0^2+\frac{2(a_{0}-2b_{0})^2}{b_0}-C<0,$$
we deduce from (\ref{e5}) and (\ref{e28}) that 
for all $0<t \leq 1$, $X \in \rr^{2n}$,
\begin{equation}\label{e29}
\Big(\frac{d}{dt}Q_t\Big)(X)+\frac{1}{2}(H_{\textrm{Im}q}Q_t)(X)-CM_0(X) \leq 0.
\end{equation}
This ends the proof of Lemma~\ref{lem3}.

\end{proof}

Let $u \in \mathscr{S}(\rr^n)$. We consider the time-dependent functional
\begin{equation}\label{e30}
G(t)=(Q_t^w(x,D_x)e^{-\frac{t}{2}q^w}u,e^{-\frac{t}{2}q^w}u)_{L^2}+C\|e^{-\frac{t}{2}q^w}u\|_{L^2}^2,
\end{equation} 
for all $0 \leq t \leq 1$, where $(e^{-tq^w})_{t \geq 0}$ denotes the contraction semigroup generated by the quadratic operator $q^w(x,D_x)$ and $C>0$ is the positive constant given by Lemma~\ref{lem3}. 
We observe from (\ref{conti}) and (\ref{e9}) that the mapping $G$ is continuous on $[0,+\infty)$, differentiable on $(0,+\infty)$ and satisfies 
\begin{equation}\label{e31}
G(0)=C\|u\|_{L^2}^2.
\end{equation} 
On the other hand, we have
\begin{multline}\label{e32}
G'(t)=\Big(\Big(\frac{d}{dt}Q_t\Big)^w(x,D_x)e^{-\frac{t}{2}q^w}u,e^{-\frac{t}{2}q^w}u\Big)_{L^2}\\
-\frac{C}{2}\big(q^w(x,D_x)e^{-\frac{t}{2}q^w}u,e^{-\frac{t}{2}q^w}u\big)_{L^2}
-\frac{C}{2}\big(e^{-\frac{t}{2}q^w}u,q^w(x,D_x)e^{-\frac{t}{2}q^w}u\big)_{L^2}\\
-\frac{1}{2}\big(Q_t^w(x,D_x)q^w(x,D_x)e^{-\frac{t}{2}q^w}u,e^{-\frac{t}{2}q^w}u\big)_{L^2}
-\frac{1}{2}\big(Q_t^w(x,D_x)e^{-\frac{t}{2}q^w}u,q^w(x,D_x)e^{-\frac{t}{2}q^w}u\big)_{L^2}.
\end{multline} 
By using that $Q_t^w(x,D_x)$ is selfadjoint since its Weyl symbol is real-valued, we obtain that 
\begin{multline}\label{e33}
G'(t)=\Big(\Big(\frac{d}{dt}Q_t\Big)^w(x,D_x)e^{-\frac{t}{2}q^w}u,e^{-\frac{t}{2}q^w}u\Big)_{L^2}\\
-C\textrm{Re}\big(q^w(x,D_x)e^{-\frac{t}{2}q^w}u,e^{-\frac{t}{2}q^w}u\big)_{L^2}
-\textrm{Re}\big(Q_t^w(x,D_x)q^w(x,D_x)e^{-\frac{t}{2}q^w}u,e^{-\frac{t}{2}q^w}u\big)_{L^2}.
\end{multline} 
We notice from (\ref{e5}) that 
\begin{multline}\label{e34}
\textrm{Re}\big(q^w(x,D_x)e^{-\frac{t}{2}q^w}u,e^{-\frac{t}{2}q^w}u\big)_{L^2}=\big((\textrm{Re }q)^w(x,D_x)e^{-\frac{t}{2}q^w}u,e^{-\frac{t}{2}q^w}u\big)_{L^2}\\
=\big(M_0^w(x,D_x)e^{-\frac{t}{2}q^w}u,e^{-\frac{t}{2}q^w}u\big)_{L^2}.
\end{multline}
On the other hand, we have
\begin{align}\label{e35}
& \textrm{Re}\big(Q_t^w(x,D_x)i(\textrm{Im }q)^w(x,D_x)e^{-\frac{t}{2}q^w}u,e^{-\frac{t}{2}q^w}u\big)_{L^2}\\ \notag
=& \textrm{Re}\big(i(\textrm{Im }q)^w(x,D_x)e^{-\frac{t}{2}q^w}u,Q_t^w(x,D_x)e^{-\frac{t}{2}q^w}u\big)_{L^2}\\ \notag
=& \frac{1}{2}\big([Q_t^w(x,D_x),i(\textrm{Im }q)^w(x,D_x)]e^{-\frac{t}{2}q^w}u,e^{-\frac{t}{2}q^w}u\big)_{L^2},
\end{align} 
since $Q_t^w(x,D_x)$ is selfadjoint and $i(\textrm{Im }q)^w(x,D_x)$ is skew-adjoint. The Weyl calculus (see e.g. \cite{hormander}, Theorem~18.5.4) shows that the commutator $$[Q_t^w(x,D_x),i(\textrm{Im }q)^w(x,D_x)],$$ 
is a quadratic operator whose Weyl symbol is exactly given by the Poisson bracket
\begin{equation}\label{e36}
\frac{1}{i}\{Q_t,i\textrm{Im }q\}=-H_{\textrm{Im}q}Q_t.
\end{equation}
It follows from (\ref{e35}) and (\ref{e36}) that 
\begin{multline}\label{e37}
\textrm{Re}\big(Q_t^w(x,D_x)i(\textrm{Im }q)^w(x,D_x)e^{-\frac{t}{2}q^w}u,e^{-\frac{t}{2}q^w}u\big)_{L^2}\\ 
= -\frac{1}{2}\big((H_{\textrm{Im}q}Q_t)^w(x,D_x)e^{-\frac{t}{2}q^w}u,e^{-\frac{t}{2}q^w}u\big)_{L^2}.
\end{multline} 
We deduce from (\ref{e33}), (\ref{e34}) and (\ref{e37}) that
\begin{multline}\label{e38}
G'(t)=\Big(\Big(\frac{d}{dt}Q_t+\frac{1}{2}H_{\textrm{Im}q}Q_t-CM_0\Big)^w(x,D_x)e^{-\frac{t}{2}q^w}u,e^{-\frac{t}{2}q^w}u\Big)_{L^2}\\
-\textrm{Re}\big(Q_t^w(x,D_x)(\textrm{Re }q)^w(x,D_x)e^{-\frac{t}{2}q^w}u,e^{-\frac{t}{2}q^w}u\big)_{L^2}.
\end{multline} 
We need the following lemma:

\medskip

\begin{lemma}\label{lem4}
Let $\tilde{q} : \rr_x^n \times \rr_{\xi}^n \rightarrow \rr$, $n \geq 1$, be a non-negative quadratic form $\tilde{q} \geq 0$. Then, the quadratic operator $\tilde{q}^w(x,D_x)$ is accretive
$$\forall u \in \mathscr{S}(\rr^n), \quad (\tilde{q}^w(x,D_x)u,u)_{L^2} \geq 0.$$
\end{lemma}

\medskip

\begin{proof}
We deduce from~\cite{hormander} (Theorem~21.5.3) that there exists a real linear symplectic transformation $\chi : \rr^{2n} \rightarrow \rr^{2n}$ such that 
\begin{equation}\label{inf2}
(\tilde{q} \circ \chi)(x,\xi)=\sum_{j=1}^{k}{\lambda_j(\xi_j^2+x_j^2)}
+\sum_{j=k+1}^{k+l}{x_j^2}, 
\end{equation}
with $k,l \geq 0$ and $\lambda_j>0$ for all $j=1,...,k$. By the symplectic invariance of the Weyl quantization~\cite{hormander} (Theorem 18.5.9), we can find a metaplectic operator $\mathcal{T}$,
which is a unitary transformation of $\lde$ and an automorphism of the Schwartz space $\mathscr{S}(\rr^n)$ satisfying 
\begin{equation}\label{inf3}
\tilde{q}^w(x,D_x)=\mathcal{T}^{-1} \Big(\sum_{j=1}^{k}{\lambda_j(D_{x_j}^2+x_j^2)}
+\sum_{j=k+1}^{k+l}{x_j^2} \Big)\mathcal{T},
\end{equation}
with $D_{x_j}=i^{-1}\partial_{x_j}$. We obtain that for all $u \in \mathscr{S}(\rr^n)$,
$$(\tilde{q}^w(x,D_x)u,u)_{L^2}=\sum_{j=1}^{k}{\lambda_j (\|D_{x_j}\mathcal{T}u\|_{L^2}^2+\|x_j\mathcal{T}u\|_{L^2}^2 )}+\sum_{j=k+1}^{k+l}{\|x_j\mathcal{T}u\|_{L^2}^2} 
\geq  0,$$ 
since $\mathcal{T}$ is a unitary operator on $L^2(\rr^{n})$. This ends the proof of Lemma~\ref{lem4}.
\end{proof}

We notice from (\ref{e5}), (\ref{e6}), (\ref{e9}), Lemmas~\ref{lem2} and~\ref{lem3} that for all $0<t \leq 1$, 
$$\frac{d}{dt}Q_t+\frac{1}{2}H_{\textrm{Im}q}Q_t-CM_0 \leq 0,$$
is a non-positive quadratic form. It follows from (\ref{conti}) and Lemma~\ref{lem4} that for all $0<t \leq 1$,
\begin{equation}\label{e39}
\Big(\Big(\frac{d}{dt}Q_t+\frac{1}{2}H_{\textrm{Im}q}Q_t-CM_0\Big)^w(x,D_x)e^{-\frac{t}{2}q^w}u,e^{-\frac{t}{2}q^w}u\Big)_{L^2} \leq 0.
\end{equation} 
We deduce from (\ref{e38}) and (\ref{e39}) that for all $0<t \leq 1$,
\begin{equation}\label{e40}
G'(t) \leq-\textrm{Re}\big(Q_t^w(x,D_x)(\textrm{Re }q)^w(x,D_x)e^{-\frac{t}{2}q^w}u,e^{-\frac{t}{2}q^w}u\big)_{L^2}.
\end{equation} 
Setting
$$\Gamma=\frac{dx^2+d\xi^2}{\langle (x,\xi)\rangle^2},$$
with $\langle (x,\xi) \rangle^2=1+|x|^2+|\xi|^2$,
we consider the symbol class $S(\langle (x,\xi)\rangle^m,\Gamma)$, with $m \in \rr$, composed of smooth functions $a \in C^{\infty}(\rr^{2n})$ satisfying
$$\forall \alpha \in \mathbb{N}^{2n}, \exists C_{\alpha}>0, \forall (x,\xi) \in \rr^{2n}, \quad |\partial_{x,\xi}^{\alpha}a(x,\xi)| \leq C_{\alpha}\langle (x,\xi)\rangle^{m-|\alpha|}.$$
The metric $\Gamma$ is admissible, that is, slowly varying, satisfying the uncertainty principle and temperate. The function $\langle (x,\xi)\rangle^m$, with $m \in \rr$, is a $\Gamma$-slowly varying weight~\cite{Le} (Lemma~2.2.18). Furthermore, the gain function in the symbolic calculus associated to the symbol classes $S(\langle (x,\xi)\rangle^m,\Gamma)$ is given by
$\Lambda_{\Gamma}=\langle (x,\xi)\rangle^2$. It follows that any quadratic form is a first order symbol belonging to the symbol class $S(\Lambda_{\Gamma},\Gamma)$. We deduce from (\ref{e5}), (\ref{e6}), (\ref{e9}) and Lemma~\ref{lem1} that the non-negative symbol $Q_t(\textrm{Re }q) \geq 0$ belongs to the symbol class $S(\Lambda_{\Gamma}^2,\Gamma)$ uniformly with respect to the parameter $0 \leq t \leq 1$. It follows from the Fefferman-Phong inequality, see~\cite{bony} or~\cite{Le} (Theorem~2.5.5), that there exists $C_1>0$ such that for all $u \in \mathscr{S}(\rr^n)$, $0 \leq t \leq 1$,
\begin{equation}\label{e41}
\big((Q_t\textrm{Re }q)^w(x,D_x)u,u\big)_{L^2}+C_1\|u\|_{L^2}^2 \geq 0.
\end{equation}
On the other hand, elements of symbolic calculus in the Weyl quantization, see e.g. the composition formula (2.1.26) in~\cite{Le}, show that 
\begin{equation}\label{e42}
Q_t \sharp^w (\textrm{Re }q)=Q_t(\textrm{Re }q)+\frac{1}{2i}\{Q_t,\textrm{Re }q\}+P_0(t),
\end{equation}
where $P_0$ is a polynomial function, since the coefficients of the quadratic form $Q_t$ are polynomial functions in the $t$-variable. It follows from (\ref{e42}) that 
\begin{equation}\label{e43}
\textrm{Re}\big(Q_t \sharp^w (\textrm{Re }q)\big)=Q_t(\textrm{Re }q)+P_1(t),
\end{equation}
with $P_1(t)=\textrm{Re}(P_0(t))$, since $Q_t$ and $\textrm{Re }q$ are real-valued quadratic forms.
We deduce from (\ref{e43}) that 
\begin{align}\label{e44}
& \ \textrm{Re}\big(Q_t^w(x,D_x)(\textrm{Re }q)^w(x,D_x)e^{-\frac{t}{2}q^w}u,e^{-\frac{t}{2}q^w}u\big)_{L^2}\\ \notag
= & \ \big(\big[\textrm{Re}\big(Q_t \sharp^w(\textrm{Re }q)\big)\big]^w(x,D_x)e^{-\frac{t}{2}q^w}u,e^{-\frac{t}{2}q^w}u\big)_{L^2}\\ \notag
=& \ \big((Q_t\textrm{Re }q)^w(x,D_x)e^{-\frac{t}{2}q^w}u,e^{-\frac{t}{2}q^w}u\big)_{L^2}+P_1(t)\|e^{-\frac{t}{2}q^w}u\|_{L^2}^2.
\end{align}
It follows from (\ref{e41}) and (\ref{e44}) that for all $u \in \mathscr{S}(\rr^n)$, $0 \leq t \leq 1$,
\begin{equation}\label{e45}
\textrm{Re}\big(Q_t^w(x,D_x)(\textrm{Re }q)^w(x,D_x)e^{-\frac{t}{2}q^w}u,e^{-\frac{t}{2}q^w}u\big)_{L^2}\geq (P_1(t)-C_1)\|e^{-\frac{t}{2}q^w}u\|_{L^2}^2.
\end{equation}
We deduce from (\ref{e40}) and (\ref{e45}) that for all $u \in \mathscr{S}(\rr^n)$, $0 < t \leq 1$,
\begin{equation}\label{e46}
G'(t) \leq (C_1-P_1(t))\|e^{-\frac{t}{2}q^w}u\|_{L^2}^2 \leq C_2\|u\|_{L^2}^2,
\end{equation} 
with $C_2=C_1+\sup_{0 \leq t \leq 1}|P_1(t)|>0$, since $(e^{-t q^w})_{t \geq 0}$ is a contraction semigroup on $L^2$.
It follows from (\ref{e31}) and (\ref{e46}) that for all $u \in \mathscr{S}(\rr^n)$, $0 \leq t \leq 1$,
\begin{equation}\label{e47}
G(t)=G(0)+\int_0^tG'(s)ds\leq (C+C_2t)\|u\|_{L^2}^2\leq (C+C_2)\|u\|_{L^2}^2.
\end{equation} 
We deduce from (\ref{conti}), (\ref{e30}), Lemmas~\ref{lem1f} and~\ref{lem4} that for all $u \in \mathscr{S}(\rr^n)$, $t \geq 0$, $X_0=(x_0,\xi_0) \in \rr^{2n}$, $X_0 \neq 0$, 
\begin{equation}\label{e48}
G(t) \geq \frac{c}{|X_0|^2}\big(\inf(1,t)\big)^{2k_{X_0}+1}
\big(\textrm{Op}^w\big((\langle x_0,x\rangle+\langle \xi_0,\xi\rangle)^2\big)e^{-\frac{t}{2}q^w}u,e^{-\frac{t}{2}q^w}u\big)_{L^2},
\end{equation}
where $0 \leq k_{X_0} \leq k_0$ denotes the index of $X_0=(x_0,\xi_0) \in \rr^{2n}$ with respect to the singular space.
On the other hand, the Weyl calculus provides that 
$$(\langle x_0,x\rangle+\langle \xi_0,\xi\rangle)^2=(\langle x_0,x\rangle+\langle \xi_0,\xi\rangle) \sharp^w(\langle x_0,x\rangle+\langle \xi_0,\xi\rangle),$$
since the symbol $\langle x_0,x\rangle+\langle \xi_0,\xi\rangle$ is a linear form.
It follows from (\ref{e48}) that 
for all $u \in \mathscr{S}(\rr^n)$, $t \geq 0$, $X_0=(x_0,\xi_0) \in \rr^{2n}$, $X_0 \neq 0$, 
\begin{equation}\label{e48.a}
G(t) \geq \frac{c}{|X_0|^2}\big(\inf(1,t)\big)^{2k_{X_0}+1}
\|(\langle x_0,x\rangle+\langle \xi_0,D_x\rangle)e^{-\frac{t}{2}q^w}u\|_{L^2}^2.
\end{equation}
We deduce from (\ref{e47}) and (\ref{e48.a}) that there exists a positive constant $C_3>0$ such that for all $u \in \mathscr{S}(\rr^n)$, $0<t \leq 1$, $X_0=(x_0,\xi_0) \in \rr^{2n}$,   
\begin{equation}\label{e49}
\|(\langle x_0,x\rangle+\langle \xi_0,D_x\rangle)e^{-\frac{t}{2}q^w}u\|_{L^2} \leq C_3|X_0| t^{-(2k_{X_0}+1)/2}\|u\|_{L^2}.
\end{equation}

We then study the large time behavior of the differentiation of the semigroup in the phase space direction $X_0=(x_0,\xi_0) \in \rr^{2n}$.
According to (\ref{jkk1}), the eigenvalue with the smallest real part of the operator $q^w(x,D_x)$ is given by
\begin{equation}\label{qz4}
\mu_0=-i \sum_{\substack{\lambda \in \sigma(F)\\  -i\lambda \in \CC_+}} \lambda r_{\lambda}, 
\end{equation}
where $F$ is the Hamilton map of $q$, and $r_{\lambda}$ stands for the dimension of the space of generalized eigenvectors of $F$ in $\cc^{2n}$ associated to the eigenvalue $\lambda \in \sigma(F)$. We know from~\cite{OPPS} (Theorem~2.1) that this eigenvalue has algebraic multiplicity one.
Setting 
\begin{equation}\label{qz4b}
Q=q^w(x,D_x)-\mu_0, \quad \tau_0=2\min_{\substack{\lambda \in \sigma(F)\\ \textrm{Im}\lambda >0}}\textrm{Im }\lambda>0,
\end{equation}
we recall from~\cite{OPPS} (Theorem~2.2) the result of exponential return to equilibrium
\begin{equation}\label{qz5}
\forall 0 \leq \tau < \tau_0, \exists C_4>0, \forall t \geq 0, \quad \|e^{-tQ}-\Pi_0\|_{\mathcal{L}(L^2)} \leq C_4e^{-\tau t},
\end{equation}
where $\Pi_0$ is the rank-one spectral projection associated with the simple eigenvalue zero of the operator $Q$, and $\|\cdot\|_{\mathcal{L}(L^2)}$ stands for the norm of bounded operators on $L^2(\rr^n)$. We deduce from (\ref{qz-1}), (\ref{qz4}), (\ref{qz4b}) and (\ref{qz5}) that there exists a positive constant $C_5>0$ such that 
\begin{equation}\label{ev1}
\forall t \geq 0, \quad \|e^{-tq^w}\|_{\mathcal{L}(L^2)} \leq C_5e^{-\omega_0 t}.
\end{equation}
We observe from the semigroup property, (\ref{conti}), (\ref{e49}) and (\ref{ev1}) that for all $u \in \mathscr{S}(\rr^n)$, $t \geq \frac{1}{2}$, $X_0=(x_0,\xi_0) \in \rr^{2n}$,
\begin{multline}\label{e50}
\|(\langle x_0,x\rangle+\langle \xi_0,D_x\rangle)e^{-t q^w}u\|_{L^2}=\|(\langle x_0,x\rangle+\langle \xi_0,D_x\rangle)e^{-\frac{1}{2}q^w}e^{-(t-\frac{1}{2})q^w}u\|_{L^2} \\
\leq C_3 \sqrt{2}^{2k_{0}+1}|X_0|\|e^{-(t-\frac{1}{2})q^w}u\|_{L^2}\leq C_3C_5 \sqrt{2}^{2k_{0}+1}|X_0|e^{-\omega_0(t-\frac{1}{2})}\|u\|_{L^2}.
\end{multline}
We deduce from (\ref{e49}) and (\ref{e50}) that there exists a positive constant $C_0>0$ such that for all $u \in \mathscr{S}(\rr^n)$, $t>0$, $X_0=(x_0,\xi_0) \in \rr^{2n}$, 
\begin{equation}\label{e51}
 \|(\langle x_0,x\rangle+\langle \xi_0,D_x\rangle)e^{-tq^w}u\|_{L^2} \leq C_0|X_0|\inf(1,t)^{-(2k_{X_0}+1)/2}e^{-\omega_0t}\|u\|_{L^2}.
\end{equation}
By using the density of the Schwartz space in $L^2(\rr^n)$, this estimate extends as 
\begin{multline}\label{e51.6}
\forall u \in L^2(\rr^n), \forall t>0, \forall X_0=(x_0,\xi_0) \in \rr^{2n}, \\ 
\|(\langle x_0,x\rangle+\langle \xi_0,D_x\rangle)e^{-tq^w}u\|_{L^2} \leq C_0|X_0|\inf(1,t)^{-(2k_{X_0}+1)/2}e^{-\omega_0t}\|u\|_{L^2}.
\end{multline}
This ends the proof of Theorem~\ref{th}.

\subsection{Proof of Corollary~\ref{thj}}
This section is devoted to the proof of Corollary~\ref{thj}.
Let $q : \rr^{n}_x \times \rr_{\xi}^n \rightarrow \cc$, with $n \geq 1$, be a quadratic form with a non-negative real part $\textrm{Re }q \geq 0$ and a zero singular space
$$S=\Big(\bigcap_{j=0}^{2n-1}\textrm{Ker}\big[\textrm{Re }F(\textrm{Im }F)^j \big]\Big)\cap \rr^{2n}=\{0\}.$$ 
Let $m \geq 1$, $X_1=(x_1,\xi_1) \in \rr^{2n}$, ..., $X_m=(x_m,\xi_m) \in \rr^{2n}$, and $0 \leq k_0 \leq 2n-1$ be the smallest integer satisfying (\ref{e2}). We know from~\cite{mehler} (Proposition~5.8 and Theorem~5.12) that the semigroup 
$$e^{-tq^w}=\mathscr{K}_{e^{-2itF}} : \mathscr{S}(\rr^n) \rightarrow \mathscr{S}(\rr^n), \quad t>0,$$ 
defined by the quadratic operator $q^w(x,D_x)$, is a Fourier integral operator whose kernel is a Gaussian distribution associated to the positive symplectic linear bijection 
$$e^{-2itF} : \cc^{2n} \rightarrow \cc^{2n}.$$ 
It was proven in the proof of Proposition~5.8 in~\cite{mehler} (pp. 444-445) that for all $(x_0,\xi_0) \in \cc^{2n}$,
\begin{equation}\label{qz2}
(\langle x_0,x\rangle+\langle \xi_0,D_x\rangle)e^{-tq^w}=e^{-tq^w}(\langle y_0,x\rangle+\langle \eta_0,D_x\rangle),
\end{equation}
where $(y_0,\eta_0)=T_t(x_0,\xi_0)$, with $T_t=\mathcal{J}e^{2itF}\mathcal{J}^{-1}$, where $\mathcal{J}(x,\xi)=(\xi,-x)$. The notation $\langle \cdot,\cdot \rangle$ stands here for the bilinear dot product on $\cc^{n}$.
By using the semigroup property, we deduce from (\ref{qz2}) that  
\begin{align*}
&  (\langle x_1, x \rangle+\langle \xi_1,D_x \rangle)\ ... \ (\langle x_m, x \rangle+\langle \xi_m,D_x \rangle)e^{-tq^w}=\\
&  (\langle x_1, x \rangle+\langle \xi_1,D_x \rangle)\ ...\ (\langle x_m, x \rangle+\langle \xi_m,D_x \rangle)\underbrace{e^{-\frac{t}{m}q^w}...\ e^{-\frac{t}{m}q^w}}_{m \textrm{ factors}}=\\
&  (\langle x_1, x \rangle+\langle \xi_1,D_x \rangle)e^{-\frac{t}{m}q^w}(\langle y_2(t),x\rangle+\langle \eta_2(t),D_x\rangle)e^{-\frac{t}{m}q^w}... (\langle y_{m}(t),x\rangle+\langle \eta_{m}(t),D_x\rangle)e^{-\frac{t}{m}q^w},
\end{align*}
where 
$$(y_{j}(t),\eta_{j}(t))=(T_{\frac{t}{m}})^{j-1}(x_j,\xi_j)=\mathcal{J}e^{\frac{2i(j-1)tF}{m}}\mathcal{J}^{-1}(x_j,\xi_j), \qquad 2 \leq j \leq m.$$
It follows from Theorem~\ref{th} that there exists a positive constant $C>1$ such that for all 
$m \geq 1$, $X_1=(x_1,\xi_1)  \in \rr^{2n}$, ..., $X_m=(x_m,\xi_m)  \in \rr^{2n}$, $u \in L^2(\rr^n)$, $0<t \leq 1$, 
\begin{align*}
& \ \|(\langle x_1, x \rangle+\langle \xi_1,D_x \rangle) \ ... \ (\langle x_m, x \rangle+\langle \xi_m,D_x \rangle)e^{-tq^w}u\|_{L^2} \leq  C|X_1|\Big(\frac{m}{t}\Big)^{\frac{2k_{0}+1}{2}}\\
\times & \ \Big[ \big\|\big(\langle \textrm{Re }y_2(t),x\rangle+\langle \textrm{Re }\eta_2(t),D_x\rangle\big)e^{-\frac{t}{m}q^w}... \ \big(\langle y_{m}(t),x\rangle+\langle \eta_{m}(t),D_x\rangle\big)e^{-\frac{t}{m}q^w}u\big\|_{L^2}\\
+ & \ \big\|\big(\langle \textrm{Im }y_2(t),x\rangle+\langle \textrm{Im }\eta_2(t),D_x\rangle\big)e^{-\frac{t}{m}q^w}...\ \big(\langle y_{m}(t),x\rangle+\langle \eta_{m}(t),D_x\rangle\big)e^{-\frac{t}{m}q^w}u\big\|_{L^2}\Big],
\end{align*}
implying that 
\begin{align*}
& \ \|(\langle x_1, x \rangle+\langle \xi_1,D_x \rangle)\ ...\ (\langle x_m, x \rangle+\langle \xi_m,D_x \rangle)e^{-tq^w}u\|_{L^2} \leq  
|\mathcal{J}e^{\frac{2itF}{m}}\mathcal{J}^{-1}X_2|\Big(\frac{m}{t}\Big)^{\frac{2(2k_{0}+1)}{2}}\\
\times & \ 2C^2|X_1| \Big[ \big\|\big(\langle \textrm{Re }y_3(t),x\rangle+\langle \textrm{Re }\eta_3(t),D_x\rangle\big)e^{-\frac{t}{m}q^w} ...  \big(\langle y_{m}(t),x\rangle+\langle \eta_{m}(t),D_x\rangle\big)e^{-\frac{t}{m}q^w}u\big\|_{L^2}\\
+ & \ \big\|\big(\langle \textrm{Im }y_3(t),x\rangle+\langle \textrm{Im }\eta_3(t),D_x\rangle\big)e^{-\frac{t}{m}q^w} ...  \big(\langle y_{m}(t),x\rangle+\langle \eta_{m}(t),D_x\rangle\big)e^{-\frac{t}{m}q^w}u\big\|_{L^2}\Big].
\end{align*}
Iterating these estimates provides that for all 
$m \geq 1$, $X_1=(x_1,\xi_1)  \in \rr^{2n}$, ..., $X_m=(x_m,\xi_m)  \in \rr^{2n}$, $u \in L^2(\rr^n)$, $0<t \leq 1$,
\begin{align}\label{fg1}
& \ \|(\langle x_1, x \rangle+\langle \xi_1,D_x \rangle)\ ...\ (\langle x_m, x \rangle+\langle \xi_m,D_x \rangle)e^{-tq^w}u\|_{L^2}\\ \notag
\leq & \  2^{m-1}C^m\Big(\frac{m}{t}\Big)^{\frac{m(2k_{0}+1)}{2}}|X_1|\Big(\prod_{j=1}^{m-1}|\mathcal{J}e^{\frac{2ijtF}{m}}\mathcal{J}^{-1}X_{j+1}|\Big)\|u\|_{L^2}\\ \notag
\leq & \  2^{m-1}C^m\Big(\frac{m}{t}\Big)^{\frac{m(2k_{0}+1)}{2}}\Big(\prod_{j=1}^{m-1}e^{\frac{2jt\|F\|}{m}}\Big)\Big(\prod_{j=1}^{m}|X_j|\Big)\|u\|_{L^2}\\ \notag
\leq & \ (2Ce^{\|F\|})^m\Big(\frac{m}{t}\Big)^{\frac{m(2k_{0}+1)}{2}}\Big(\prod_{j=1}^{m}|X_j|\Big)\|u\|_{L^2}.
\end{align}
We deduce from the Stirling formula and (\ref{fg1}) that there exists a positive constant $\tilde{c}>1$ such that for all 
$m \geq 1$, $X_1=(x_1,\xi_1)  \in \rr^{2n}$, ..., $X_m=(x_m,\xi_m)  \in \rr^{2n}$, $u \in L^2(\rr^n)$, $0<t \leq 1$,
\begin{multline}\label{fg1444}
\|(\langle x_1, x \rangle+\langle \xi_1,D_x \rangle)\ ...\ (\langle x_m, x \rangle+\langle \xi_m,D_x \rangle)e^{-tq^w}u\|_{L^2} \\
\leq  \tilde{c}^m(m!)^{\frac{2k_{0}+1}{2}}\Big(\prod_{j=1}^{m}|X_j|\Big)t^{-\frac{m(2k_{0}+1)}{2}}\|u\|_{L^2}.
\end{multline}
On the other hand, it follows from the semigroup property and (\ref{fg1444}) that for all 
$m \geq 1$, $X_1=(x_1,\xi_1)  \in \rr^{2n}$, ..., $X_m=(x_m,\xi_m)  \in \rr^{2n}$, $u \in L^2(\rr^n)$, $t \geq 1$,
\begin{align}\label{fg46}
& \ \|(\langle x_1, x \rangle+\langle \xi_1,D_x \rangle)\ ...\ (\langle x_m, x \rangle+\langle \xi_m,D_x \rangle)e^{-tq^w}u\|_{L^2}\\ \notag
= & \ \|(\langle x_1, x \rangle+\langle \xi_1,D_x \rangle)\ ...\ (\langle x_m, x \rangle+\langle \xi_m,D_x \rangle)e^{-q^w}e^{-(t-1)q^w}u\|_{L^2}\\ \notag
\leq & \ \tilde{c}^m(m!)^{\frac{2k_{0}+1}{2}}\Big(\prod_{j=1}^{m}|X_j|\Big)\|e^{-(t-1)q^w}u\|_{L^2}.
\end{align}
We deduce from (\ref{ev1}) and (\ref{fg46}) that there exists a positive constant $\tilde{C}>1$ such that for all 
$m \geq 1$, $X_1=(x_1,\xi_1)  \in \rr^{2n}$, ..., $X_m=(x_m,\xi_m)  \in \rr^{2n}$, $u \in L^2(\rr^n)$, $t \geq 1$,
\begin{align*}
& \ \|(\langle x_1, x \rangle+\langle \xi_1,D_x \rangle)\ ...\ (\langle x_m, x \rangle+\langle \xi_m,D_x \rangle)e^{-tq^w}u\|_{L^2}\\
\leq & \ C_5\tilde{c}^m(m!)^{\frac{2k_{0}+1}{2}}\Big(\prod_{j=1}^{m}|X_j|\Big)e^{-\omega_0 (t-1)}\|u\|_{L^2} \\
\leq & \ \tilde{C}^m(m!)^{\frac{2k_{0}+1}{2}}\Big(\prod_{j=1}^{m}|X_j|\Big)e^{-\omega_0 t}\|u\|_{L^2}.
\end{align*}
This ends the proof of Corollary~\ref{thj}.

\subsection{Proof of Corollary~\ref{th0}}
This section is devoted to the proof of Corollary~\ref{th0}. To that end, we consider the Hilbert spaces 
\begin{equation}\label{jk1}
\mathcal{H}^s=\big\{u \in L^2(\rr^n) : \langle (x,D_x) \rangle^su \in L^2(\rr^n)\big\}, \quad \|u\|_{\mathcal{H}^s}=\|\langle (x,D_x) \rangle^su\|_{L^2}, \quad s  \geq 0,
\end{equation}
with $\|\cdot\|_{L^2}$ the $L^2(\rr^n)$-norm and $\langle (x,D_x) \rangle^2=1+|D_x|^2+|x|^2$, equipped with the dot product
\begin{equation}\label{jk2}
(u,v)_{\mathcal{H}^s}=\big(\langle (x,D_x) \rangle^su,\langle (x,D_x) \rangle^sv\big)_{L^2}.
\end{equation}
Here the operator 
$$\langle (x,D_x) \rangle^s=(1+D_x^2+x^2)^{\frac{s}{2}},$$ 
equipped with the domain 
\begin{equation}\label{dv1}
D(\langle (x,D_x) \rangle^s)=\mathcal{H}^s,
\end{equation}
is defined through the functional calculus of the harmonic oscillator. Notice that this definition coincides with the definition of fractional powers of positive operators given in~\cite{interpolation} (Chapter~4).

Let $q : \rr^{n}_x \times \rr_{\xi}^n \rightarrow \cc$, with $n \geq 1$, be a quadratic form with a non-negative real part $\textrm{Re }q \geq 0$ and a zero singular space $S=\{0\}$. 
Let $0 \leq k_0 \leq 2n-1$ be the smallest integer satisfying
$$\Big(\bigcap_{j=0}^{k_0}\textrm{Ker}\big[\textrm{Re }F(\textrm{Im }F)^j \big]\Big)\cap \rr^{2n}=\{0\}.$$
By using that 
$$\|u\|_{\mathcal{H}^1}^2=\big(\langle (x,D_x) \rangle^2u,u\big)_{L^2}=\|u\|_{L^2}^2+\|xu\|_{L^2}^2+\|D_xu\|_{L^2}^2,$$
it follows from (\ref{qz-1}), (\ref{ev1}) and Theorem~\ref{th} that there exists a positive constant $C>0$ such that 
\begin{equation}\label{ert1}
\forall t>0, \forall u \in L^2(\rr^n), \quad \|e^{-tq^w}u\big\|_{\mathcal{H}^1} \leq \frac{C}{t^{\frac{2k_0+1}{2}}}\|u\|_{L^2}.
\end{equation}
We need the following result of real interpolation:

\medskip

\begin{proposition}\label{lprop}
Let $X$ be a Hilbert space and $A : D(A) \subset X \rightarrow X$ a maximal accretive operator
$$\forall u \in D(A), \quad \emph{\textrm{Re}}(Au,u) \geq 0,$$
such that $(-A,D(A))$ is the generator of a strongly continuous semigroup $T(t)$. Assume that there exists a Banach space $E \subset X$, $\rho>1$, $C>0$ such that 
$$\forall t>0, \quad \|T(t)\|_{\mathcal{L}(X,E)} \leq \frac{C}{t^{\rho}},$$
and that $t \mapsto T(t)u$ is measurable with values in $E$ for each $u \in X$. Then, the continuous inclusion 
$$D(A) \subset (X,E)_{\frac{1}{\rho},2},$$
holds, where $(X,E)_{\frac{1}{\rho},2}$ denotes the space obtained by real interpolation.
\end{proposition}

\medskip

We refer the reader to the book~\cite{interpolation} (Corollary~5.13) for a proof of this result.
We consider the degenerate case when $k_0 \geq 1$.
We deduce from (\ref{ert1}) and Proposition~\ref{lprop} that the continuous inclusion 
\begin{equation}\label{jk3}
D(q^w) \subset \big(L^2(\rr^n),\mathcal{H}^1\big)_{\frac{2}{2k_0+1},2},
\end{equation} 
holds. We observe that the Hilbert space $\mathcal{H}^1$ is included and dense in $L^2(\rr^n)$. It follows from~\cite{interpolation} (Corollary~4.37) the following correspondence between real and complex interpolation spaces
\begin{equation}\label{jk4}
\forall 0 <\theta <1, \quad \big[L^2(\rr^n),\mathcal{H}^1\big]_{\theta}=\big(L^2(\rr^n),\mathcal{H}^1\big)_{\theta,2}.
\end{equation}
By using that $\mathcal{H}^1=D(\langle (x,D_x) \rangle)$ is the domain of the selfadjoint operator $\langle (x,D_x) \rangle$, we deduce from (\ref{dv1}) and~\cite{interpolation} (Theorem~4.36) that 
\begin{equation}\label{jk5}
\big[L^2(\rr^n),\mathcal{H}^1\big]_{\theta}=\big[D\big(\langle (x,D_x) \rangle^0\big),D\big(\langle (x,D_x) \rangle^1\big)\big]_{\theta}=D\big(\langle (x,D_x) \rangle^\theta\big)=\mathcal{H}^{\theta},
\end{equation}
for all $0 <\theta <1$. We observe from (\ref{jk3}), (\ref{jk4}) and (\ref{jk5}) that the continuous inclusion 
\begin{equation}\label{jk6}
D(q^w)  \subset \mathcal{H}^{\frac{2}{2k_0+1}},
\end{equation} 
holds. This ends the proof of Corollary~\ref{th0} in the case when $k_0 \geq 1$.

When $k_0=0$, we check that the non-negative quadratic form $\textrm{Re }q$ is in fact positive definite. 
Indeed, the very same arguments as in (\ref{2.3.99}) show that
$$X \in \textrm{Ker}(\textrm{Re }F) \Leftrightarrow \textrm{Re }q(X)=0,$$
when $X \in \rr^{2n}$. When $k_0=0$, the condition (\ref{e2}) therefore reads as
$$\textrm{Ker}(\textrm{Re }F)\cap \rr^{2n}=\{0\}.$$
This implies that the quadratic form $\textrm{Re }q$ is positive definite.
The symbol $q$ is thus elliptic and the result of Corollary~\ref{th0} is then a direct consequence of the ellipticity of the operator $q^w(x,D_x)$, see e.g.~\cite{hor_math} (Lemma~3.1 and Theorem~3.3).

\subsection{Proof of Theorem~\ref{th-1}}
This section is devoted to the proof of Theorem~\ref{th-1}.
Let $q : \rr^{n}_x \times \rr_{\xi}^n \rightarrow \cc$, with $n \geq 1$, be a quadratic form with a non-negative real part $\textrm{Re }q \geq 0$ and a zero singular space $S=\{0\}$. 
Let $0 \leq k_0 \leq 2n-1$ be the smallest integer satisfying
$$\Big(\bigcap_{j=0}^{k_0}\textrm{Ker}\big[\textrm{Re }F(\textrm{Im }F)^j \big]\Big)\cap \rr^{2n}=\{0\}.$$
Let $0 \leq k \leq k_0$. We begin by noticing from (\ref{rg0}), (\ref{e5}) and Lemma~\ref{lem1} that there exists a positive constant $c_1>0$ such that for all $0 \leq t \leq 1$, $X \in \rr^{2n}$, 
$$Q_t(X) \geq c_1 t^{2k+1}r_k(X).$$
It follows from Lemma~\ref{lem4} that 
\begin{equation}\label{rg5}
\forall u \in \mathscr{S}(\rr^n), \quad (Q_t^w(x,D_x)u,u)_{L^2} \geq c_1 t^{2k+1}(r_k^w(x,D_x)u,u)_{L^2} \geq 0.
\end{equation}
We deduce from (\ref{conti}), (\ref{e30}), (\ref{e47}) and (\ref{rg5}) that for all $0 \leq t \leq 1$, $u \in \mathscr{S}(\rr^n)$,
\begin{equation}\label{rg6}
0 \leq c_1 t^{2k+1}(r_k^w(x,D_x)e^{-\frac{t}{2}q^w}u,e^{-\frac{t}{2}q^w}u)_{L^2}\leq (C+C_2)\|u\|_{L^2}^2.
\end{equation}
On the other hand, we have for all $0 \leq t \leq 1$, $u \in \mathscr{S}(\rr^n)$,
\begin{equation}\label{rg7}
0 \leq c_1 t^{2k+1}\|e^{-\frac{t}{2}q^w}u\|_{L^2}^2\leq c_1\|u\|_{L^2}^2,
\end{equation}
since $(e^{-tq^w})_{t \geq 0}$ is a contraction semigroup on $L^2$. It follows from (\ref{rg1}), (\ref{rg6}) and (\ref{rg7}) that for all $0 \leq t \leq 1$, $u \in \mathscr{S}(\rr^n)$,
\begin{equation}\label{rg6.cc}
0 \leq c_1 t^{2k+1}(\Lambda_k^2e^{-\frac{t}{2}q^w}u,e^{-\frac{t}{2}q^w}u)_{L^2}\leq (c_1+C+C_2)\|u\|_{L^2}^2,
\end{equation}
implying that
\begin{equation}\label{rg8}
\exists c_2>0, \forall 0 < t \leq 1, \forall u \in \mathscr{S}(\rr^n), \quad  \|\Lambda_k e^{-\frac{t}{2}q^w}u\|_{L^2}\leq \frac{c_2}{t^{\frac{2k+1}{2}}}\|u\|_{L^2}.
\end{equation}
We observe from (\ref{ev1}), (\ref{rg8}) and the semigroup property that for all $u \in \mathscr{S}(\rr^n)$, $t \geq 1$,
\begin{multline}\label{rg9}
\|\Lambda_k e^{-\frac{t}{2}q^w}u\|_{L^2}=\|\Lambda_ke^{-\frac{1}{2}q^w}e^{-\frac{1}{2}(t-1)q^w}u\|_{L^2} \\
\leq c_2\|e^{-\frac{1}{2}(t-1)q^w}u\|_{L^2}\leq c_2C_5e^{-\frac{\omega_0}{2}(t-1)}\|u\|_{L^2}.
\end{multline}
We consider the Hilbert spaces
\begin{equation}\label{jk1.t}
\mathscr{H}_k=\big\{u \in L^2(\rr^n) : \Lambda_ku \in L^2(\rr^n)\big\}, \quad \|u\|_{\mathscr{H}_k}=\|\Lambda_ku\|_{L^2}, \quad 0 \leq k \leq k_0,
\end{equation}
equipped with the dot product
\begin{equation}\label{jk2.t}
(u,v)_{\mathscr{H}_k}=(\Lambda_ku,\Lambda_kv)_{L^2}.
\end{equation} 
It follows from (\ref{qz-1}), (\ref{rg8}), (\ref{rg9}) and (\ref{jk1.t}) that for all $0 \leq k \leq k_0$, 
\begin{equation}\label{rg10}
\exists C_k>0, \forall t>0, \forall u \in \mathscr{S}(\rr^n), \quad  \|e^{-tq^w}u\|_{\mathscr{H}_k}=\|\Lambda_k e^{-tq^w}u\|_{L^2}\leq \frac{C_k}{t^{\frac{2k+1}{2}}}\|u\|_{L^2}.
\end{equation}
We observe that the Hilbert spaces $L^2(\rr^n)$ and $\mathscr{H}_k$ satisfy $\mathscr{H}_k \subset L^2(\rr^n)$, $\mathscr{H}_k$ is dense in $L^2(\rr^n)$. It follows from~\cite{interpolation} (Corollary~4.37) the following correspondence between real and complex interpolation spaces
\begin{equation}\label{jk4.t}
\forall 0 \leq k \leq k_0, \forall 0 <\theta <1, \quad \big[L^2(\rr^n),\mathscr{H}_k\big]_{\theta}=\big(L^2(\rr^n),\mathscr{H}_k\big)_{\theta,2}.
\end{equation}
By using that $\mathscr{H}_k=D(\Lambda_k)$ is the domain of the operator $\Lambda_k$ and that $\Lambda_k^2$ is a positive selfadjoint operator, we deduce from~\cite{interpolation} (Theorem~4.36) that for all $0 \leq k \leq k_0$, $0 <\theta <1$,
\begin{equation}\label{jk5.t}
\big[L^2(\rr^n),\mathscr{H}_k\big]_{\theta}=\big[D\big((\Lambda_k^2)^0\big),D\big((\Lambda_k^2)^{\frac{1}{2}}\big)\big]_{\theta}=D\big((\Lambda_k^2)^{\frac{\theta}{2}}\big)=D(\Lambda_k^{\theta}).
\end{equation}
When $k \geq 1$, we deduce from (\ref{rg10}) and~Proposition~\ref{lprop} that the continuous inclusion
\begin{equation}\label{jk3.t}
D(q^w) \subset \big(L^2(\rr^n),\mathscr{H}_k\big)_{\frac{2}{2k+1},2},
\end{equation} 
holds. We therefore obtain from (\ref{jk4.t}), (\ref{jk5.t}) and (\ref{jk3.t}) that the continuous inclusion
\begin{equation}\label{jk6.t}
\forall k \geq 1, \quad D(q^w)  \subset D(\Lambda_k^{\frac{2}{2k+1}}),
\end{equation} 
holds. This implies that there exists a positive constant $c>0$ such that 
\begin{equation}\label{yul1}
\forall 1 \leq k \leq k_0, \forall u \in D(q^w), \quad \|\Lambda_k^{\frac{2}{2k+1}}u\|_{L^2} \leq c(\|q^w(x,D_x)u\|_{L^2}+\|u\|_{L^2}).
\end{equation}
On the other hand, we deduce from (\ref{rg0}) and (\ref{rg1}) that 
\begin{equation}\label{wd1}
\forall u \in D(q^w), \quad \textrm{Re}(q^w(x,D_x)u+u,u)_{L^2}=(\Lambda_0^2u,u)_{L^2}=\|\Lambda_0u\|_{L^2}^2.
\end{equation}
It follows from (\ref{wd1}) and the Cauchy-Schwarz inequality that for all $u \in D(q^w)$, 
\begin{equation}\label{wd1.5}
\|\Lambda_0u\|_{L^2}^2 \leq \|q^w(x,D_x)u+u\|_{L^2}\|u\|_{L^2}\\ \leq \|q^w(x,D_x)u\|_{L^2}\|u\|_{L^2}+\|u\|_{L^2}^2,
\end{equation}
that is
\begin{equation}\label{wd1.6}
\exists C_0>0, \forall u \in D(q^w), \quad \|\Lambda_0u\|_{L^2} \leq C_0(\|q^w(x,D_x)u\|_{L^2}+\|u\|_{L^2}).
\end{equation}
Gathering the estimates (\ref{yul1}) and (\ref{wd1.6}) provides the result of Theorem~\ref{th-1}.
This ends the proof of Theorem~\ref{th-1}.

\subsection{Proof of Corollary~\ref{thj_symp}}
This section is devoted to the proof of Corollary~\ref{thj_symp}. To that end, we begin by noticing that both vector subspaces $S$ and $S^{\sigma \perp}$ are stable by the real and imaginary parts $\textrm{Re }F$ and $\textrm{Im }F$ of the Hamilton map $F$. Indeed, it follows from the Cayley-Hamilton theorem and the definition (\ref{h1bis}) that the singular space is equal to the following infinite intersection of kernels
$$S=\Big(\bigcap_{j=0}^{+\infty}\textrm{Ker}
\big[\textrm{Re }F(\textrm{Im }F)^j \big]\Big)\cap \rr^{2n}.$$
With this description, we readily check that the vector subspace $S$ is stable by the real and imaginary parts $\textrm{Re }F$ and $\textrm{Im }F$ of the Hamilton map $F$, 
\begin{equation}\label{vb1}
(\textrm{Re }F) S=\{0\}, \quad (\textrm{Im }F)S \subset S.
\end{equation}
By duality, we also notice from (\ref{a1}) that these stability properties hold as well for the vector subspace~$S^{\sigma \perp}$,
\begin{equation}\label{vb2}
(\textrm{Re }F) S^{\sigma \perp} \subset S^{\sigma \perp}, \quad (\textrm{Im }F)S^{\sigma \perp} \subset S^{\sigma \perp}.
\end{equation}
When the singular space has a symplectic structure, the vector space $S^{\sigma \perp}$ has also necessarily a symplectic structure and the phase space can be split into the direct sum of the two symplectically orthogonal spaces $S$ and $S^{\sigma \perp}$,
$$\rr^{2n}=S\oplus^{\sigma \perp} S^{\sigma \perp}, \quad \textrm{dim}_{\rr}S^{\sigma \perp}=2n', \quad  \textrm{dim}_{\rr}S=2n'', \quad n=n'+n''.$$
In the case when the singular space $S= \rr^{2n}$ is equal to the whole phase space, the result of Corollary~\ref{thj_symp} is trivial as $S^{\sigma \perp}=\{0\}$. We therefore assume from now that the singular space $S \neq \rr^{2n}$ is distinct from the whole phase space. It implies in particular that $n' \geq 1$.
By using this splitting of the phase space
\begin{equation}\label{evg22}
X=X'+X'', \quad X=(x,\xi)\in \rr^{2n}, \quad X' \in S^{\sigma \perp}, \quad X'' \in S,
\end{equation} 
we deduce from definition (\ref{10}), (\ref{vb1}), (\ref{vb2}) and the symplectic orthogonality that
\begin{multline*}
q(X)=\sigma(X,FX)=\sigma\big(X'+X'',F(X'+X'')\big)=\sigma(X',FX')+\sigma(X'',FX'')\\
=\sigma(X',F|_{S^{\sigma \perp}}X')+i\sigma\big(X'',(\textrm{Im }F)|_{S}X''\big).
\end{multline*}
It follows that the quadratic form $q$ is equal to the sum of a  purely imaginary-valued quadratic form defined on $S$ and another one defined on~$S^{\sigma \perp}$ with a non-negative real part
\begin{equation}\label{evg23}
q=q|_{S^{\sigma \perp}}+i (\textrm{Im }q)|_S.
\end{equation}
We notice that the quadratic form $q|_{S^{\sigma \perp}}$ satisfies the assumptions of Corollary~\ref{thj} since $S \cap S^{\sigma \perp}=\{0\}$. On the other hand, we deduce from the symplectic structure of the two vector subspaces $S$ and $S^{\sigma \perp}$ that the two operators $q|_{S^{\sigma \perp}}^w$ and $(\textrm{Im }q)|_S^w$ commute. We denote $(e_1,...,e_n,\eps_1,...,\eps_n)$ the canonical symplectic basis of the phase space
$$(x,\xi)=x_1e_1+...+x_ne_n+\xi_1\eps_1+...+\xi_n \eps_n \in \rr_x^n \times \rr_{\xi}^n.$$
Let $(\tilde{e}_1,...,\tilde{e}_{n'},\tilde{\eps}_1,...,\tilde{\eps}_{n'})$ be a symplectic basis of $S^{\sigma \perp}$ and $(\tilde{e}_{n'+1},...,\tilde{e}_{n'+n''},\tilde{\eps}_{n'+1},...,\tilde{\eps}_{n'+n''})$ a symplectic basis of $S$, 
$$\sigma(\tilde{e}_j,\tilde{e}_k)=\sigma(\tilde{\eps}_j,\tilde{\eps}_k)=0, \quad \sigma(\tilde{\eps}_j,\tilde{e}_k)=\delta_{j,k},$$
with $\delta_{j,k}$ being the Kronecker symbol. We consider $\chi : \rr^{2n} \rightarrow \rr^{2n}$ the linear symplectic transformation satisfying 
$$\chi(\tilde{e}_j)=e_j, \quad \chi(\tilde{\eps}_j)=\eps_j, \quad 1 \leq j \leq n.$$
By the symplectic invariance of the Weyl quantization~\cite{hormander} (Theorem 18.5.9), we can find a metaplectic operator $\mathcal{T}$,
which is a unitary transformation of $\lde$ and an automorphism of the Schwartz space $\mathscr{S}(\rr^n)$ and of the space of tempered distributions $\mathscr{S}'(\rr^n)$ such that for all $a \in \mathscr{S}'(\rr^{2n})$,
\begin{equation}\label{inf3_1}
(a \circ \chi^{-1})^w(x,D_x)=\mathcal{T}^{-1}a^w(x,D_x) \mathcal{T}.
\end{equation}
According to definition (\ref{10}), the Hamilton map of the quadratic form 
$$(x',\xi') \mapsto \tilde{q}|_{S^{\sigma \perp}}(x',\xi')=(q|_{S^{\sigma \perp}} \circ \chi^{-1})(x',0'';\xi',0''),$$ 
is given by $\chi \circ F|_{S^{\sigma \perp}} \circ \chi^{-1}$. We readily check that this quadratic form satisfies the assumptions of Corollary~\ref{thj}. Furthermore, we observe that the smallest integer $0 \leq k_0 \leq 2n'-1$  satisfying
$$\Big(\bigcap_{j=0}^{k_0}\textrm{Ker}
\big[\textrm{Re }F|_{S^{\sigma \perp}}(\textrm{Im }F|_{S^{\sigma \perp}})^j \big]\Big)\cap \rr^{2n'}=\{0\},$$
is also the  smallest integer $0 \leq k_0 \leq 2n'-1$  satisfying
$$\Big(\bigcap_{j=0}^{k_0}\textrm{Ker}
\big[\textrm{Re}\big(\chi|_{S^{\sigma \perp}}\circ F|_{S^{\sigma \perp}}\circ (\chi|_{S^{\sigma \perp}})^{-1}\big)\big(\textrm{Im}\big(\chi|_{S^{\sigma \perp}}\circ F|_{S^{\sigma \perp}}\circ(\chi|_{S^{\sigma \perp}})^{-1}\big)\big)^j \big]\Big)\cap \rr^{2n'}=\{0\}.$$
In the following, this integer is denoted $k_0$. We also notice that 
$$\sigma\big(\chi|_{S^{\sigma \perp}} \circ F|_{S^{\sigma \perp}} \circ (\chi|_{S^{\sigma \perp}})^{-1}\big)=\sigma\big(F|_{S^{\sigma \perp}}\big).$$
By applying the results of Corollary~\ref{thj}, we obtain that there exists a positive constant $C>1$ such that for all $m \geq 1$, $X'_1=(x_1',\xi_1')  \in \rr^{2n'}$, ..., $X_m'=(x_m',\xi_m')  \in \rr^{2n'}$, $u \in L^2(\rr^{n'})$,
\begin{multline}\label{gh1}
\forall 0<t \leq 1, \quad \big\|(\langle x_1', x' \rangle+\langle \xi_1',D_{x'} \rangle)\ ...\ (\langle x_m', x' \rangle+\langle \xi_m',D_{x'} \rangle) e^{-t\tilde{q}|_{S^{\sigma \perp}}^w}u\big\|_{L^2(\rr^{n'})}\\
\leq C^{m}(m!)^{\frac{2k_0+1}{2}}\Big(\prod_{j=1}^m|X_j'|\Big)t^{-\frac{(2k_0+1)m}{2}}\|u\|_{L^2(\rr^{n'})},
\end{multline}
\begin{multline}\label{gh2}
\forall t \geq 1, \quad \big\|(\langle x_1', x' \rangle+\langle \xi_1',D_{x'} \rangle)\ ...\ (\langle x_m', x' \rangle+\langle \xi_m',D_{x'} \rangle) e^{-t\tilde{q}|_{S^{\sigma \perp}}^w}u\big\|_{L^2(\rr^{n'})}\\
\leq C^m(m!)^{\frac{2k_0+1}{2}}\Big(\prod_{j=1}^m|X_j'|\Big)e^{-\omega_0t}\|u\|_{L^2(\rr^{n'})},
\end{multline}
with 
$$\omega_0=\sum_{\substack{\lambda \in \sigma(F|_{S^{\sigma \perp}}) \\
-i \lambda \in \cc_+}}r_{\lambda}\textrm{Re}(-i\lambda)>0, \qquad \cc_+=\{z \in \cc : \textrm{Re }z>0\},$$
where $r_{\lambda}$ stands for the dimension of the space of generalized eigenvectors of $F|_{S^{\sigma \perp}}$ associated to the eigenvalue $\lambda \in \sigma(F|_{S^{\sigma \perp}})$ as this dimension is the same for 
$$\chi|_{S^{\sigma \perp}} \circ F|_{S^{\sigma \perp}} \circ (\chi|_{S^{\sigma \perp}})^{-1}.$$ 
Setting
$$(x'',\xi'') \mapsto (\textrm{Im }\tilde{q})|_{S}(x'',\xi'')=\big((\textrm{Im } q)|_{S} \circ \chi^{-1}\big)(0',x'';0',\xi'')$$
and
$$(x,\xi) \mapsto \tilde{q}(x,\xi)=(q \circ \chi^{-1})(x,\xi),$$
we deduce from the commutation of the two differential operators $\tilde{q}|_{S^{\sigma \perp}}^w$ and $(\textrm{Im}\tilde{q})|_{S}^w$ that 
\begin{multline*}
(\langle x_1', x' \rangle+\langle \xi_1',D_{x'} \rangle)\ ...\ (\langle x_m', x' \rangle+\langle \xi_m',D_{x'} \rangle) e^{-t\tilde{q}^w}\\
=e^{-it(\textrm{Im}\tilde{q})|_{S}^w}(\langle x_1', x' \rangle+\langle \xi_1',D_{x'} \rangle)\ ...\ (\langle x_m', x' \rangle+\langle \xi_m',D_{x'} \rangle) e^{-t\tilde{q}|_{S^{\sigma \perp}}^w}.
\end{multline*}
By using that $(e^{-it(\textrm{Im}\tilde{q})|_{S}^w})_{t \geq 0}$ is a contraction semigroup on $L^2$ as the quadratic symbol $i(\textrm{Im}\tilde{q})|_{S}$ of its generator has a non-negative real part\footnote{It is actually a unitary group as $(\textrm{Im}\tilde{q})|_{S}^w$ is a selfadjoint operator on $L^2$.}, we deduce from (\ref{gh1}) and (\ref{gh2}) that 
for all $m \geq 1$, $X_1'=(x_1',\xi_1')  \in \rr^{2n'}$, ..., $X_m'=(x_m',\xi_m')  \in \rr^{2n'}$, $u \in L^2(\rr^{n})$,
\begin{multline}\label{gh1b}
\forall 0<t \leq 1, \quad \big\|(\langle x_1', x' \rangle+\langle \xi_1',D_{x'} \rangle)\ ...\ (\langle x_m', x' \rangle+\langle \xi_m',D_{x'} \rangle) e^{-t\tilde{q}^w}u\big\|_{L^2(\rr^{n})}\\
\leq C^{m}(m!)^{\frac{2k_0+1}{2}}\Big(\prod_{j=1}^m|X_j'|\Big)t^{-\frac{(2k_0+1)m}{2}}\|u\|_{L^2(\rr^{n})},
\end{multline}
\begin{multline}\label{gh2b}
\forall t \geq 1, \quad \big\|(\langle x_1', x' \rangle+\langle \xi_1',D_{x'} \rangle)\ ...\ (\langle x_m', x' \rangle+\langle \xi_m',D_{x'} \rangle) e^{-t\tilde{q}^w}u\big\|_{L^2(\rr^{n})}\\
\leq C^m(m!)^{\frac{2k_0+1}{2}}\Big(\prod_{j=1}^m|X_j'|\Big)e^{-\omega_0t}\|u\|_{L^2(\rr^{n})}.
\end{multline}
By using the symplectic invariance (\ref{inf3_1}), we notice that 
$$e^{-t\tilde{q}^w}=\mathcal{T}^{-1}e^{-tq^w}\mathcal{T}$$
and 
\begin{multline*}
\langle x_k', x' \rangle+\langle \xi_k',D_{x'} \rangle=\langle (x_k',0'';\xi_k',0''), (x,D_x) \rangle=\mathcal{T}^{-1}\langle (x_k',0'';\xi_k',0''), \chi^w(x,D_x) \rangle\mathcal{T}\\
=\mathcal{T}^{-1}\langle \chi(\tilde{X}_k), \chi^w(x,D_x) \rangle\mathcal{T},
\end{multline*}
with $\tilde{X}_k=\chi^{-1}(x_k',0'';\xi_k',0'') \in S^{\sigma \perp}$, it follows from (\ref{gh1b}) and (\ref{gh2b}) that there exists a positive constant $C_0>1$ for all $m \geq 1$, $X_1 \in S^{\sigma \perp}$, ..., $X_m \in S^{\sigma \perp}$, $u \in L^2(\rr^{n})$,
\begin{multline}\label{gh1br}
\forall 0<t \leq 1, \quad \big\|\langle \chi(X_1), \chi^w(x,D_x) \rangle\ ...\ \langle \chi(X_m), \chi^w(x,D_x) \rangle e^{-tq^w}u\big\|_{L^2(\rr^{n})}\\
\leq C_0^{m}(m!)^{\frac{2k_0+1}{2}}\Big(\prod_{j=1}^m|X_j|\Big)t^{-\frac{(2k_0+1)m}{2}}\|u\|_{L^2(\rr^{n})},
\end{multline}
\begin{multline}\label{gh2br}
\forall t \geq 1, \quad \big\|\langle \chi(X_1), \chi^w(x,D_x) \rangle\ ...\ \langle \chi(X_m), \chi^w(x,D_x) \rangle e^{-tq^w}u\big\|_{L^2(\rr^{n})}\\
\leq C_0^m(m!)^{\frac{2k_0+1}{2}}\Big(\prod_{j=1}^m|X_j|\Big)e^{-\omega_0t}\|u\|_{L^2(\rr^{n})}.
\end{multline}
This ends the proof of Corollary~\ref{thj_symp}.

\section{Some applications}\label{applications}

In this section, we provide some applications of the above results about general quadratic operators to the specific study of degenerate hypoelliptic Ornstein-Uhlenbeck operators and  degenerate hypoelliptic Fokker-Planck operators. We show in particular how our results allow to recover some properties of degenerate hypoelliptic Ornstein-Uhlenbeck operators proven by Farkas and Lunardi in the work~\cite{lunardi1}, and how they relate to the recent results of Arnold and Erb~\cite{anton} on degenerate hypoelliptic Fokker-Planck operators with linear drift.

\subsection{Applications to degenerate hypoelliptic Ornstein-Uhlenbeck operators}\label{orn}
We consider Ornstein-Uhlenbeck operators
\begin{equation}\label{jen0}
P=\frac{1}{2}\sum_{i,j=1}^{n}q_{i,j}\partial_{x_i,x_j}^2+\sum_{i,j=1}^nb_{i,j}x_j\partial_{x_i}=\frac{1}{2}\textrm{Tr}(Q\nabla_x^2)+\langle Bx,\nabla_x\rangle, \quad x \in \rr^n,
\end{equation}
where $Q=(q_{i,j})_{1 \leq i,j \leq n}$ and $B=(b_{i,j})_{1 \leq i,j \leq n}$ are real $n \times n$-matrices, with $Q$ symmetric positive semidefinite. We denote $\langle A,B \rangle$ and $|A|^2$ the scalar operators
\begin{equation}\label{not}
\langle A,B\rangle=\sum_{j=1}^nA_jB_j, \quad |A|^2=\langle A,A\rangle=\sum_{j=1}^nA_j^2,
\end{equation}
when $A=(A_1,...,A_n)$ and $B=(B_1,...,B_n)$ are vector-valued operators. Notice that $\langle A,B\rangle\neq\langle B,A\rangle$ in general, since e.g., 
$\langle \nabla_x,Bx\rangle=\langle Bx,\nabla_x\rangle +\textrm{Tr}(B).$

We study degenerate hypoelliptic Ornstein-Uhlenbeck operators when the symmetric matrix $Q$ is possibly not positive definite. 
These degenerate operators have been studied in the recent works~\cite{lanco,farkas,lunardi1,lanco2,Lorenzi,Metafune_al2002,OPPS2}. We recall from these works that 
the assumption of hypoellipticity is classically characterized by the following equivalent assertions:

\medskip
 
\begin{itemize}
\item[$(i)$] The Ornstein-Uhlenbeck operator $P$ is hypoelliptic
\item[$(ii)$] The symmetric positive semidefinite matrices
\begin{equation}\label{pav0}
Q_t=\int_0^{t}e^{sB}Qe^{sB^T}ds,
\end{equation}
with $B^T$ the transpose matrix of $B$, are nonsingular for some (equivalently, for all) $t>0$, i.e. $\det Q_t>0$
\item[$(iii)$] The Kalman rank condition holds: 
\begin{equation}\label{kal1}
\textrm{Rank}[B|Q^{\frac{1}{2}}]=n,
\end{equation} 
where 
$$[B|Q^{\frac{1}{2}}]=[Q^{\frac{1}{2}},BQ^{\frac{1}{2}},\dots, B^{n-1}Q^{\frac{1}{2}}],$$ 
is the $n\times n^2$ matrix obtained by writing consecutively the columns of the matrices $B^jQ^{\frac{1}{2}}$, with $Q^{\frac{1}{2}}$ the symmetric positive semidefinite matrix given by the square root of $Q$
\item[$(iv)$] The H\"ormander condition holds:
$$\forall x \in \rr^n, \quad \textrm{Rank } \mathcal{L}(X_1,X_2,...,X_n,Y_0)(x)=n,$$
with 
$$Y_0=\langle Bx,\nabla_x\rangle, \quad X_i=\sum_{j=1}^nq_{i,j}\partial_{x_j}, \quad i=1,...,n,$$
where $ \mathcal{L}(X_1,X_2,...,X_n,Y_0)(x)$ denotes the Lie algebra generated by the vector fields $X_1$, $X_2$, ..., $X_n$, $Y_0$, at point $x \in \rr^n$ 
 \end{itemize}

\medskip

\noindent
When the Ornstein-Uhlenbeck operator is hypoelliptic, that is, when one (equivalently, all) of the above conditions holds, the associated Markov semigroup $(T(t))_{t \geq 0}$ has the following explicit representation
\begin{equation}\label{gen}
(T(t)f)(x)=\frac{1}{(2\pi)^{\frac{n}{2}}\sqrt{\det Q_t}}\int_{\rr^n}e^{-\frac{1}{2}\langle Q_t^{-1}y,y\rangle}f(e^{tB}x-y)dy, \quad t>0.
\end{equation}
This formula is due to Kolmogorov~\cite{kolmo}. On the other hand, the existence of an invariant measure $\mu$ for the Markov semigroup $(T(t))_{t \geq 0}$, that is, a probability measure on $\rr^n$ verifying 
$$\forall t \geq 0, \forall f \in C_b(\rr^n), \quad \int_{\rr^n}(T(t)f)(x)d\mu(x)=\int_{\rr^n}f(x)d\mu(x),$$
where $C_b(\rr^n)$ stands for the space of continuous and bounded functions on $\rr^n$, is known to be equivalent~\cite{DaPrato} (Section~11.2.3)
to the following localization of the spectrum of~$B$,
\begin{equation}\label{pav1}
\sigma(B) \subset \mathbb{C}_-=\{z \in \mathbb{C} : \textrm{Re }z<0\}.
\end{equation}
When this condition holds, the invariant measure is unique and is given by $d\mu(x)=\rho(x)dx$, where the density with respect to the Lebesgue measure is
\begin{equation}\label{pav2}
\rho(x)=\frac{1}{(2\pi)^{\frac{n}{2}}\sqrt{\det Q_{\infty}}}e^{-\frac{1}{2}\langle Q_{\infty}^{-1}x,x\rangle},
\end{equation}
with 
\begin{equation}\label{pav3}
Q_{\infty}=\int_0^{+\infty}e^{sB}Qe^{sB^T}ds.
\end{equation}

In the work~\cite{lunardi1} (Proposition~2), Farkas and Lunardi obtain sharp $L^2(\rr^n,d\mu)$-estimates for the spatial derivatives of the semigroup $T(t)f$ for short times $0<t <1$. We begin by recalling this result:

Let $P$ be the Ornstein-Uhlenbeck operator defined in (\ref{jen0}). We assume that $P$ is hypoelliptic and that it admits the invariant measure $d\mu(x)=\rho(x)dx$. 
We consider the Ornstein-Uhlenbeck operator acting on the space $L^2_{\mu}=L^2(\rr^n,d\mu)$, equipped with the domain
\begin{equation}\label{xc5}
D(P)=\{u \in L_{\mu}^2 : Pu \in L_{\mu}^2\}.
\end{equation}
We define $(\mathcal{V}_k)_{k \geq 0}$ the vector subspaces
\begin{equation}\label{jen2.1}
\mathcal{V}_k=\big(\textrm{Ran}(Q^{\frac{1}{2}})+\textrm{Ran}(BQ^{\frac{1}{2}})+...+\textrm{Ran}(B^kQ^{\frac{1}{2}})\big) \cap \rr^n \subset \rr^n, \quad k \geq 0,
\end{equation} 
where the notation $\textrm{Ran}$ denotes the range. The Kalman rank condition
$$\textrm{Rank}[Q^{\frac{1}{2}},BQ^{\frac{1}{2}},\dots, B^{n-1}Q^{\frac{1}{2}}]=n,$$
allows one to consider $0 \leq k_0 \leq n-1$ the smallest integer satisfying
\begin{equation}\label{kal2}
\textrm{Rank}[Q^{\frac{1}{2}},BQ^{\frac{1}{2}},\dots, B^{k_0}Q^{\frac{1}{2}}]=n.
\end{equation}
We observe from (\ref{jen2.1}) and (\ref{kal2}) that the following strict inclusions  
\begin{equation}\label{byebye}
\mathcal{V}_0 \subsetneq \mathcal{V}_1 \subsetneq ... \subsetneq \mathcal{V}_{k_0}=\rr^n,
\end{equation}
hold. The work by Lanconelli and Polidoro~\cite{lanco2} shows that a fan orthonormal basis
\begin{equation}\label{ui0}
\mathcal{B}=(e_1,...,e_n), \quad  \mathcal{V}_k=\textrm{Span}\{e_j : 1 \leq j \leq \textrm{dim }\mathcal{V}_k\}, \quad 0 \leq k \leq k_0,
\end{equation} 
for the subspaces $\mathcal{V}_0,...,\mathcal{V}_{k_0}$ may be chosen so that the Ornstein-Uhlenbeck operator writes in these new coordinates as 
\begin{equation}\label{jen1}
\tilde{P}=\sum_{i,j=1}^{p_0}\tilde{a}_{i,j}\partial_{x_i,x_j}^2+\sum_{i,j=1}^nc_{i,j}x_j\partial_{x_i},
\end{equation}
with $1 \leq p_0 \leq n$, where $\tilde{A}=(\tilde{a}_{i,j})_{1 \leq i,j \leq p_0} \in M_{p_0}(\rr)$ is symmetric positive definite and $C=(c_{i,j})_{1 \leq i,j \leq n}$ has the block structure
$$C=\left( \begin{array}{ccccc}
* & * & \cdots & * & * \\
C_1 & * & \cdots & * & * \\
0 & C_2 & \ddots & * & * \\
\vdots & \ddots & \ddots & * & * \\
0 & \cdots & 0 & C_{k_0} & * \\
\end{array} \right)\in M_{n}(\rr),$$
where $C_j$ is a $p_j \times p_{j-1}$ block with rank $p_j$ for all $j=1,...,k_0$ satisfying
$$p_0 \geq p_1 \geq p_2 \geq ... \geq p_{k_0} \geq 1,$$
where 
$$p_0=\textrm{dim }\mathcal{V}_0, \quad p_i=\textrm{dim }\mathcal{V}_i-\textrm{dim }\mathcal{V}_{i-1}, \quad 1 \leq i \leq k_0.$$
We consider the sets of indices
\begin{equation}\label{ui2}
I_0=\{j \in \mathbb{N} : 1\leq j \leq \textrm{dim }\mathcal{V}_0\}, \quad I_k=\{j \in \mathbb{N} : \textrm{dim }\mathcal{V}_{k-1}+1\leq j \leq \textrm{dim }\mathcal{V}_k\},
\end{equation}
for $1 \leq k \leq k_0$, providing the partition
$$I_0 \sqcup I_1 \sqcup .... \sqcup I_{k_0}=\{1,...,n\}.$$
By denoting $\tilde{\partial}_j$ the partial derivative in the direction $e_j$ given by the basis (\ref{ui0}), the result of~\cite{lunardi1} (Proposition~2) shows that 
\begin{multline}\label{we10}
\forall N \geq 1, \exists c_N>0, \forall f \in L^2(\rr^n,d\mu), \forall j_1 \in I_{\nu_1}, ....,\forall j_N \in I_{\nu_N}, \forall 0<t<1, \\ \|\tilde{\partial}_{j_1}...\tilde{\partial}_{j_N}T(t)f\|_{L_{\mu}^2} \leq \frac{c_N}{t^{\frac{N}{2}+\nu_1+...+\nu_N}}\|f\|_{L^2_{\mu}}, 
\end{multline}
where the sets of indices $I_{\nu_1}$, ..., $I_{\nu_N}$, with $0 \leq \nu_1,...,\nu_N \leq k_0$, are defined in (\ref{ui2}) and where $\|\cdot\|_{L^2_{\mu}}$ denotes the $L^2(\rr^n,d\mu)$-norm. This result well points out how the short-time asymptotics of the regularizing effect depend on the directions of the derivatives. 
In the case when $N=1$, the estimates may rephrased as follows.

According to (\ref{byebye}), we can consider for any $\xi_0 \in \rr^n$ the smallest integer $0 \leq k \leq k_0$ 
such that $\xi_0 \in \mathcal{V}_k$. This integer is denoted $k_{\xi_0}$ and is called the frequency index of the point $\xi_0 \in \rr^n$ with respect to the Ornstein-Uhlenbeck operator $P$.

The estimate (\ref{we10}) reads as 
\begin{equation}\label{we11}
\exists c>0, \forall f \in L^2(\rr^n,d\mu), \forall 0<t<1, \quad \|\langle \xi_0,D_x\rangle T(t)f\|_{L_{\mu}^2} \leq \frac{c}{t^{\frac{2k_{\xi_0}+1}{2}}}\|f\|_{L^2_{\mu}}. 
\end{equation}
We explain in the following how Theorem~\ref{th} allows us to recover the estimates (\ref{we11}) and to derive further decay estimates for $T(t)f$.

To that end, we may associate to the operator $P$ acting on $L^2_{\mu}=L^2(\rr^n,d\mu)$, the quadratic operator $\mathscr{L}$ acting on $L^2=L^2(\rr^n,dx)$,
\begin{equation}\label{pav4.5}
\mathscr{L}u=-\sqrt{\rho}P\big((\sqrt{\rho})^{-1}u\big)-\frac{1}{2}\textrm{Tr}(B)u.
\end{equation}
We notice that the localization of the spectrum (\ref{pav1}) implies that 
\begin{equation}\label{conti5}
\textrm{Tr}(B)<0,
\end{equation}
since $B \in M_n(\rr)$. 
Recalling the notation (\ref{not}), a direct computation led in the work~\cite{OPPS2} (see (3.7) in Section~3.1) shows that
\begin{multline}\label{pav5}
\mathscr{L}=-\frac{1}{2}|Q^{\frac{1}{2}}\nabla_x|^2+\frac{1}{8}|Q^{\frac{1}{2}}Q_{\infty}^{-1}x|^2-\Big\langle\Big(\frac{1}{2}QQ_{\infty}^{-1}+B\Big)x,\nabla_x\Big\rangle\\
=\frac{1}{2}|Q^{\frac{1}{2}}D_x|^2+\frac{1}{8}|Q^{\frac{1}{2}}Q_{\infty}^{-1}x|^2-i\Big\langle\Big(\frac{1}{2}QQ_{\infty}^{-1}+B\Big)x,D_x\Big\rangle, 
\end{multline} 
with $D_x=i^{-1}\nabla_x$, where $Q_{\infty}$ is the symmetric positive definite matrix (\ref{pav3}).
The operator $\mathscr{L}$ may be considered as a pseudodifferential operator
\begin{equation}\label{3}
\mathscr{L}u=q^w(x,D_x)u(x) =\frac{1}{(2\pi)^n}\int_{\R^{2n}}{e^{i(x-y) \cdot \xi}q\Big(\frac{x+y}{2},\xi\Big)u(y)dyd\xi},
\end{equation}
defined by the Weyl quantization of the quadratic symbol
\begin{equation}\label{pav6}
q(x,\xi)=\frac{1}{2}|Q^{\frac{1}{2}}\xi|^2+\frac{1}{8}|Q^{\frac{1}{2}}Q_{\infty}^{-1}x|^2-i\Big\langle\Big(\frac{1}{2}QQ_{\infty}^{-1}+B\Big)x,\xi\Big\rangle, \quad (x,\xi) \in \rr^{2n},
\end{equation}
with a non-negative real part $\textrm{Re }q \geq 0$.
We check in~\cite{OPPS2}, see (3.16) in Section~3.2, that the Hamilton map of the quadratic form (\ref{pav6}) is given by
\begin{equation}\label{F}
F=\lv\begin{array}{cc}
 -\frac{i}{4}(QQ_{\infty}^{-1}+2B)&\frac{1}{2}Q \\
 -\frac{1}{8}Q_{\infty}^{-1}QQ_{\infty}^{-1}  & \frac{i}{4}(QQ_{\infty}^{-1}+2B)^T       
\end{array}\rv.
\end{equation}
On the other hand, we prove in~\cite{OPPS2}, see formulas (3.22), (3.23) and (3.24), that the singular space of the quadratic operator $\mathscr{L}$ is zero 
\begin{equation}\label{sd1}
S=\{0\}. 
\end{equation}
More precisely, we show in~\cite{OPPS2} that the smallest integer $0 \leq k_0 \leq 2n-1$ satisfying 
\begin{equation}\label{pav30}
\Big(\bigcap_{j=0}^{k_0}\textrm{Ker}
\big[\textrm{Re }F(\textrm{Im }F)^j \big]\Big)\cap \rr^{2n}=\{0\},
\end{equation}
corresponds exactly to the smallest integer $0 \leq k_0 \leq n-1$ satisfying
\begin{equation}\label{pav31}
\textrm{Rank}[Q^{\frac{1}{2}},BQ^{\frac{1}{2}},\dots, B^{k_0}Q^{\frac{1}{2}}]=n.
\end{equation}
This property justifies a posteriori the identical notation chosen for the integers $k_0$ defined in (\ref{e2}) and (\ref{kal2}).
We deduce from~\cite{OPPS2} (Lemma~3.1) and (\ref{we1}) that for all $0 \leq k \leq k_0$,
\begin{multline}\label{upt1}
V_k^{\perp}=\Big(\bigcap_{j=0}^{k}\textrm{Ker}\big[\textrm{Re }F(\textrm{Im }F)^j \big]\Big)\cap \rr^{2n}\\
=\big\{(x,\xi) \in \rr^{2n} : \forall 0 \leq j \leq k,\ QQ_{\infty}^{-1}B^jx=0,\ Q(B^T)^j\xi=0\big\},
\end{multline}
that is
\begin{equation}\label{we20}
V_k=\Big[\Big(\Big(\bigcap_{j=0}^{k}\textrm{Ker}(QQ_{\infty}^{-1}B^j)\Big)\cap \rr^{n}\Big) \times \Big(\Big(\bigcap_{j=0}^{k}\textrm{Ker}\big(Q(B^T)^j\big)\Big)\cap \rr^{n}\Big)\Big]^{\perp}.
\end{equation}
We observe from (\ref{jen2.1}) and (\ref{we20}) that 
\begin{multline}\label{we21}
(0,\xi_0) \in V_k \Leftrightarrow \xi_0 \in \Big(\Big(\bigcap_{j=0}^{k}\textrm{Ker}\big(Q(B^T)^j\big)\Big)\cap \rr^{n}\Big)^{\perp} \\
\Leftrightarrow \xi_0 \in\Big(\Big(\bigcap_{j=0}^{k}\textrm{Ker}\big(Q^{\frac{1}{2}}(B^T)^j\big)\Big)\cap \rr^{n}\Big)^{\perp}
\Leftrightarrow \xi_0 \in \mathcal{V}_k=\big(\textrm{Ran}(Q^{\frac{1}{2}})+...+\textrm{Ran}(B^kQ^{\frac{1}{2}})\big) \cap \rr^n,
\end{multline}
and
\begin{multline}\label{we22}
(x_0,0) \in V_k \Leftrightarrow x_0 \in \Big(\Big(\bigcap_{j=0}^{k}\textrm{Ker}(QQ_{\infty}^{-1}B^j)\Big)\cap \rr^{n}\Big)^{\perp}\\
\Leftrightarrow x_0 \in \Big(\Big(\bigcap_{j=0}^{k}\textrm{Ker}(Q^{\frac{1}{2}}Q_{\infty}^{-1}B^j)\Big)\cap \rr^{n}\Big)^{\perp}
\Leftrightarrow x_0 \in \tilde{\mathcal{V}}_k,
\end{multline}
with 
\begin{equation}\label{we23}
\tilde{\mathcal{V}}_k=\big(\textrm{Ran}(Q_{\infty}^{-1}Q^{\frac{1}{2}})+\textrm{Ran}(B^TQ_{\infty}^{-1}Q^{\frac{1}{2}})+...+\textrm{Ran}((B^T)^kQ_{\infty}^{-1}Q^{\frac{1}{2}})\big) \cap \rr^n, \quad k \geq 0,
\end{equation}
where the orthogonality is taken with respect to the Euclidean structure on $\rr^n$. The following lemma makes explicit the link between the two families of vector subspaces $(\mathcal{V}_k)_{0 \leq k \leq k_0}$  and $(\tilde{\mathcal{V}}_k)_{0 \leq k \leq k_0}$:

\medskip

\begin{lemma}\label{li1}
We have for all $0 \leq k \leq k_0$,
$$\tilde{\mathcal{V}}_k=Q_{\infty}^{-1}\mathcal{V}_k,$$
where $Q_{\infty}$ stands for the positive definite symmetric matrix (\ref{pav3}). 
\end{lemma}

\medskip

\begin{proof}
We begin by noticing from (\ref{pav0}) and (\ref{pav3}) that for all $t \geq 0$,
\begin{equation}\label{pav0.1}
Q_{\infty}=Q_t+e^{tB}Q_{\infty}e^{tB^T}.
\end{equation}
We deduce from (\ref{pav0.1}) the steady state variance equation
\begin{equation}\label{pav10}
\frac{d}{dt}(Q_t+e^{tB}Q_{\infty}e^{tB^T})|_{t=0}=Q+BQ_{\infty}+Q_{\infty}B^T=0.
\end{equation}
It follows from (\ref{we22}) and (\ref{pav10}) that 
\begin{multline}\label{wq1}
\tilde{\mathcal{V}}_k^{\perp}=\Big(\bigcap_{j=0}^{k}\textrm{Ker}(Q^{\frac{1}{2}}Q_{\infty}^{-1}B^j)\Big)\cap \rr^{n}
=\Big(\bigcap_{j=0}^{k}\textrm{Ker}\big(Q^{\frac{1}{2}}(Q_{\infty}^{-1}BQ_{\infty})^jQ_{\infty}^{-1}\big)\Big)\cap \rr^{n}\\
=\Big(\bigcap_{j=0}^{k}\textrm{Ker}\big(Q^{\frac{1}{2}}(Q_{\infty}^{-1}Q+ B^T)^jQ_{\infty}^{-1}\big)\Big)\cap \rr^{n}.
\end{multline}
Next we check by induction on $k$ that for all $k \geq 0$,
\begin{equation}\label{wq2}
\Big(\bigcap_{j=0}^{k}\textrm{Ker}\big(Q^{\frac{1}{2}}(Q_{\infty}^{-1}Q+ B^T)^jQ_{\infty}^{-1}\big)\Big)\cap \rr^{n}=\Big(\bigcap_{j=0}^{k}\textrm{Ker}\big(Q^{\frac{1}{2}}(B^T)^jQ_{\infty}^{-1}\big)\Big)\cap \rr^{n}.
\end{equation}
Indeed, the formula (\ref{wq2}) trivially holds true for $k=0$. By expanding the products, we observe that 
\begin{equation}\label{wq3}
Q^{\frac{1}{2}}(Q_{\infty}^{-1}Q+ B^T)^{k+1}Q_{\infty}^{-1}x=Q^{\frac{1}{2}}(B^T)^{k+1}Q_{\infty}^{-1}x,
\end{equation}
when 
$$\forall 0 \leq j \leq k, \quad Q^{\frac{1}{2}}(B^T)^jQ_{\infty}^{-1}x=0.$$
We deduce from (\ref{wq3}) that the formula (\ref{wq2}) holds true at the rank $k+1$ if it holds true at the rank $k$. According to (\ref{jen2.1}), (\ref{wq1}) and (\ref{wq2}), we have for all $0 \leq k \leq k_0$,
\begin{multline}\label{wq4}
\tilde{\mathcal{V}}_k=\Big(\Big(\bigcap_{j=0}^{k}\textrm{Ker}\big(Q^{\frac{1}{2}}(B^T)^jQ_{\infty}^{-1}\big)\Big)\cap \rr^{n}\Big)^{\perp}\\
=\big(\textrm{Ran}(Q_{\infty}^{-1}Q^{\frac{1}{2}})+\textrm{Ran}(Q_{\infty}^{-1}BQ^{\frac{1}{2}})+...+\textrm{Ran}(Q_{\infty}^{-1}B^kQ^{\frac{1}{2}})\big) \cap \rr^n=Q_{\infty}^{-1}\mathcal{V}_k,
\end{multline}
since $Q_{\infty}^{-1}$ is symmetric. This ends the proof of Lemma~\ref{li1}.
\end{proof}

We deduce from Lemma~\ref{li1} and (\ref{byebye}) that 
\begin{equation}\label{byebye2}
\tilde{\mathcal{V}}_0 \subsetneq \tilde{\mathcal{V}}_1 \subsetneq ... \subsetneq \tilde{\mathcal{V}}_{k_0}=\rr^n.
\end{equation}
According to (\ref{byebye2}), we can consider for any $x_0 \in \rr^n$ the smallest integer $0 \leq k \leq k_0$ 
such that $x_0 \in \tilde{\mathcal{V}}_k$. This integer is denoted $\tilde{k}_{x_0}$ and is called the space index of the point $x_0 \in \rr^n$ with respect to the Ornstein-Uhlenbeck operator $P$.

The following result shows that Theorem~\ref{th} allows one to recover the estimates (\ref{we11}) proven in~\cite{lunardi1} and to derive new decay estimates:

\medskip

\begin{corollary}\label{coro1}
Let 
$$P=\frac{1}{2}\emph{\textrm{Tr}}(Q\nabla_x^2)+\langle Bx,\nabla_x\rangle,\quad x \in \rr^n,$$ 
be a hypoelliptic Ornstein-Uhlenbeck operator, which admits the invariant measure $d\mu(x)=\rho(x) dx$. 
Then, the global solution $v(t)=e^{tP}v_0$ to the Cauchy problem
\begin{equation}\label{p2}
\left\lbrace\begin{array}{c}
\partial_tv=Pv,\\
v|_{t=0}=v_0 \in L^2_{\mu},
\end{array}\right.
\end{equation}
with $L^2_{\mu}=L^2(\rr^n,d\mu)$, satisfies: $\exists C>0$, $\forall x_0 \in \rr^{n}$, $\forall \xi_0 \in \rr^{n}$, $\forall v_0 \in L^2_{\mu}$,    
$$\forall 0<t \leq 1, \quad \|\langle \xi_0,D_x \rangle e^{tP}v_0\|_{L^2_{\mu}}\leq C|\xi_0| t^{-(2k_{\xi_0}+1)/2}\|v_0\|_{L^2_{\mu}},$$
$$\forall 0<t \leq 1, \quad \|\langle x_0,x \rangle e^{tP}v_0\|_{L^2_{\mu}}\leq C|x_0| t^{-(2\tilde{k}_{x_0}+1)/2}\|v_0\|_{L^2_{\mu}},$$
$$\forall t \geq 1, \quad \|\langle \xi_0,D_x \rangle e^{tP}v_0\|_{L^2_{\mu}}\leq C|\xi_0|\|v_0\|_{L^2_{\mu}},$$
$$\forall t \geq 1, \quad \|\langle x_0,x \rangle e^{tP}v_0\|_{L^2_{\mu}}\leq C|x_0|\|v_0\|_{L^2_{\mu}},$$
with $\|\cdot\|_{L^2_{\mu}}$ the $L^2(\rr^n,d\mu)$-norm, where $0 \leq k_{\xi_0} \leq k_0$ denotes the frequency index of $\xi_0 \in \rr^n$ with respect to $P$, and where $0 \leq \tilde{k}_{x_0} \leq k_0$ denotes the space index of $x_0 \in \rr^n$ with respect to $P$.
\end{corollary}

\medskip

\begin{proof}
By using from (\ref{pav6}) and (\ref{sd1}) that the quadratic operator
$$\mathscr{L}=q^w(x,D_x)=\frac{1}{2}|Q^{\frac{1}{2}}D_x|^2+\frac{1}{8}|Q^{\frac{1}{2}}Q_{\infty}^{-1}x|^2-i\Big\langle\Big(\frac{1}{2}QQ_{\infty}^{-1}+B\Big)x,D_x\Big\rangle,$$
has a Weyl symbol with a zero singular space $S=\{0\}$ and a non-negative real part $\textrm{Re }q \geq 0$,
we can deduce from~\cite{HPS} (Theorem~1.2.1) that the evolution equation associated to the accretive operator $\mathscr{L}$, 
$$\left\lbrace\begin{array}{c}
\partial_tu+\mathscr{L}u=0,\\
u|_{t=0}=u_0 \in L^2(\rr^n,dx),
\end{array}\right.$$
is smoothing in the Schwartz space $\mathscr{S}(\rr^n)$ for any positive time $t>0$, that is,
\begin{equation}\label{p1}
\forall t>0, \quad u(t)=e^{-t\mathscr{L}}u_0 \in \mathscr{S}(\rr^n),
\end{equation}
where $(e^{-t\mathscr{L}})_{t \geq 0}$ denotes the contraction semigroup generated by $\mathscr{L}$. It follows from (\ref{pav4.5}) that the solution to the evolution equation (\ref{p2}) is given by
\begin{equation}\label{bn1}
v(t)=e^{tP}v_0=(\sqrt{\rho})^{-1}e^{-t(\mathscr{L}+\frac{1}{2}\textrm{Tr}(B))}(\sqrt{\rho}v_0)=e^{-\frac{t}{2}\textrm{Tr}(B)}(\sqrt{\rho})^{-1}e^{-t\mathscr{L}}(\sqrt{\rho}v_0),
\end{equation}
for all $t \geq 0$.
By using from (\ref{pav2}) that
\begin{equation}\label{dfg5}
\sqrt{\rho}\partial_{x_i}\big((\sqrt{\rho})^{-1}u\big)=e^{-\frac{1}{4}\langle Q_{\infty}^{-1}x,x\rangle}\partial_{x_i}(e^{\frac{1}{4}\langle Q_{\infty}^{-1}x,x\rangle}u)=\Big(\partial_{x_i}+\frac{1}{2}(Q_{\infty}^{-1}x)_i\Big)u,
\end{equation}
where $(Q_{\infty}^{-1}x)_i$ denotes the $i^{\textrm{th}}$ coordinate,
since the matrix $Q_{\infty}^{-1}$ is symmetric, we notice that for all $t > 0$,
\begin{equation}\label{dfg2}
\langle \xi_0,D_x \rangle e^{tP}v_0=e^{-\frac{t}{2}\textrm{Tr}(B)}(\sqrt{\rho})^{-1}
\Big(\langle \xi_0,D_x \rangle-\frac{i}{2}\langle \xi_0,Q_{\infty}^{-1}x \rangle \Big)e^{-t\mathscr{L}}(\sqrt{\rho}v_0).
\end{equation}
By using that the mappings 
\begin{equation}\label{1w1}
\begin{array}{cc}
T :  L^2_{\mu} & \rightarrow  L^2\\
\ \ \ u &  \mapsto  \sqrt{\rho}u
\end{array}, \qquad \begin{array}{cc}
T^{-1} :  L^2 & \rightarrow L^2_{\mu}\\
\quad \ \ \ u & \mapsto \sqrt{\rho}^{-1}u
\end{array},
\end{equation}
are isometric,
\begin{multline}\label{dfg1}
\|\langle \xi_0,D_x \rangle e^{tP}v_0\|_{L^2_{\mu}}\\
\leq e^{-\frac{t}{2}\textrm{Tr}(B)}\big(\|\langle \xi_0,D_x \rangle e^{-t\mathscr{L}}(\sqrt{\rho}v_0)\|_{L^2}
+\|\langle Q_{\infty}^{-1}\xi_0,x \rangle e^{-t\mathscr{L}}(\sqrt{\rho}v_0)\|_{L^2}\big).
\end{multline}
According to (\ref{pav4.5}) and (\ref{conti5}), we observe from~\cite{OPPS2} (Proposition~2.2 and formula (2.6)) that 
\begin{equation}\label{conti6}
\omega_0=\sum_{\substack{\lambda \in \sigma(F) \\
-i \lambda \in \cc_+}}r_{\lambda}\textrm{Re}(-i\lambda)=-\frac{1}{2}\textrm{Tr}(B)>0,
\end{equation}
where $r_{\lambda}$ stands for the dimension of the space of generalized eigenvectors of the Hamilton map (\ref{F}) in $\cc^{2n}$ associated to the eigenvalue $\lambda$.

According to Lemma~\ref{li1}, (\ref{we21}) and (\ref{we22}), the property $\xi_0 \in \mathcal{V}_{k_{\xi_0}}$ implies that 
\begin{equation}\label{yy1}
(0,\xi_0) \in V_{k_{\xi_0}}, \quad (Q_{\infty}^{-1}\xi_0,0) \in V_{k_{\xi_0}},
\end{equation}
since $Q_{\infty}^{-1}\xi_0 \in \tilde{\mathcal{V}}_{k_{\xi_0}}$, where $0 \leq k_{\xi_0} \leq k_0$ denotes the frequency index of $\xi_0 \in \rr^n$ with respect to $P$. The indices with respect to the singular space of the quadratic operator $\mathscr{L}$ of both directions $(0,\xi_0)$ and $(Q_{\infty}^{-1}\xi_0,0)$, are therefore lower than $k_{\xi_0}$.
We deduce from (\ref{1w1}), (\ref{dfg1}), (\ref{conti6}), (\ref{yy1}) and Theorem~\ref{th} that there exists a positive constant $C>0$ such that for all $\xi_0 \in \rr^{n}$, $v_0 \in L^2_{\mu}$,  
$$\forall 0<t \leq 1, \quad \|\langle \xi_0,D_x \rangle e^{tP}v_0\|_{L^2_{\mu}}\leq C|\xi_0| t^{-(2k_{\xi_0}+1)/2}\|v_0\|_{L^2_{\mu}},$$
$$\forall t \geq 1, \quad \|\langle \xi_0,D_x \rangle e^{tP}v_0\|_{L^2_{\mu}}\leq C|\xi_0|\|v_0\|_{L^2_{\mu}}.$$
On the other hand, we deduce from (\ref{bn1}) that for all $t \geq 0$, $v_0 \in L^2_{\mu}$,
\begin{equation}\label{3bn1}
\langle x_0,x \rangle e^{tP}v_0=e^{-\frac{t}{2}\textrm{Tr}(B)}(\sqrt{\rho})^{-1}\langle x_0,x \rangle e^{-t\mathscr{L}}(\sqrt{\rho}v_0).
\end{equation}
We deduce from (\ref{1w1}) and (\ref{3bn1}) that for all $t > 0$, $v_0 \in L^2_{\mu}$,
\begin{equation}\label{3bn2}
\|\langle x_0,x \rangle e^{tP}v_0\|_{L^2_{\mu}}=e^{-\frac{t}{2}\textrm{Tr}(B)}\|\langle x_0,x \rangle e^{-t\mathscr{L}}(\sqrt{\rho}v_0)\|_{L^2}.
\end{equation}
We observe from (\ref{we22}) that the property $x_0 \in \tilde{\mathcal{V}}_{\tilde{k}_{x_0}}$ implies that $(x_0,0) \in V_{\tilde{k}_{x_0}}$,
where $0 \leq \tilde{k}_{x_0} \leq k_0$ denotes the space index of $x_0 \in \rr^n$ with respect to $P$.
The index with respect to the singular space of the quadratic operator $\mathscr{L}$ of the direction $(x_0,0)$ is therefore lower than $\tilde{k}_{x_0}$.
We deduce from (\ref{1w1}), (\ref{conti6}), (\ref{3bn2}) and Theorem~\ref{th} that there exists a positive constant $C>0$ such that for all $x_0 \in \rr^{n}$, $v_0 \in L^2_{\mu}$,  
$$\forall 0<t \leq 1, \quad \|\langle x_0,x \rangle e^{tP}v_0\|_{L^2_{\mu}}\leq C|x_0| t^{-(2\tilde{k}_{x_0}+1)/2}\|v_0\|_{L^2_{\mu}},$$
$$\forall t \geq 1, \quad \|\langle x_0,x \rangle e^{tP}v_0\|_{L^2_{\mu}}\leq C|x_0|\|v_0\|_{L^2_{\mu}}.$$
This ends the proof of Corollary~\ref{coro1}.
\end{proof}

We can deduce from Corollary~\ref{thj} the following result:  

\medskip

\begin{corollary}\label{coro111}
Under the assumptions of Corollary~\ref{coro1}, there exists a positive constant $C>1$ such that for all $m \geq 1$, $X_1=(x_1,\xi_1)  \in \rr^{2n}$, ..., $X_m=(x_m,\xi_m)  \in \rr^{2n}$, $v_0 \in L^2(\rr^n,d\mu)$,
\begin{multline*}
\forall 0<t \leq 1, \quad \|(\langle x_1, x \rangle+\langle \xi_1,D_x \rangle)\ ...\ (\langle x_m, x \rangle+\langle \xi_m,D_x \rangle) e^{tP}v_0\|_{L^2_{\mu}}\\
\leq C^{m}(m!)^{\frac{2k_0+1}{2}}\Big(\prod_{j=1}^m|X_j|\Big)t^{-\frac{(2k_0+1)m}{2}}\|v_0\|_{L^2_{\mu}},
\end{multline*}
\begin{multline*}
\forall t \geq 1, \quad \|(\langle x_1, x \rangle+\langle \xi_1,D_x \rangle)\ ...\ (\langle x_m, x \rangle+\langle \xi_m,D_x \rangle)e^{tP}v_0\|_{L^2_{\mu}}\\
\leq C^m(m!)^{\frac{2k_0+1}{2}}\Big(\prod_{j=1}^m|X_j|\Big)\|v_0\|_{L^2_{\mu}},
\end{multline*}
with $\|\cdot\|_{L^2_{\mu}}$ the $L^2(\rr^n,d\mu)$-norm, where $0 \leq k_0 \leq n-1$ is the smallest integer satisfying (\ref{kal2}),
and where $|\cdot|$ denotes the Euclidean norm on $\rr^{2n}$.

\end{corollary}

\begin{proof}
By resuming the proof of Corollary~\ref{coro1}, it follows from (\ref{bn1}), (\ref{dfg5}), (\ref{dfg2}) and (\ref{3bn1}) that for all $m \geq 1$, $X_1=(x_1,\xi_1)  \in \rr^{2n}$, ..., $X_m=(x_m,\xi_m)  \in \rr^{2n}$, $v_0 \in L^2(\rr^n,d\mu)$,
\begin{multline}\label{dfg3}
(\langle x_1, x \rangle+\langle \xi_1,D_x \rangle)\ ...\ (\langle x_m, x \rangle+\langle \xi_m,D_x \rangle) e^{tP}v_0=e^{-\frac{t}{2}\textrm{Tr}(B)}(\sqrt{\rho})^{-1}\\ 
\times (\langle x_1-i 2^{-1}Q_{\infty}^{-1}\xi_1,x \rangle+\langle \xi_1,D_x \rangle)\ ...\ (\langle x_m-i 2^{-1}Q_{\infty}^{-1}\xi_m,x \rangle+\langle \xi_m,D_x \rangle)e^{-t\mathscr{L}}(\sqrt{\rho}v_0).
\end{multline}
We directly deduce from Corollary~\ref{thj} that there exists a positive constant $C>1$ such that for all $m \geq 1$, $X_1=(x_1,\xi_1)  \in \rr^{2n}$, ..., $X_m=(x_m,\xi_m)  \in \rr^{2n}$, $v_0 \in L^2(\rr^n,d\mu)$,
\begin{multline*}
\forall 0<t \leq 1, \quad \big\|(\langle x_1, x \rangle+\langle \xi_1,D_x \rangle)\ ...\ (\langle x_m, x \rangle+\langle \xi_m,D_x \rangle) e^{tP}v_0\|_{L^2_{\mu}}\\
\leq C^{m}(m!)^{\frac{2k_0+1}{2}}\Big(\prod_{j=1}^m|X_j|\Big)t^{-\frac{(2k_0+1)m}{2}}\|v_0\|_{L^2_{\mu}},
\end{multline*}
\begin{multline*}
\forall t \geq 1, \quad \|(\langle x_1, x \rangle+\langle \xi_1,D_x \rangle)\ ...\ (\langle x_m, x \rangle+\langle \xi_m,D_x \rangle)e^{tP}v_0\|_{L^2_{\mu}}\\
\leq C^m(m!)^{\frac{2k_0+1}{2}}\Big(\prod_{j=1}^m|X_j|\Big)\|v_0\|_{L^2_{\mu}}.
\end{multline*}
This ends the proof of Corollary~\ref{coro111}.
\end{proof}

The upper bound in Corollary~\ref{coro111} for iterated differentiations of the semigroup generated by a degenerate hypoelliptic Ornstein-Uhlenbeck operator is not as sharp as the estimate (\ref{we10}) established by Lunardi and Farkas in~\cite{lunardi1}. Indeed, the result of Corollary~\ref{coro111} does not capture the dependence of the short-time asymptotics according to the phase space directions. However, the result of Corollary~\ref{coro111} provides new decay estimates and upper bounds for moments of any order in the short-time limit $t \to 0$.   

As an application of Corollary~\ref{th0}, we recover the result proven in~\cite{OPPS2} (Proposition~2.3). Before recalling the statement of this result, we set some notations. We recall that the 1-dimensional Hermite polynomials are defined by
$$h_n(x)=\frac{(-1)^n}{\sqrt{n!}}e^{\frac{x^2}{2}}\frac{d^n}{dx^n}e^{-\frac{x^2}{2}}, \quad n \geq 0,\ x \in \rr.$$
By using the fact that the symmetric matrix $Q_{\infty}$ is positive definite, we can introduce an orthogonal matrix $U$ such that $UQ_{\infty}U^{-1}=\textrm{diag}[\lambda_1,...,\lambda_n]$ is diagonal. We define for any multi-index $\beta=(\beta_1,...,\beta_n) \in \mathbb{N}^n$,
$$H_{\beta}(x)=\prod_{j=1}^nh_{\beta_j}\Big(\frac{(U\mathcal{T}^{-1}x)_{j}}{\sqrt{\lambda_j}}\Big), \quad x \in \rr^n,$$
where $\mathcal{T}$ denotes the invertible matrix representing the change of basis from the canonical basis of $\rr^n$ to the basis $\mathcal{B}$ defined in (\ref{ui0}). 
As eigenfunctions of the selfadjoint non-positive Ornstein-Uhlenbeck operator 
\begin{equation}\label{hj10}
AH_{\beta}=\Big(\frac{1}{2}\textrm{Tr}(Q_{\infty}\nabla_x^2)-\frac{1}{2}\langle x,\nabla_x\rangle\Big)H_{\beta}=-\frac{|\beta|}{2}H_{\beta},
\end{equation}
with $|\beta|=\beta_1+...+\beta_n$, these polynomials constitute an orthonormal basis of $L^2_{\mu}$. For any $s>0$,
the Sobolev space $H^s(\rr^n,d\mu)$ is defined as the domain of the operator $(\sqrt{I-A})^s$, that is, the set of $L^2_{\mu}$ functions satisfying
\begin{equation}\label{hj12}
\|u\|_{H^s(\rr^n,d\mu)}^2=\|(\sqrt{I-A})^su\|_{L^2_{\mu}}^2
=\sum_{\beta \in \mathbb{N}^n}\Big(1+\frac{|\beta|}{2}\Big)^{s}|(u,H_{\beta})_{L^2_{\mu}}|^2<+\infty.
\end{equation}
For $s_0,s_1,...,s_{k_0}>0$, the anisotropic Sobolev space $H^{s_0,s_1,....,s_{k_0}}(\rr^n,d\mu)$ is defined as the space of $L^2_{\mu}$ functions satisfying
$$\|u\|_{H^{s_0,s_1,....,s_{k_0}}(\rr^n,d\mu)}^2=\sum_{\beta \in \mathbb{N}^n}\sum_{k=0}^{k_0}\Big(1+\sum_{j \in I_k}\frac{\beta_j}{2}\Big)^{s_k}|(u,H_{\beta})_{L_{\mu}^2}|^2,$$ 
where the sets of indices $(I_k)_{0 \leq k \leq k_0}$ are defined in (\ref{ui2}).
The result of~\cite{lunardi1} (Theorem~8) shows that the domain of the infinitesimal generator of the Ornstein-Ulhenbeck semigroup $(T(t))_{t \geq 0}$ satisfies the following embedding into the anisotropic Sobolev space
\begin{equation}\label{ui5}
D(P) \subset H^{2,\frac{2}{3},\frac{2}{5}....,\frac{2}{2k_0+1}}(\rr^n,d\mu) \subset H^{\frac{2}{2k_0+1}}(\rr^n,d\mu).
\end{equation}
It was shown in~\cite{OPPS2} (Proposition~2.3) that Corollary~\ref{th0} allows one to recover the weaker embedding of the domain into the isotropic Sobolev space
$D(P) \subset H^{2/(2k_0+1)}(\rr^n,d\mu).$

\medskip

\begin{proposition}\label{subest}
Let 
$$P=\frac{1}{2}\emph{\textrm{Tr}}(Q\nabla_x^2)+\langle Bx,\nabla_x\rangle, \quad x \in \rr^n,$$ 
be a hypoelliptic Ornstein-Uhlenbeck operator, which admits the invariant measure $d\mu(x)=\rho(x) dx$.
Then, there exists a positive constant $C>0$ such that 
$$\forall v \in D(P), \quad \|v\|_{H^{\frac{2}{2k_0+1}}(\rr^n,d\mu)}  \leq C(\|Pv\|_{L_{\mu}^2}+\|v\|_{L_{\mu}^2}),$$
with $\|\cdot\|_{L^2_{\mu}}$ the $L^2(\rr^n,d\mu)$-norm, 
where $0 \leq k_0 \leq n-1$ is the smallest integer satisfying (\ref{kal2}) and
where the Sobolev space $H^{2/(2k_0+1)}(\rr^n,d\mu)$ is defined in (\ref{hj12}).
\end{proposition}

\medskip

\begin{proof}
See the proof of Proposition~2.3 in~\cite{OPPS2}.
\end{proof}

Before applying Theorem~\ref{th-1} to the study of degenerate hypoelliptic Ornstein-Uhlenbeck operators, we need to introduce auxiliary operators.
To that end, we define the matrices
\begin{equation}\label{qwe1}
\mathfrak{Q}_j=2^{-4j}(QQ_{\infty}^{-1}+2B)^jQ(Q_{\infty}^{-1}Q+2B^T)^j, \quad \mathfrak{B}_j=-2^{-1}\mathfrak{Q}_jQ_{\infty}^{-1}, 
\end{equation} 
and we consider the following Ornstein-Uhlenbeck operators 
\begin{equation}\label{qwe2}
\mathfrak{P}_j=\frac{1}{2}\textrm{Tr}(\mathfrak{Q}_j\nabla_x^2)+\langle \mathfrak{B}_jx,\nabla_x\rangle, \quad x \in \rr^n,
\end{equation} 
for any $0 \leq j \leq k_0$. Notice that the matrix $\mathfrak{Q}_j$ is symmetric positive semidefinite.
These Ornstein-Uhlenbeck operators are equipped with the domains
\begin{equation}\label{qwe3}
D(\mathfrak{P}_j)=\{u \in L_{\mu}^2 : \mathfrak{P}_ju \in L_{\mu}^2\}.
\end{equation}
As in (\ref{pav4.5}), we can associate to these Ornstein-Uhlenbeck operators some quadratic operators
\begin{equation}\label{qwe4}
\mathfrak{L}_ju=-\sqrt{\rho}\mathfrak{P}_j\big((\sqrt{\rho})^{-1}u\big)-\frac{1}{2}\textrm{Tr}(\mathfrak{B}_j)u,
\end{equation}
for all $0 \leq j \leq k_0$, which are computed explicitly in the following lemma:

\medskip

\begin{lemma}\label{fghj1}
For all $0 \leq j \leq k_0$, the Weyl symbol of the quadratic operator
$$\mathfrak{L}_ju=-\sqrt{\rho}\mathfrak{P}_j\big((\sqrt{\rho})^{-1}u\big)-\frac{1}{2}\emph{\textrm{Tr}}(\mathfrak{B}_j)u=\tilde{r}_j^w(x,D_x)u,$$
is given by
$$\tilde{r}_j(X)=\emph{\textrm{Re }}q((\emph{\textrm{Im }}F)^jX)
=\frac{1}{2^{4j+1}}|Q^{\frac{1}{2}}(Q_{\infty}^{-1}Q+2B^T)^j\xi|^2+\frac{1}{2^{4j+3}}|Q^{\frac{1}{2}}Q_{\infty}^{-1}(QQ_{\infty}^{-1}+2B)^jx|^2,$$
with $X=(x,\xi) \in \rr^{2n}$, where $F$ is the Hamilton map defined in (\ref{F}).
\end{lemma}

\medskip

\begin{proof}
We deduce from (\ref{pav2}), (\ref{dfg5}), (\ref{qwe2}) and (\ref{qwe4}) that for all $0 \leq j \leq k_0$,
\begin{align*}
\mathfrak{L}_ju =& \ -e^{-\frac{1}{4}\langle Q_{\infty}^{-1}x,x\rangle}\Big(\frac{1}{2}\sum_{l,k=1}^{n}\mathfrak{q}(j)_{l,k}\partial_{x_l,x_k}^2+\sum_{l,k=1}^n\mathfrak{b}(j)_{l,k}x_k\partial_{x_l}\Big)(e^{\frac{1}{4}\langle Q_{\infty}^{-1}x,x\rangle}u)-\frac{1}{2}\textrm{Tr}(\mathfrak{B}_j)u\\
 = & \  -\frac{1}{2}\sum_{l,k=1}^{n}\mathfrak{q}(j)_{l,k}\Big(\partial_{x_l}+\frac{1}{2}(Q_{\infty}^{-1}x)_l\Big)\Big(\partial_{x_k}+\frac{1}{2}(Q_{\infty}^{-1}x)_k\Big)u\\
 & \  -\sum_{l,k=1}^n\mathfrak{b}(j)_{l,k}x_k\Big(\partial_{x_l}+\frac{1}{2}(Q_{\infty}^{-1}x)_l\Big)u-\frac{1}{2}\textrm{Tr}(\mathfrak{B}_j)u,
\end{align*}
with $\mathfrak{Q}_j=\big(\mathfrak{q}(j)_{l,k}\big)_{1 \leq l,k \leq n}$ and $\mathfrak{B}_j=\big(\mathfrak{b}(j)_{l,k}\big)_{1 \leq l,k \leq n}$.
Furthermore, since
\begin{equation}\label{yp1}
\sum_{l,k=1}^{n}\mathfrak{q}(j)_{l,k}\partial_{x_l}\big((Q_{\infty}^{-1}x)_k\big)=\textrm{Tr}(\mathfrak{Q}_jQ_{\infty}^{-1}),
\end{equation}
it follows that for all $0 \leq j \leq k_0$,
\begin{multline*}
\mathfrak{L}_j =-\frac{1}{2}\langle \mathfrak{Q}_j\nabla_x,\nabla_x\rangle-\frac{1}{8}\langle \mathfrak{Q}_jQ_{\infty}^{-1}x,Q_{\infty}^{-1}x\rangle-\frac{1}{2}\langle \mathfrak{B}_jx,Q_{\infty}^{-1}x\rangle
\\ -\langle \mathfrak{B}_jx,\nabla_x\rangle-\frac{1}{2}\langle  \mathfrak{Q}_jQ_{\infty}^{-1}x,\nabla_x\rangle-\frac{1}{2}\textrm{Tr}(\mathfrak{B}_j)-\frac{1}{4}\textrm{Tr}(\mathfrak{Q}_jQ_{\infty}^{-1}).
\end{multline*}
By recalling the notation (\ref{not}), we obtain that for all $0 \leq j \leq k_0$,
\begin{multline}\label{pav12}
\mathfrak{L}_j = \frac{1}{2}\langle \mathfrak{Q}_jD_x,D_x\rangle-\frac{1}{8}\langle \mathfrak{Q}_jQ_{\infty}^{-1}x,Q_{\infty}^{-1}x\rangle-\frac{1}{2}\langle \mathfrak{B}_jx,Q_{\infty}^{-1}x\rangle \\  -i\langle \mathfrak{B}_jx,D_x\rangle-\frac{i}{2}\langle \mathfrak{Q}_jQ_{\infty}^{-1}x,D_x\rangle-\frac{1}{2}\textrm{Tr}(\mathfrak{B}_j)-\frac{1}{4}\textrm{Tr}(\mathfrak{Q}_jQ_{\infty}^{-1}),
\end{multline}
with $D_x=i^{-1}\nabla_x$.
We deduce from (\ref{qwe1}) that for all $0 \leq j \leq k_0$,
\begin{multline}\label{pav13}
\mathfrak{L}_j = \frac{1}{2}\langle \mathfrak{Q}_jD_x,D_x\rangle+\frac{1}{8}\langle \mathfrak{Q}_jQ_{\infty}^{-1}x,Q_{\infty}^{-1}x\rangle\\
= \frac{1}{2^{4j+1}}|Q^{\frac{1}{2}}(Q_{\infty}^{-1}Q+2B^T)^jD_x|^2+ \frac{1}{2^{4j+3}}|Q^{\frac{1}{2}}(Q_{\infty}^{-1}Q+2B^T)^jQ_{\infty}^{-1}x|^2.
\end{multline}
We notice from the steady state equation (\ref{pav10}) that 
\begin{equation}\label{asd1}
(Q_{\infty}^{-1}Q+2B^T)Q_{\infty}^{-1}=-Q_{\infty}^{-1}(QQ_{\infty}^{-1}+2B).
\end{equation}
We deduce from (\ref{asd1}) that for all $0 \leq k \leq k_0$,
\begin{equation}\label{asd2}
(Q_{\infty}^{-1}Q+2B^T)^jQ_{\infty}^{-1}=(-1)^jQ_{\infty}^{-1}(QQ_{\infty}^{-1}+2B)^j.
\end{equation}
It follows from (\ref{pav6}), (\ref{F}), (\ref{pav13}) and (\ref{asd1}) that the Weyl symbol of the quadratic operator $\mathfrak{L}_j=\tilde{r}_j^w(x,D_x)$ is given by 
$$\tilde{r}_j(X)=\textrm{Re }q((\textrm{Im }F)^jX)=\frac{1}{2^{4j+1}}|Q^{\frac{1}{2}}(Q_{\infty}^{-1}Q+2B^T)^j\xi|^2+ \frac{1}{2^{4j+3}}|Q^{\frac{1}{2}}Q_{\infty}^{-1}(QQ_{\infty}^{-1}+2B)^jx|^2,$$
with $X=(x,\xi) \in \rr^{2n}$. This ends the proof of Lemma~\ref{fghj1}.
\end{proof}

Next, we consider the Ornstein-Uhlenbeck operators 
\begin{equation}\label{qwe5}
\mathscr{P}_k=\frac{1}{2}\textrm{Tr}(\mathscr{Q}_k\nabla_x^2)+\langle \mathscr{B}_kx,\nabla_x\rangle, \quad 0 \leq k \leq k_0,
\end{equation} 
acting on the space $L_{\mu}^2$, with
\begin{equation}\label{qwe7}
\mathscr{Q}_k=\sum_{j=0}^k\mathfrak{Q}_j, \qquad \mathscr{B}_k=\sum_{j=0}^k\mathfrak{B}_j,
\end{equation} 
and equipped with the domains
\begin{equation}\label{qwe6}
D(\mathscr{P}_k)=\big\{u \in L_{\mu}^2 : \mathscr{P}_ku \in L_{\mu}^2\big\}.
\end{equation}
It follows from (\ref{rg0}), (\ref{rg2}), (\ref{qwe2}), (\ref{qwe5}), (\ref{qwe7}) and Lemma~\ref{fghj1} that the positive operator defined in (\ref{rg1}) is given by
\begin{equation}\label{qwe8}
\Lambda_k^2u=\sqrt{\rho}\Big(1-\mathscr{P}_k-\frac{1}{2}\textrm{Tr}(\mathscr{B}_k)\Big)\big((\sqrt{\rho})^{-1}u\big),
\end{equation}
for any $0 \leq k \leq k_0$. As in (\ref{tio1}), we can consider the fractional powers of these positive operators 
\begin{equation}\label{qwe9.33}
\Big(1-\mathscr{P}_0-\frac{1}{2}\textrm{Tr}(\mathscr{B}_0)\Big)^{\frac{1}{2}}u=\sqrt{\rho}^{-1}\Lambda_0(\sqrt{\rho}u)
\end{equation} 
and
\begin{equation}\label{qwe9}
\Big(1-\mathscr{P}_k-\frac{1}{2}\textrm{Tr}(\mathscr{B}_k)\Big)^{\frac{1}{2k+1}}u=\sqrt{\rho}^{-1}\Lambda_k^{\frac{2}{2k+1}}(\sqrt{\rho}u), \quad 1 \leq k \leq k_0.
\end{equation} 
The application of Theorem~\ref{th-1} provides the following result: 

\medskip

\begin{corollary}\label{th-12}
Let 
$$P=\frac{1}{2}\emph{\textrm{Tr}}(Q\nabla_x^2)+\langle Bx,\nabla_x\rangle, \quad x \in \rr^n,$$ 
be a hypoelliptic Ornstein-Uhlenbeck operator, which admits the invariant measure $d\mu(x)=\rho(x) dx$.
Then, there exists a positive constant $C>0$ such that for all $v \in D(P)$,
$$\Big\|\Big(1-\mathscr{P}_0-\frac{1}{2}\emph{\textrm{Tr}}(\mathscr{B}_0)\Big)^{\frac{1}{2}}v\Big\|_{L_{\mu}^2}+\sum_{k=1}^{k_0}\Big\|\Big(1-\mathscr{P}_k-\frac{1}{2}\emph{\textrm{Tr}}(\mathscr{B}_k)\Big)^{\frac{1}{2k+1}}v\Big\|_{L_{\mu}^2} \leq C(\|Pv\|_{L_{\mu}^2}+\|v\|_{L^2_{\mu}}),$$
with $\|\cdot\|_{L^2_{\mu}}$ the $L^2(\rr^n,d\mu)$-norm, 
where $0 \leq k_0 \leq n-1$ is the smallest integer satisfying (\ref{kal2}).
\end{corollary}

\medskip

\begin{proof}
Recalling that the quadratic operator $\mathscr{L}=q^w(x,D_x)$ defined in (\ref{pav5}) has a Weyl symbol with a zero singular space $S=\{0\}$ and a non-negative real part $\textrm{Re }q \geq 0$, we deduce from Theorem~\ref{th-1} that there exists a positive constant $C>0$ such that for all $u \in D(q^w)$,
\begin{equation}\label{sdf5}
\|\Lambda_0u\|_{L^2(\rr^n)}+\sum_{k=1}^{k_0}\|\Lambda_k^{\frac{2}{2k+1}}u\|_{L^2(\rr^n)} \leq C(\|q^w(x,D_x)u\|_{L^2(\rr^n)}+\|u\|_{L^2(\rr^n)}),
\end{equation}
where $0 \leq k_0 \leq 2n-1$ is the smallest integer satisfying (\ref{e2}), and where the operators $\Lambda_k^2$ are defined in (\ref{rg1}). It follows from (\ref{xc5}), (\ref{pav4.5}), (\ref{1w1}), (\ref{qwe9.33}), (\ref{qwe9}) and (\ref{sdf5}) that
$$\forall v \in D(P), \quad \|\sqrt{\rho}^{-1}\Lambda_0(\sqrt{\rho}v)\|_{L_{\mu}^2}+\sum_{k=1}^{k_0}\|\sqrt{\rho}^{-1}\Lambda_k^{\frac{2}{2k+1}}(\sqrt{\rho}v)\|_{L_{\mu}^2} \leq C(\|Pv\|_{L_{\mu}^2}+\|v\|_{L^2_{\mu}}),$$
that is, for all $v \in D(P)$, 
$$\Big\|\Big(1-\mathscr{P}_0-\frac{1}{2}\textrm{Tr}(\mathscr{B}_0)\Big)^{\frac{1}{2}}v\Big\|_{L_{\mu}^2}+\sum_{k=1}^{k_0}\Big\|\Big(1-\mathscr{P}_k-\frac{1}{2}\textrm{Tr}(\mathscr{B}_k)\Big)^{\frac{1}{2k+1}}v\Big\|_{L_{\mu}^2} \leq C(\|Pv\|_{L_{\mu}^2}+\|v\|_{L^2_{\mu}}).$$
This ends the proof of Corollary~\ref{th-12}.
\end{proof}

The result of Corollary~\ref{th-12} is reminiscent from the continuous inclusion (\ref{ui5}) proven by Farkas and Lunardi in~\cite{lunardi1} (Theorem~8). Our result is however not as sharp as (\ref{ui5}), since the power associated to the elliptic directions should be equal to 1 instead of $1/2$.

\subsection{Applications to degenerate hypoelliptic Fokker-Planck operators}

We consider the Fokker-Planck operator
\begin{equation}\label{jen077}
\mathscr{P}=\frac{1}{2}\textrm{Tr}(Q\nabla_x^2)-\langle Bx,\nabla_x\rangle-\textrm{Tr}(B), \quad x \in \rr^n,
\end{equation}
where $Q=(q_{i,j})_{1 \leq i,j \leq n}$ and $B=(b_{i,j})_{1 \leq i,j \leq n}$ are real $n \times n$-matrices, with $Q$ symmetric positive semidefinite. We assume that the Kalman rank condition 
and the localization of the spectrum of~$B$,
\begin{equation}\label{kal177}
\textrm{Rank}[Q^{\frac{1}{2}},BQ^{\frac{1}{2}},\dots, B^{n-1}Q^{\frac{1}{2}}]=n, \qquad \sigma(B) \subset \mathbb{C}_-,
\end{equation} 
hold. As before, we consider 
\begin{equation}\label{pav277}
\rho(x)=\frac{1}{(2\pi)^{\frac{n}{2}}\sqrt{\det Q_{\infty}}}e^{-\frac{1}{2}\langle Q_{\infty}^{-1}x,x\rangle},
\end{equation}
with 
\begin{equation}\label{pav377}
Q_{\infty}=\int_0^{+\infty}e^{sB}Qe^{sB^T}ds.
\end{equation}
We aim at studying the operator $\mathscr{P}$ acting on the $L^2_{1/\rho}=L^2(\rr^n,\rho(x)^{-1}dx)$ space.
In the recent work~\cite{anton}, Arnold and Erb study the degenerate parabolic Fokker-Planck equation
$$\partial_tf=\frac{1}{2}\textrm{Tr}(Q\nabla_x^2)f-\langle Bx,\nabla_x\rangle f-\textrm{Tr}(B)f.$$
By employing a new entropy method based on a modified, non-degenerate entropy dissipation like functional, whose construction has inspired the proof of Theorem~\ref{th},
they establish some results of exponential convergence of the solution to the equilibrium and the exponential decay in relative entropy (logarithmic till quadratic) with a sharp rate.
More specifically, the result of~\cite{anton} (Theorem~4.8) is closely related to the following result obtained as a corollary of Theorem~\ref{th}: 

We begin by associating to the operator $\mathscr{P}$ acting on $L^2_{1/\rho}$, the quadratic operator $\mathfrak{L}$ acting on $L^2(\rr^n,dx)$ given by
\begin{equation}\label{pav4.577}
\mathfrak{L}u=-\sqrt{\rho}^{-1}\mathscr{P}\big(\sqrt{\rho}u\big)-\frac{1}{2}\textrm{Tr}(B)u.
\end{equation}
An explicit computation~\cite{OPPS2}, see formula (2.54), gives that
\begin{equation}\label{pav1277}
\mathfrak{L}=\frac{1}{2}|Q^{\frac{1}{2}}D_x|^2+\frac{1}{8}|Q^{\frac{1}{2}}Q_{\infty}^{-1}x|^2+i\Big\langle \Big(\frac{1}{2} QQ_{\infty}^{-1}+B\Big)x,D_x\Big\rangle.
\end{equation}
We observed in the proof of~\cite{OPPS2} (Proposition~2.8) that this operator $\mathfrak{L}$ coincides with the $L^2(\rr^n)$-adjoint of the operator (\ref{pav5}), and that its Weyl symbol is $\overline{q}$ the complex conjugate of the quadratic form (\ref{pav6}). The Hamilton map of $\overline{q}$ is $\overline{F}$ the complex conjugate of the Hamilton map (\ref{F}). 
By using that 
$$\textrm{Ker}\big(\textrm{Re }\overline{F}(\textrm{Im }\overline{F})^j\big)=\textrm{Ker}\big(\textrm{Re }F(\textrm{Im }F)^j\big),$$
it follows from (\ref{sd1}) that the quadratic operator $\mathfrak{L}=\overline{q}^w(x,D_x)$ has a Weyl symbol with a zero singular space and a non-negative real part. As before in (\ref{pav30}) and (\ref{pav31}), the smallest integer $0 \leq k_0 \leq 2n-1$ satisfying 
\begin{equation}\label{pav30.7}
\Big(\bigcap_{j=0}^{k_0}\textrm{Ker}
\big[\textrm{Re }F(\textrm{Im }F)^j \big]\Big)\cap \rr^{2n}=\{0\},
\end{equation}
corresponds exactly to the smallest integer $0 \leq k_0 \leq n-1$ satisfying
\begin{equation}\label{pav31.7}
\textrm{Rank}[Q^{\frac{1}{2}},BQ^{\frac{1}{2}},\dots, B^{k_0}Q^{\frac{1}{2}}]=n.
\end{equation}
Following (\ref{we20}), the vector subspaces (\ref{we1}) are anew given in this case by 
\begin{equation}\label{we20.7}
V_k=\Big[\Big(\Big(\bigcap_{j=0}^{k}\textrm{Ker}(QQ_{\infty}^{-1}B^j)\Big)\cap \rr^{n}\Big) \times \Big(\Big(\bigcap_{j=0}^{k}\textrm{Ker}\big(Q(B^T)^j\big)\Big)\cap \rr^{n}\Big)\Big]^{\perp}.
\end{equation}
We observe from (\ref{we21}), (\ref{we22}) and Lemma~\ref{li1} that 
\begin{equation}\label{we21.7}
(0,\xi_0) \in V_k \Leftrightarrow \xi_0 \in \mathcal{V}_k=\big(\textrm{Ran}(Q^{\frac{1}{2}})+...+\textrm{Ran}(B^kQ^{\frac{1}{2}})\big) \cap \rr^n,
\end{equation}
\begin{equation}\label{we22.7}
(x_0,0) \in V_k \Leftrightarrow  x_0 \in \tilde{\mathcal{V}}_k=Q_{\infty}^{-1}\mathcal{V}_k.
\end{equation}
We recall from (\ref{byebye}) and (\ref{byebye2}) that
$$\mathcal{V}_0 \subsetneq \mathcal{V}_1 \subsetneq ... \subsetneq \mathcal{V}_{k_0}=\rr^n, \qquad \tilde{\mathcal{V}}_0 \subsetneq \tilde{\mathcal{V}}_1 \subsetneq ... \subsetneq \tilde{\mathcal{V}}_{k_0}=\rr^n.$$

As in Section~\ref{orn}, we can consider for any $\xi_0 \in \rr^n$ the smallest integer $0 \leq k \leq k_0$ 
such that $\xi_0 \in \mathcal{V}_k$. This integer is denoted $k_{\xi_0}$ and is called the frequency index of the point $\xi_0 \in \rr^n$ with respect to the Fokker-Planck operator $\mathscr{P}$. Similarly, we can consider for any $x_0 \in \rr^n$ the smallest integer $0 \leq k \leq k_0$ 
such that $x_0 \in \tilde{\mathcal{V}}_k$. This integer is denoted $\tilde{k}_{x_0}$ and is called the space index of the point $x_0 \in \rr^n$ with respect to the Fokker-Planck operator $\mathscr{P}$.

The result of Theorem~\ref{th} allows one to obtain the following estimates for degenerate hypoelliptic Fokker-Planck operators:

\medskip

\begin{corollary}\label{coro3}
Let 
$$\mathscr{P}=\frac{1}{2}\emph{\textrm{Tr}}(Q\nabla_x^2)-\langle Bx,\nabla_x\rangle-\emph{\textrm{Tr}}(B), \quad x \in \rr^n,$$ 
be a Fokker-Planck operator satisfying the assumption (\ref{kal177}).  
Then, the global solution $v(t)=e^{t\mathscr{P}}v_0$ to the Cauchy problem
\begin{equation}\label{p2.7}
\left\lbrace\begin{array}{c}
\partial_tv=\mathscr{P}v,\\
v|_{t=0}=v_0 \in L^2_{1/\rho},
\end{array}\right.
\end{equation}
with $L^2_{1/\rho}=L^2(\rr^n,\rho(x)^{-1}dx)$, satisfies: $\exists C>0$, $\forall x_0 \in \rr^{n}$, $\forall \xi_0 \in \rr^{n}$, $\forall v_0 \in L^2_{1/\rho}$,    
$$\forall 0<t \leq 1, \quad \|\langle \xi_0,D_x \rangle e^{t\mathscr{P}}v_0\|_{L^2_{1/\rho}}\leq C|\xi_0| t^{-(2k_{\xi_0}+1)/2}\|v_0\|_{L^2_{1/\rho}},$$
$$\forall 0<t \leq 1, \quad \|\langle x_0,x \rangle e^{t\mathscr{P}}v_0\|_{L^2_{1/\rho}}\leq C|x_0| t^{-(2\tilde{k}_{x_0}+1)/2}\|v_0\|_{L^2_{1/\rho}},$$
$$\forall t \geq 1, \quad \|\langle \xi_0,D_x \rangle e^{t\mathscr{P}}v_0\|_{L^2_{1/\rho}}\leq C|\xi_0|\|v_0\|_{L^2_{1/\rho}},$$
$$\forall t \geq 1, \quad \|\langle x_0,x \rangle e^{t\mathscr{P}}v_0\|_{L^2_{1/\rho}}\leq C|x_0|\|v_0\|_{L^2_{1/\rho}},$$
with $\|\cdot\|_{L^2_{1/\rho}}$ the $L^2(\rr^n,\rho(x)^{-1}dx)$-norm, where $0 \leq k_{\xi_0} \leq k_0$ denotes the frequency index of $\xi_0 \in \rr^n$ with respect to $\mathscr{P}$, and where $0 \leq \tilde{k}_{x_0} \leq k_0$ denotes the space index of $x_0 \in \rr^n$ with respect to $\mathscr{P}$.
\end{corollary}

\medskip

\begin{proof}
By using that the quadratic operator
$$\mathfrak{L}=\frac{1}{2}|Q^{\frac{1}{2}}D_x|^2+\frac{1}{8}|Q^{\frac{1}{2}}Q_{\infty}^{-1}x|^2+i\Big\langle \Big(\frac{1}{2} QQ_{\infty}^{-1}+B\Big)x,D_x\Big\rangle,$$
has a Weyl symbol with a zero singular space $S=\{0\}$ and a non-negative real part, we can deduce from~\cite{HPS} (Theorem~1.2.1) that the evolution equation associated to the accretive operator $\mathfrak{L}$, 
$$\left\lbrace\begin{array}{c}
\partial_tu+\mathfrak{L}u=0,\\
u|_{t=0}=u_0 \in L^2(\rr^n,dx),
\end{array}\right.$$
is smoothing in the Schwartz space $\mathscr{S}(\rr^n)$ for any positive time $t>0$, that is,
\begin{equation}\label{p1.7}
\forall t>0, \quad u(t)=e^{-t\mathfrak{L}}u_0 \in \mathscr{S}(\rr^n),
\end{equation}
where $(e^{-t\mathfrak{L}})_{t \geq 0}$ denotes the contraction semigroup generated by $\mathfrak{L}$. It follows from (\ref{pav4.577}) that the solution to the evolution equation (\ref{p2.7}) is given by
\begin{equation}\label{bn1.7}
v(t)=e^{t\mathscr{P}}v_0=\sqrt{\rho}e^{-t(\mathfrak{L}+\frac{1}{2}\textrm{Tr}(B))}(\sqrt{\rho}^{-1}v_0)=e^{-\frac{t}{2}\textrm{Tr}(B)}\sqrt{\rho}e^{-t\mathfrak{L}}(\sqrt{\rho}^{-1}v_0),
\end{equation}
for all $t \geq 0$.
By using from (\ref{pav277}) that
$$\sqrt{\rho}^{-1}\partial_{x_i}(\sqrt{\rho}u)=e^{\frac{1}{4}\langle Q_{\infty}^{-1}x,x\rangle}\partial_{x_i}(e^{-\frac{1}{4}\langle Q_{\infty}^{-1}x,x\rangle}u)=\Big(\partial_{x_i}-\frac{1}{2}(Q_{\infty}^{-1}x)_i\Big)u,$$
where $(Q_{\infty}^{-1}x)_i$ denotes the $i^{\textrm{th}}$ coordinate,
since the matrix $Q_{\infty}^{-1}$ is symmetric, we notice that for all $t > 0$,
$$\langle \xi_0,D_x \rangle e^{t\mathscr{P}}v_0=e^{-\frac{t}{2}\textrm{Tr}(B)}\sqrt{\rho}
\Big(\langle \xi_0,D_x \rangle+\frac{i}{2}\langle \xi_0,Q_{\infty}^{-1}x \rangle \Big)e^{-t\mathfrak{L}}(\sqrt{\rho}^{-1}v_0).$$
By using that the mappings 
\begin{equation}\label{alw1}
\begin{array}{cc}
\mathfrak{T} :  L^2  & \rightarrow  L^2_{1/\rho}\\
\ u   & \mapsto  \sqrt{\rho}u
\end{array}, \quad \begin{array}{cc}
\mathfrak{T}^{-1} :  L^2_{1/\rho} & \rightarrow L^2\\
\quad \ u & \mapsto \sqrt{\rho}^{-1}u
\end{array},
\end{equation}
are isometric, we deduce from (\ref{conti5}) that 
\begin{multline}\label{dfg21}
\|\langle \xi_0,D_x \rangle e^{t\mathscr{P}}v_0\|_{L^2_{1/\rho}}\\
\leq e^{-\frac{t}{2}\textrm{Tr}(B)}\big(\|\langle \xi_0,D_x \rangle e^{-t\mathfrak{L}}(\sqrt{\rho}^{-1}v_0)\|_{L^2}
+\|\langle Q_{\infty}^{-1}\xi_0,x \rangle e^{-t\mathfrak{L}}(\sqrt{\rho}^{-1}v_0)\|_{L^2}\big).
\end{multline}
According to (\ref{pav4.577}), we observe from (\ref{conti5}) and~\cite{OPPS2} (Proposition~2.8), see also the formula (2.55) in~\cite{OPPS2}, that 
\begin{equation}\label{conti6.1}
\omega_0=\sum_{\substack{\lambda \in \sigma(\overline{F}) \\
-i \lambda \in \cc_+}}r_{\lambda}\textrm{Re}(-i\lambda)=-\frac{1}{2}\textrm{Tr}(B)>0,
\end{equation}
where $r_{\lambda}$ stands for the dimension of the space of generalized eigenvectors of $\overline{F}$ the complex conjugate of the Hamilton map (\ref{F}) in $\cc^{2n}$ associated to the eigenvalue $\lambda$.
According to (\ref{we21.7}) and (\ref{we22.7}), the property $\xi_0 \in \mathcal{V}_{k_{\xi_0}}$ implies that 
\begin{equation}\label{yy1.7}
(0,\xi_0) \in V_{k_{\xi_0}}, \quad (Q_{\infty}^{-1}\xi_0,0) \in V_{k_{\xi_0}},
\end{equation}
since $Q_{\infty}^{-1}\xi_0 \in \tilde{\mathcal{V}}_{k_{\xi_0}}$,
where $0 \leq k_{\xi_0} \leq k_0$ denotes the frequency index of $\xi_0 \in \rr^n$ with respect to $\mathscr{P}$. The indices with respect to the singular space of the quadratic operator $\mathfrak{L}$ of both directions $(0,\xi_0)$ and $(Q_{\infty}^{-1}\xi_0,0)$, are therefore lower than $k_{\xi_0}$.
We deduce from (\ref{alw1}), (\ref{dfg21}), (\ref{conti6.1}), (\ref{yy1.7}) and Theorem~\ref{th} that there exists a positive constant $C>0$ such that for all $\xi_0 \in \rr^{n}$, $v_0 \in L^2_{1/\rho}$,  
$$\forall 0<t \leq 1, \quad \|\langle \xi_0,D_x \rangle e^{t\mathscr{P}}v_0\|_{L^2_{1/\rho}}\leq C|\xi_0| t^{-(2k_{\xi_0}+1)/2}\|v_0\|_{L^2_{1/\rho}},$$
$$\forall t \geq 1, \quad \|\langle \xi_0,D_x \rangle e^{t\mathscr{P}}v_0\|_{L^2_{1/\rho}}\leq C|\xi_0|\|v_0\|_{L^2_{1/\rho}}.$$
On the other hand, we deduce from (\ref{bn1.7}) that for all $t > 0$, $v_0 \in L^2_{1/\rho}$,
\begin{equation}\label{3bn1.55}
\langle x_0,x \rangle e^{t\mathscr{P}}v_0=e^{-\frac{t}{2}\textrm{Tr}(B)}\sqrt{\rho}\langle x_0,x \rangle e^{-t\mathfrak{L}}(\sqrt{\rho}^{-1}v_0).
\end{equation}
We deduce from (\ref{alw1}) and (\ref{3bn1.55}) that for all $t > 0$, $v_0 \in L^2_{1/\rho}$,
\begin{equation}\label{3bn2.55}
\|\langle x_0,x \rangle e^{t\mathscr{P}}v_0\|_{L^2_{1/\rho}}=e^{-\frac{t}{2}\textrm{Tr}(B)}\|\langle x_0,x \rangle e^{-t\mathfrak{L}}(\sqrt{\rho}^{-1}v_0)\|_{L^2}.
\end{equation}
We observe from (\ref{we22.7}) that the property $x_0 \in \tilde{\mathcal{V}}_{\tilde{k}_{x_0}}$ implies that $(x_0,0) \in V_{\tilde{k}_{x_0}}$,
where $0 \leq \tilde{k}_{x_0} \leq k_0$ denotes the space index of $x_0 \in \rr^n$ with respect to $\mathscr{P}$.
The index with respect to the singular space of the quadratic operator $\mathfrak{L}$ of the direction $(x_0,0)$ is therefore lower than $\tilde{k}_{x_0}$.
We deduce from (\ref{alw1}), (\ref{conti6.1}), (\ref{3bn2.55}) and Theorem~\ref{th} that there exists a positive constant $C>0$ such that for all $x_0 \in \rr^{n}$, $v_0 \in L^2_{1/\rho}$,  
$$\forall 0<t \leq 1, \quad \|\langle x_0,x \rangle e^{t\mathscr{P}}v_0\|_{L^2_{1/\rho}}\leq C|x_0| t^{-(2\tilde{k}_{x_0}+1)/2}\|v_0\|_{L^2_{1/\rho}},$$
$$\forall t \geq 1, \quad \|\langle x_0,x \rangle e^{t\mathscr{P}}v_0\|_{L^2_{1/\rho}}\leq C|x_0|\|v_0\|_{L^2_{1/\rho}}.$$
This ends the proof of Corollary~\ref{coro3}.
\end{proof}

By mimicking the arguments given in the proof of Corollary~\ref{coro111},  we can deduce from Corollary~\ref{thj} the following result:  

\medskip

\begin{corollary}\label{coro111.55}
Under the assumptions of Corollary~\ref{coro3}, there exists a positive constant $C>1$ such that for all $m \geq 1$, $X_1=(x_1,\xi_1)  \in \rr^{2n}$, ..., $X_m=(x_m,\xi_m)  \in \rr^{2n}$, $v_0 \in L^2(\rr^n,\rho(x)^{-1}dx)$,
\begin{multline*}
\forall 0<t \leq 1, \quad \|(\langle x_1, x \rangle+\langle \xi_1,D_x \rangle)\ ...\ (\langle x_m, x \rangle+\langle \xi_m,D_x \rangle) e^{t\mathscr{P}}v_0\|_{L^2_{1/\rho}}\\
\leq C^{m}(m!)^{\frac{2k_0+1}{2}}\Big(\prod_{j=1}^m|X_j|\Big)t^{-\frac{(2k_0+1)m}{2}}\|v_0\|_{L^2_{1/\rho}},
\end{multline*}
\begin{multline*}
\forall t \geq 1, \quad \|(\langle x_1, x \rangle+\langle \xi_1,D_x \rangle)\ ...\ (\langle x_m, x \rangle+\langle \xi_m,D_x \rangle)e^{t\mathscr{P}}v_0\|_{L^2_{1/\rho}}\\
\leq C^m(m!)^{\frac{2k_0+1}{2}}\Big(\prod_{j=1}^m|X_j|\Big)\|v_0\|_{L^2_{1/\rho}},
\end{multline*}
with $\|\cdot\|_{L^2_{1/\rho}}$ the $L^2(\rr^n,\rho(x)^{-1}dx)$-norm, where $0 \leq k_0 \leq n-1$ is the smallest integer satisfying (\ref{pav31.7}),
and where $|\cdot|$ denotes the Euclidean norm on $\rr^{2n}$.
\end{corollary}

Finally, we refrain from providing the statements and the proofs of the applications of Corollary~\ref{th0} and Theorem~\ref{th-1} to the study of degenerate hypoelliptic Fokker-Planck operators. The reader interested by this matter will easily obtain these results by mimicking the proofs given in Section~\ref{orn} for degenerate hypoelliptic Ornstein-Uhlenbeck operators.

\section{Appendix}\label{appendix}

We refer the reader to the works~\cite{gelfand,rodino1,rodino,toft} and the references herein for extensive expositions of the Gelfand-Shilov regularity theory.
The Gelfand-Shilov spaces $S_{\nu}^{\mu}(\rr^n)$, with $\mu,\nu>0$, $\mu+\nu\geq 1$, are defined as the spaces of smooth functions $f \in C^{\infty}(\rr^n)$ satisfying the estimates
$$\exists A,C>0, \quad |\partial_x^{\alpha}f(x)| \leq C A^{|\alpha|}(\alpha !)^{\mu}e^{-\frac{1}{A}|x|^{1/\nu}}, \quad x \in \rr^n, \ \alpha \in \mathbb{N}^n,$$
or, equivalently
$$\exists A,C>0, \quad \sup_{x \in \rr^n}|x^{\beta}\partial_x^{\alpha}f(x)| \leq C A^{|\alpha|+|\beta|}(\alpha !)^{\mu}(\beta !)^{\nu}, \quad \alpha, \beta \in \mathbb{N}^n.$$
These Gelfand-Shilov spaces  $S_{\nu}^{\mu}(\rr^n)$ may also be characterized as the spaces of Schwartz functions $f \in \mathscr{S}(\rr^n)$ satisfying the estimates
$$\exists C>0, \eps>0, \quad |f(x)| \leq C e^{-\eps|x|^{1/\nu}}, \quad x \in \rr^n, \qquad |\widehat{f}(\xi)| \leq C e^{-\eps|\xi|^{1/\mu}}, \quad \xi \in \rr^n.$$
In particular, we notice that Hermite functions belong to the symmetric Gelfand-Shilov space  $S_{1/2}^{1/2}(\rr^n)$. More generally, the symmetric Gelfand-Shilov spaces $S_{\mu}^{\mu}(\rr^n)$, with $\mu \geq 1/2$, can be nicely characterized through the decomposition into the Hermite basis $(\Psi_{\alpha})_{\alpha \in \mathbb{N}^n}$, see e.g. \cite{toft} (Proposition~1.2),
\begin{multline*}
f \in S_{\mu}^{\mu}(\rr^n) \Leftrightarrow f \in L^2(\rr^n), \ \exists t_0>0, \ \big\|\big((f,\Psi_{\alpha})_{L^2}\exp({t_0|\alpha|^{\frac{1}{2\mu}})}\big)_{\alpha \in \mathbb{N}^n}\big\|_{l^2(\mathbb{N}^n)}<+\infty\\
\Leftrightarrow f \in L^2(\rr^n), \ \exists t_0>0, \ \|e^{t_0\mathcal{H}^{1/2\mu}}f\|_{L^2}<+\infty,
\end{multline*}
where $\mathcal{H}=-\Delta_x+x^2$ stands for the harmonic oscillator.

\bigskip
\noindent
{\bf Acknowledgements.}
The authors are very grateful to the anonymous referee for the very interesting suggestion to try to extend the results obtained for quadratic operators with zero singular spaces to quadratic operators whose singular spaces enjoy more generally a symplectic structure. 
The first author acknowledges the kind hospitality of the Centre Henri Lebesgue and the University of Rennes~1 in June 2015. The research of the second and third authors was supported by the ANR NOSEVOL (Project: ANR 2011 BS0101901).

\end{document}